\documentclass[12pt]{amsart}
\usepackage{amssymb,amsmath,tabularx,mathrsfs,mathbbol,yfonts,upgreek}
\usepackage{amsthm,verbatim,comment}
\usepackage[bookmarks=true]{hyperref}
\usepackage{pstricks,pst-node,pst-plot}
\usepackage{geometry}
\usepackage{stmaryrd}
\usepackage{paralist}
\usepackage[all]{xy}
\usepackage{array}
\usepackage{url}
\usepackage{longtable}
\numberwithin{equation}{section}

\DeclareSymbolFontAlphabet{\mathbb}{AMSb}
\DeclareSymbolFontAlphabet{\mathbbl}{bbold}

\newtheorem{thm}{Theorem}[section]

\newtheorem{lem}[thm]{Lemma}

\newtheorem{cor}[thm]{Corollary}

\theoremstyle{definition}

\newtheorem{nota}[thm]{Notation}

\newtheorem{rem}[thm]{Remark}

\newtheorem*{rem*}{Remarks}

\newtheoremstyle{case}{}{}{}{}{}{:}{ }{}
\theoremstyle{case}

\newcommand{\T}{\mathrm{T}}

\newcommand{\F}{\mathbb{F}}

\title[Terwilliger $\F$-algebras of certain Cayley tables]{Terwilliger $\F$-algebras of certain Cayley tables}
\begin{document}
\author{Yu Jiang}
\address[Y. Jiang]{School of Mathematical Sciences, Anhui University (Qingyuan Campus), No. 111, Jiulong Road, Hefei, 230601, China}
\email[Y. Jiang]{jiangyu@ahu.edu.cn}
%\author{Kay Jin Lim}
%\address[K. J. Lim]{Division of Mathematical Sciences, School of Physical and Mathematical Sciences, Nanyang Technological University, 21 Nanyang Link, Singapore 637371.}
%\email[K. J. Lim]{limkj@ntu.edu.sg}
\begin{abstract}
Let $\F$ be an arbitrary field. In this paper, we continue studying the Terwilliger algebras of association schemes over $\F$ that were called the Terwilliger $\F$-algebras of association schemes in \cite{Jiang2}. We determine the algebraic structures of the Terwilliger $\F$-algebras of the association schemes induced from the Cayley tables of elementary abelian $2$-groups, which provides infinite many examples and counterexamples for the study of the Terwilliger $\F$-algebras of association schemes.
\end{abstract}
\maketitle
\vspace{-1.5em}
\section{Introduction}
Association schemes on nonempty finite sets, briefly called schemes, are important research objects in algebraic combinatorics. They have very strong connections with many different research objects such as groups, codes, graphs, statistical models, and so on. Conversely, many different tools have already been developed to study them.

The subconstituent algebras of commutative schemes, introduced by Terwilliger in \cite{T}, are also known as the Terwilliger algebras of commutative schemes. They are new algebraic tools for the study of schemes and are known to be finite-dimensional semisimple associative $\mathbb{C}$-algebras. While the original Terwilliger algebras were only defined for commutative schemes, Terwilliger algebras can be defined for an arbitrary scheme and an arbitrary field $\F$ (see \cite{Han}). In \cite{Jiang2}, the Terwilliger algebras of schemes over $\F$ are called the Terwilliger $\F$-algebras of schemes. In particular, the Terwilliger algebras of commutative schemes are the Terwilliger $\mathbb{C}$-algebras of these schemes.

Many advancements have been obtained in studying the Terwilliger $\mathbb{C}$-algebras of commutative schemes (for example, see \cite{CD}, \cite{LMP}, \cite{LM}, \cite{LMW}, \cite{T}, \cite{T1}, \cite{T2}). However, the investigation of the Terwilliger $\F$-algebras of schemes is almost completely open (see \cite{Her}). In \cite{CD}, the algebraic structures of the Terwilliger $\mathbb{C}$-algebras of the schemes induced from Latin squares were determined. As the main result of this paper, the algebraic structures of the Terwilliger $\F$-algebras of the schemes induced from the Cayley tables of elementary abelian $2$-groups are similarly given (see Theorem \ref{T;Coretheorem}). This result can be applied to provide infinite many examples and counterexamples for investigating the Terwilliger $\F$-algebras of schemes (see Remarks \ref{R;Coreremark1} and \ref{R;Coreremark2}).

This paper is presented as follows. In Section 2, we collect the basic notation and preliminaries. In Sections 3 and 4, we present some computational results that shall be repeatedly used in the following sections. We divide the whole proof of our main result into four cases and finish the proofs of these cases in Sections 5, 6, 7, and 8, respectively. Some remarks concerning our main result are also in Section 8.
\section{Basic notation and preliminaries}
For a general background on association schemes, one may refer to \cite{B}, \cite{EI}, or \cite{Z2}.

\subsection{General conventions}
Throughout this paper, fix a field $\F$ of characteristic $p$. Let $\mathbb{N}$ denote the set of all natural numbers. Set $\mathbb{N}_0=\mathbb{N}\cup\{0\}$. If $a, b\in \mathbb{N}_0$, put $[a,b]=\{c: a\leq c\leq b\}\subseteq\mathbb{N}_0$. Fix a nonempty finite set $X$. For a nonempty subset $Y$  of an $\F$-linear space, let $\langle Y\rangle_{\F}$ be the $\F$-linear subspace generated by $Y$. The addition, multiplication, and scalar multiplication of matrices displayed in this paper are the usual matrix operations. A scheme always means an association scheme on $X$.

\subsection{Schemes}
Let $S=\{R_0, R_1,\ldots, R_d\}$ denote a partition of the cartesian product $X\times X$. Then $S$ is called a scheme of class $d$ if the following conditions hold:
\begin{enumerate}[(i)]
\item $R_0=\{(x,x): x\in X\}$;
\item For any $i\in [0,d]$, there is $i'\in [0,d]$ such that $R_{i'}=\{(x,y): (y,x)\in R_i\}\in S$;
\item For any $g,h,i\in [0,d]$ and $(x,y), (\widetilde{x}, \widetilde{y})\in R_i$, the following equality holds:
$$ |\{z: (x,z)\in R_g,\ (z,y)\in R_h\}|=|\{z: (\widetilde{x}, z)\in R_g,\ (z, \widetilde{y})\in R_h\}|.$$
\end{enumerate}
Throughout this paper, $S=\{R_0, R_1, \ldots, R_d\}$ is a fixed scheme of class $d$. For any $g, h, i\in [0,d]$ and $(x,y)\in R_i$, (iii) implies that $|\{z: (x,z)\in R_g,\ (z,y)\in R_h\}|$ only depends on $g, h, i$ and is independent of the choices of the elements of $R_i$. Hence there is a constant $p_{gh}^i\in \mathbb{N}_0$ such that $p_{gh}^i=|\{z: (x,z)\in R_g,\ (z,y)\in R_h\}|$. Call $S$ a symmetric scheme if $R_i=R_{i'}$ for any $i\in [0,d]$. If $S$ is a symmetric scheme, notice that $p_{gh}^i=p_{hg}^i$ for any $g, h,i\in [0,d]$. Let $i\in [0,d]$ and $x,y\in X$. Set $k_i=p_{ii'}^0$ and $xR_i=\{z: (x,z)\in R_i\}$. Call $k_i$ the valency of $R_i$ and note that $|xR_i|=|yR_i|=k_i$.

Let $G$ be a finite group and $|G|\geq 3$. Let $G_{(3)}$ be the set of all $3$-tuples whose entries are from $G$. Put $X_G=\{\mathbf{x}: \mathbf{x}=(x_1, x_2, x_3)\in G_{(3)},\ x_1x_2=x_3\}$. So $X_G$ encodes the Cayley tables of $G$. For any $\mathbf{x}=(x_1, x_2, x_3), \mathbf{y}=(y_1, y_2, y_3)\in X_G$, and $\Delta\in [1,3]$, set $\mathbf{x}=_{\Delta}\mathbf{y}$ if $x_{\Delta}=y_{\Delta}$ and $\mathbf{x}\neq \mathbf{y}$. Also define $G_\Delta=\{(\mathbf{x},\mathbf{y}): \mathbf{x},\mathbf{y}\in X_G,\ \mathbf{x}=_{\Delta}\mathbf{y}\}$. By \cite{B}, call $S$ the scheme induced from the Cayley table of $G$ if $X=X_G$, $d=4$, $R_0\!=\!\{(\mathbf{x},\mathbf{x}): \mathbf{x}\in X_G\}$, $R_1=G_1$, $R_2=G_2$, $R_3=G_3$, $R_4=(X_G\times X_G)\setminus(\bigcup_{\ell=0}^3R_\ell)$. \ If $S$ is the scheme induced from the Cayley table of $G$, then $S$ is a symmetric scheme.

\subsection{Algebras}\label{S;Subsection3}
Let $\mathbb{Z}$ be the integer ring. Let $\F_p$ be the prime subfield of $\F$. Given $a\in \mathbb{Z}$, let $\overline{a}$ be the image of $a$ under the unital ring homomorphism from $\mathbb{Z}$ to $\F_p$.

Let $A$ be a finite-dimensional associative unital $\F$-algebra. Let $B$ be a two-sided ideal of $A$. Write $A/B$ for the quotient $\F$-algebra of $A$ with respect to the two-sided ideal $B$. Call $B$ a nilpotent two-sided ideal of $A$ if there exists $m\in \mathbb{N}$ such that the product of any $m$ elements of $B$ equals zero. Let $\mathrm{Rad}A$ be the Jacobson radical of $A$. Recall that $\mathrm{Rad}A$ is the nilpotent two-sided ideal of $A$ that contains all nilpotent two-sided ideals of $A$. Call $A$ a semisimple algebra if $A$ is a direct sum of its minimal two-sided ideals. Recall that $A$ is semisimple if and only if $\mathrm{Rad}A$ is the zero space. If $e\in A$ and $e^2=e$, write $eBe=\{ebe: b\in B\}$. Notice that $eBe$ is an $\F$-subalgebra of $A$ with the identity element $e$. The following lemma shall be occasionally used.
\begin{lem}\label{L;Radical}\cite[Proposition 3.2.4]{DK}
If $e\in A$ and $e^2=e$, then $\mathrm{Rad}eAe=e(\mathrm{Rad}A)e$. In particular, $\mathrm{Rad}eAe\subseteq\mathrm{Rad}A$.
\end{lem}

\subsection{Terwilliger $\F$-algebras of schemes}\label{S;Subsection4}
Let $M_X(\F)$ be the full matrix algebra of $\F$-square matrices whose rows and columns are labeled by the elements in $X$. Let $I, J, O$ be the identity matrix, the all-one matrix, and the all-zero matrix in $M_X(\F)$, respectively. If $M\in M_X(\F)$, the transpose of $M$ is denoted by $M^T$. If $x, y\in X$, let $E_{xy}$ be the $\{0,1\}$-matrix in $M_X(\F)$ whose unique nonzero entry is the $(x,y)$-entry. If $\ell\in \mathbb{N}$ and $m, n\in [1,\ell]$, let $M_\ell(\F)$ denote the full matrix algebra of $(\ell\times \ell)$-matrices whose entries are from $\F$. Use $E_{mn}(\ell)$ to denote the $\{0,1\}$-matrix in $M_{\ell}(\F)$ whose unique nonzero entry is the $(m,n)$-entry.

Assume that $h, i\in [0,d]$ and $x\in X$. The adjacency $\F$-matrix with respect to $R_i$, denoted by $A_i$, is the matrix $\sum_{(y,z)\in R_i}E_{yz}$. The dual $\F$-idempotent with respect to $R_i$ and $x$, denoted by $E_i^*(x)$, is the matrix $\sum_{y\in xR_i}E_{yy}$. If $a, b\in \mathbb{N}_0$, write $\delta_{ab}$ for the Kronecker delta of $a$ and $b$ whose values are from $\F$. By these definitions, note that
\begin{align}\label{Eq;1}
A_i^T=A_{i'}\ \text{and}\ {E_i^*(x)}^T=E_i^*(x)\ (A_i^T=A_i\ \text{if $S$ is a symmetric scheme}),
\end{align}
\begin{align}\label{Eq;2}
A_hA_i=\sum_{j=0}^d \overline{p_{hi}^j}A_j\ \text{and}\ E_h^*(x)E_i^*(x)=\delta_{hi}E_i^*(x),
\end{align}
\begin{align}\label{Eq;3}
J=\sum_{j=0}^dA_j\ \text{and}\ A_0=I=\sum_{j=0}^dE_j^*(x),
\end{align}
\begin{align}\label{Eq;4}
E_h^*(x)JE_i^*(x)\neq O\ \text{and}\ JE_i^*(x)J=\overline{k_i}J.
\end{align}
The Terwilliger $\F$-algebra of $S$ with respect to $x$ is the unital $\F$-subalgebra of $M_X(\F)$ generated by $\{E_a^*(x)A_bE_c^*(x):a,b,c\in [0,d]\}$. Use $T(x)$ to denote the Terwilliger $\F$-algebra of $S$ with respect to $x$. By the definition of $T(x)$ and \eqref{Eq;1}, $M\in T(x)$ if and only if $M^T\in T(x)$. By Lemma \cite[Lemma 2.3]{Jiang2}, $M\in \mathrm{Rad}T(x)$ if and only if $M^T\in \mathrm{Rad}T(x)$. The algebraic structures of $T(x)$ and $\mathrm{Rad}T(x)$ may depend on the choices of the fixed field $\F$ and $x$ (see \cite[5.1]{Han}). For some advancements on the algebraic structure of $T(x)$, one may refer to \cite{Han} and \cite{Jiang2}. Let $G$ be a finite group and $|G|\geq 3$. From now on, assume that $S$ is the scheme induced from the Cayley table of $G$. Fix $\mathbf{x}\in X_G$. Set $n=|G|$, $T=T(\mathbf{x})$, and $E_j^*=E_j^*(\mathbf{x})$ for any $j\in [0,4]$. Put $T_0=\langle \{E_a^*A_bE_c^*: a, b,c\in [0,4]\}\rangle_\F\subseteq T$. We conclude this section with a lemma.
\begin{lem}\label{L;Tripleproducts}
Assume that $g, h, i, r, s\in [0,4]$ and $\mathbf{y}, \mathbf{z}\in X_G$.
\begin{enumerate}[(i)]
\item [\em (i)] \cite[Lemma 3.2]{Han} $E_g^*A_hE_i^*J=\overline{p_{ih}^g}E_g^*J$ and $JE_g^*A_hE_i^*=\overline{p_{gh}^i}JE_i^*$.
\item [\em (ii)] \cite[Lemma 3.2]{Jiang2} If $E_g^*A_hE_i^*\neq O$ and $\min\{k_g, k_i\}=1$, then $E_g^*A_hE_i^*=E_g^*JE_i^*$.
\item [\em (iii)]\cite[Lemma 2.4 (i)]{Jiang2} $E_g^*A_hE_i^*\neq O$ if and only if $p_{ih}^g\neq 0$.
\item [\em (iv)]\cite[Lemma 2.4 (ii)]{Jiang2} $\{E_a^*A_bE_c^*: a, b,c\in [0,4],\ p_{cb}^a\neq0\}$ is an $\F$-basis of $T_0$.
\item [\em (v)] \cite[Lemma 4.3 (iii)]{Jiang1} $\langle\{E_a^*JE_b^*: a, b\in [0,4],\ p\mid k_ak_b\}\rangle_\F$ is a nilpotent two-sided ideal of $T$. In particular, $\langle\{E_a^*JE_b^*: a, b\in [0,4],\ p\mid k_ak_b\}\rangle_\F\subseteq \mathrm{Rad}T$.
\item [\em (vi)] The $(\mathbf{y},\mathbf{z})$-entry of $E_g^*A_hE_i^*A_rE_s^*$ is nonzero only if $\mathbf{y}\in \mathbf{x}R_g$ and $\mathbf{z}\in \mathbf{x}R_s$. If $\mathbf{y}\in \mathbf{x}R_g$ and $\mathbf{z}\in \mathbf{x}R_s$, then the $(\mathbf{y},\mathbf{z})$-entry of $E_g^*A_hE_i^*A_rE_s^*$ is equal to $\overline{|\mathbf{y}R_h\cap\mathbf{x}R_i\cap \mathbf{z}R_r|}$.
\end{enumerate}
\end{lem}

\section{Some computational results: Part I}
In this section, we prove many computational results for the elements in $T$. For the case that $G$ is an elementary abelian $2$-group, we will apply these computational results to give an $\F$-linearly independent subset of $T$. We first list a known result.
\begin{thm}\cite[Theorem 4.1]{B, CD}\label{T;Insectionnumbers}
Assume that $g, h\in [0,4]$ and $i\in[1,3]$.
\begin{enumerate}[(i)]
\item [\em (i)]We have  \[p_{gh}^0=\begin{cases} 1,\ &\ \text{if $g=h=0$},\\
n-1,\ &\ \text{if $g,h\in [1,3]$ and $g=h$},\\
n^2-3n+2,\ &\ \text{if $g=h=4$},\\
0,\ &\ \text{otherwise}.\end{cases}\]
\item [\em (ii)]We have \[p_{gh}^i=\begin{cases}
1,\ &\ \text{if $\{g, h\}=\{0,i\}$ or $\{g, h, i\}=[1, 3]$},\\
n-2,\ &\ \text{if $g=h=i$ or $\max\{g, h\}=4$ and $\min\{g, h\}\in [1, 3]\setminus\{i\}$},\\
n^2-5n+6, \ &\ \text{if $g=h=4$},\\
0,\ &\ \text{otherwise}.\end{cases}\]
\item [\em (iii)] We have \[ p_{gh}^4=\begin{cases}
1,\ &\ \text{if $\{g, h\}=\{0, 4\}$ or $g\neq h$ and $g, h\in [1,3]$},\\
n-3, \ &\ \text{if $\max\{g, h\}=4$ and $\min\{g, h\}\in [1,3]$} ,\\
n^2-6n+10, \ &\ \text{if $g=h=4$},\\
0, \ &\ \text{otherwise}. \end{cases}\]
\end{enumerate}
\end{thm}
The following lemma is necessary for us to prove many results in this section.
\begin{lem}\label{L;Lemma1.2}
The finite group $G$ is an elementary abelian $2$-group if and only if, for any $\mathbf{y}\in X_G$, the product of any two components of $\mathbf{y}$ equals its remaining component.
\end{lem}
\begin{proof}
The desired lemma is from the definition of $X_G$ and direct computation.
\end{proof}
\begin{lem}\label{L;Lemma1.3}
Assume that $G$ is an elementary abelian $2$-group and $\{g, h, i\}=[1,3]$. If $\mathbf{y}\in \mathbf{x}R_g$, $\mathbf{z}\in \mathbf{x}R_4$, and $\mathbf{z}\in \mathbf{y}R_h\cup \mathbf{y}R_i$, the $(\mathbf{y},\mathbf{z})$-entry of $E_g^*A_hE_i^*A_gE_4^*$ is zero.
\end{lem}
\begin{proof}
Assume that $\mathbf{x}=(x_1, x_2, x_3)$, $\mathbf{y}=(y_1, y_2, y_3)$, and $\mathbf{z}=(z_1, z_2, z_3)$. By Lemma \ref{L;Tripleproducts} (vi), the $(\mathbf{y}, \mathbf{z})$-entry of $E_g^*A_hE_i^*A_gE_4^*$ is $\overline{|\mathbf{y}R_h\cap\mathbf{x}R_i\cap \mathbf{z}R_g|}$. Since $\mathbf{y}\in \mathbf{x}R_g$ and $\mathbf{z}\in \mathbf{x}R_4$, notice that $x_g=y_g$ and $x_j\neq z_j$ for any $j\in [1,3]$.

Assume that $\mathbf{z}\in \mathbf{y}R_h$. So $|\mathbf{y}R_h\cap\mathbf{x}R_i\cap \mathbf{z}R_g|\leq |\mathbf{y}R_h\cap\mathbf{z}R_g|=0$ by Theorem \ref{T;Insectionnumbers} (ii). Therefore $\mathbf{y}R_h\cap\mathbf{x}R_i\cap \mathbf{z}R_g=\varnothing$ and the $(\mathbf{y},\mathbf{z})$-entry of $E_g^*A_hE_i^*A_gE_4^*$ is zero. Assume that $\mathbf{z}\in \mathbf{y}R_i$ and $\mathbf{y}R_h\cap\mathbf{x}R_i\cap \mathbf{z}R_g\neq\varnothing$. Pick $\mathbf{u}\!=\!(u_1, u_2, u_3)\in \mathbf{y}R_h\cap\mathbf{x}R_i\cap \mathbf{z}R_g$. Notice that $y_i=z_i$, $u_h=y_h$, $u_i=x_i$, and $u_g=z_g$. As $G$ is an elementary abelian $2$-group and Lemma \ref{L;Lemma1.2} holds, observe that $u_hu_g=u_i$, $x_gx_h=x_i$, $y_gy_i=y_h$, and $z_iz_g=z_h$. So $x_gz_h\!=\!y_gz_h\!=\!y_gz_iz_g\!=\!y_gy_iz_g\!=\!y_hz_g=u_hu_g=u_i=x_i\!=\!x_gx_h$. This is a contradiction as $x_j\neq z_j$ for any $j\in [1,3]$. Therefore $\mathbf{y}R_h\cap\mathbf{x}R_i\cap \mathbf{z}R_g=\varnothing$ and the $(\mathbf{y},\mathbf{z})$-entry of $E_g^*A_hE_i^*A_gE_4^*$ is zero. The desired lemma is proved.
\end{proof}
\begin{lem}\label{L;Lemma1.4}
Assume that $G$ is an elementary abelian $2$-group and $\{g, h, i\}=[1,3]$. If $n>4$, then $E_g^*A_hE_i^*A_gE_4^*, E_4^*A_gE_i^*A_hE_g^*\notin T_0$.
\end{lem}
\begin{proof}
Assume that $E_g^*A_hE_i^*A_gE_4^*\notin T_0$. By combining \eqref{Eq;2}, Lemma \ref{L;Tripleproducts} (iii), and Theorem \ref{T;Insectionnumbers} (ii), notice that
$E_g^*A_hE_i^*A_gE_4^*=c_1E_g^*A_hE_4^*+c_2E_g^*A_iE_4^*+c_3E_g^*A_4E_4^*$, where $c_1, c_2, c_3\in \F$. If $\mathbf{y}\in \mathbf{x}R_g$, $\mathbf{z}\in \mathbf{x}R_4$, and $\mathbf{z}\in \mathbf{y}R_h\cup \mathbf{y}R_i$, then the $(\mathbf{y},\mathbf{z})$-entry of $E_g^*A_hE_i^*A_gE_4^*$ is either $c_1$ or $c_2$. Hence $E_g^*A_hE_i^*A_gE_4^*=c_3E_g^*A_4E_4^*$ by Lemma \ref{L;Lemma1.3}.

As $E_4^*JE_g^*A_hE_i^*A_gE_4^*=E_4^*JE_4^*\neq O$ by combining Lemma \ref{L;Tripleproducts} (i), Theorem \ref{T;Insectionnumbers} (ii), (iii), and \eqref{Eq;4}, notice that $E_g^*A_hE_i^*A_gE_4^*\neq O$ and $c_3$ is not zero. By Lemma \ref{L;Tripleproducts} (vi), there is $\mathbf{u}\in \mathbf{x}R_g$ such that the $\mathbf{u}$-row of $E_g^*A_hE_i^*A_gE_4^*$ has at least one nonzero entry. Pick $\mathbf{v}\in \mathbf{x}R_4$ and assume that the $(\mathbf{u}, \mathbf{v})$-entry of $E_g^*A_hE_i^*A_gE_4^*$ is nonzero. Then $1\leq |\mathbf{u}R_h\cap\mathbf{x}R_i\cap\mathbf{v}R_g|\leq |\mathbf{u}R_h\cap\mathbf{x}R_i|=1$ by Lemma \ref{L;Tripleproducts} (vi) and Theorem \ref{T;Insectionnumbers} (ii). So  $\mathbf{u}R_h\cap\mathbf{x}R_i\cap\mathbf{v}R_g\!=\!\mathbf{u}R_h\cap\mathbf{x}R_i=\{\mathbf{w}\}$ and $\mathbf{v}\in \mathbf{x}R_4\cap \mathbf{w}R_g$. As $\mathbf{v}\in \mathbf{x}R_4\cap \mathbf{w}R_g$ and Theorem \ref{T;Insectionnumbers} (ii) holds, the $\mathbf{u}$-row of $E_g^*A_hE_i^*A_gE_4^*$ has at most $n-2$ nonzero entries. As $E_g^*A_hE_i^*A_gE_4^*=c_3E_g^*A_4E_4^*$ and the $\mathbf{u}$-row of $c_3E_g^*A_4E_4^*$ has exactly $n^2-5n+6$ nonzero entries by Lemma \ref{L;Tripleproducts} (i) and Theorem \ref{T;Insectionnumbers} (ii), note that $n^2-5n+6\leq n-2$. This is a contradiction as $n>4$. So $E_g^*A_hE_i^*A_gE_4^*\notin T_0$. Notice that $E_4^*A_gE_i^*A_hE_g^*\notin T_0$ by combining the fact $E_g^*A_hE_i^*A_gE_4^*\notin T_0$, taking transposes of matrices, and \eqref{Eq;1}. The desired lemma is proved.
\end{proof}
\begin{lem}\label{L;Lemma1.5}
Assume that $G$ is a Klein four group and $\{g, h,i\}=[1,3]$.
\begin{enumerate}[(i)]
\item [\em (i)] $E_g^*A_hE_i^*A_gE_4^*=E_g^*A_4E_4^*$.
\item [\em (ii)]$E_4^*A_gE_i^*A_hE_g^*=E_4^*A_4E_g^*$.
\end{enumerate}
\end{lem}
\begin{proof}
Assume that $\mathbf{x}\!=\!(x_1, x_2, x_3)$ and $\mathbf{y}\!=\!(y_1, y_2, y_3)\in\mathbf{x}R_g$. Then $|\mathbf{y}R_h\cap \mathbf{x}R_i|=1$ by Theorem \ref{T;Insectionnumbers} (ii). So there is $\mathbf{u}=(u_1, u_2, u_3)\in X_G$ such that $\mathbf{y}R_h\cap \mathbf{x}R_i=\{\mathbf{u}\}$. So $u_h=y_h$ and $u_i=x_i$. As $G$ is a Klein four group and Lemma \ref{L;Lemma1.2} holds, note that $u_g\!=\!u_hu_i\!=\!y_hx_i$. As $n\!=\!4$ and Theorem \ref{T;Insectionnumbers} (ii) holds, notice that $|\mathbf{x}R_4\cap\mathbf{u}R_g|=2$.

Pick $\mathbf{v}=(v_1, v_2, v_3)\in \mathbf{x}R_4\cap\mathbf{u}R_g$. As $\mathbf{y}\in \mathbf{u}R_h$ and $\mathbf{v}\in \mathbf{u}R_g$, $\mathbf{v}\in \mathbf{y}R_i\cup\mathbf{y}R_4$ by Theorem \ref{T;Insectionnumbers} (i), (ii), (iii). Assume further that $\mathbf{v}\in \mathbf{y}R_i$. Hence $v_g=u_g=y_hx_i$ and $v_i=y_i$. As $\mathbf{y}\in \mathbf{x}R_g$ and $\mathbf{v}\in\mathbf{x}R_4$, notice that $x_g=y_g$ and $v_j\neq x_j$ for any $j\in [1,3]$. As $G$ is a Klein four group and Lemma \ref{L;Lemma1.2} holds, notice that $v_h=v_iv_g$, $y_g=y_iy_h$, and $x_h=x_gx_i$. Hence $v_h=v_iv_g=y_iy_hx_i=y_gx_i=x_gx_i=x_h$, which contradicts the inequality $v_j\neq x_j$ for any $j\in [1,3]$. Hence $\mathbf{v}\in \mathbf{y}R_4$.
As $\mathbf{v}$ is chosen from $\mathbf{x}R_4\cap\mathbf{u}R_g$ arbitrarily,
$\mathbf{x}R_4\cap\mathbf{u}R_g=\mathbf{x}R_4\cap\mathbf{y}R_4\cap\mathbf{u}R_g\subseteq\mathbf{x}R_4\cap\mathbf{y}R_4$. As $n=4$ and Theorem \ref{T;Insectionnumbers} (ii) holds, note that $|\mathbf{x}R_4\cap\mathbf{u}R_g|=|\mathbf{x}R_4\cap\mathbf{y}R_4|=2$ and $\mathbf{x}R_4\cap\mathbf{u}R_g=\mathbf{x}R_4\cap\mathbf{y}R_4$.

Pick $\mathbf{z}\in \mathbf{x}R_4$. Notice that the $(\mathbf{y},\mathbf{z})$-entry of $E_g^*A_hE_i^*A_gE_4^*$ is $\overline{|\mathbf{y}R_h\cap\mathbf{x}R_i\cap \mathbf{z}R_g|}$ by Lemma \ref{L;Tripleproducts} (vi). Since $\mathbf{y}\in\mathbf{x}R_g$ and $\mathbf{z}\in \mathbf{x}R_4$, $\mathbf{z}\in\mathbf{y}R_h\cup\mathbf{y}R_i\cup\mathbf{y}R_4$ by Theorem \ref{T;Insectionnumbers} (i), (ii), (iii). If $\mathbf{z}\in \mathbf{y}R_h\cup\mathbf{y}R_i$, then the $(\mathbf{y},\mathbf{z})$-entry of $E_g^*A_hE_i^*A_gE_4^*$ is zero by Lemma \ref{L;Lemma1.3}. If $\mathbf{z}\in\mathbf{y}R_4$, then $\mathbf{z}\in \mathbf{x}R_4\cap \mathbf{y}R_4=\mathbf{x}R_4\cap\mathbf{u}R_g$. So $\mathbf{u}\in\mathbf{y}R_h\cap \mathbf{x}R_i\cap \mathbf{z}R_g$, which implies that $\mathbf{y}R_h\cap \mathbf{x}R_i=\{\mathbf{u}\}=\mathbf{y}R_h\cap \mathbf{x}R_i\cap \mathbf{z}R_g$. Hence the $(\mathbf{y},\mathbf{z})$-entry of $E_g^*A_hE_i^*A_gE_4^*$ is equal to one. As $\mathbf{y}, \mathbf{z}$ are chosen from $\mathbf{x}R_g$, $\mathbf{x}R_4$ arbitrarily, notice that $E_g^*A_hE_i^*A_gE_4^*=E_g^*A_4E_4^*$ by Lemma \ref{L;Tripleproducts} (vi) and the above discussion. (i) thus follows. (ii) is proved by combining (i), taking transposes of matrices, and \eqref{Eq;1}. The desired lemma is proved.  \end{proof}
The following computational results study the products of elements with two $E_4^*$'s.
\begin{lem}\label{L;Lemma1.6}
If $\{g, h, i\}=[1,3]$, $\mathbf{y},\mathbf{z}\in X_G$, and the $(\mathbf{y},\mathbf{z})$-entry of $E_4^*A_gE_h^*A_iE_4^*$ is nonzero, then $|\mathbf{y}R_g\cap\mathbf{x}R_h\cap\mathbf{z}R_i|=1$ and the $(\mathbf{y},\mathbf{z})$-entry of $E_4^*A_gE_h^*A_iE_4^*$ is one.
\end{lem}
\begin{proof}
As the $(\mathbf{y},\mathbf{z})$-entry of $E_4^*A_gE_h^*A_iE_4^*$ is nonzero, note that $\mathbf{y},\mathbf{z}\in \mathbf{x}R_4$ and the $(\mathbf{y},\mathbf{z})$-entry of $E_4^*A_gE_h^*A_iE_4^*$ is $\overline{|\mathbf{y}R_g\cap\mathbf{x}R_h\cap\mathbf{z}R_i|}$ by Lemma \ref{L;Tripleproducts} (vi). By Theorem (iii), $1\leq |\mathbf{y}R_g\cap\mathbf{x}R_h\cap\mathbf{z}R_i|\leq |\mathbf{y}R_g\cap\mathbf{x}R_h|=1$. The desired lemma thus follows.
\end{proof}
\begin{lem}\label{L;Lemma1.7}
Assume that $G$ is an elementary abelian $2$-group and $\{g,h,i\}=[1,3]$. Assume that $\mathbf{x}=(x_1, x_2, x_3)$, $\mathbf{y}=(y_1, y_2, y_3)\in \mathbf{x}R_4$, $\mathbf{z}=(z_1, z_2, z_3)\in \mathbf{x}R_4\cap\mathbf{y}R_4$. Then $x_iy_g\notin\{x_h, y_h\}$ and $|\{x_h, y_h, x_iy_g\}|=3$. If the $(\mathbf{y},\mathbf{z})$-entry of $E_4^*A_gE_h^*A_iE_4^*$ is nonzero, then $z_i=x_hy_g$ and $z_h\in G\setminus\{x_h, y_h, x_iy_g\}$.
\end{lem}
\begin{proof}
As $G$ is an elementary abelian $2$-group and Lemma \ref{L;Lemma1.2} holds, $x_h=x_ix_g$ and $y_h=y_iy_g$. As $\mathbf{y}\in\mathbf{x}R_4$, notice that $x_j\neq y_j$ for any $j\in [1,3]$.
Hence $x_iy_g\neq x_ix_g=x_h$ and $x_iy_g\neq y_iy_g=y_h$. So $x_iy_g\notin\{x_h, y_h\}$, which implies that $|\{x_h, y_h, x_iy_g\}|=3$.

Assume that the $(\mathbf{y},\mathbf{z})$-entry of $E_4^*A_gE_h^*A_iE_4^*$ is nonzero. So $|\mathbf{y}R_g\cap\mathbf{x}R_h\cap\mathbf{z}R_i|=1$ by Lemma \ref{L;Lemma1.6}. Assume further that $\{\mathbf{u}\}\!=\!\mathbf{y}R_g\cap\mathbf{x}R_h\cap\mathbf{z}R_i$ and $\mathbf{u}\!=\!(u_1, u_2, u_3)$. So $u_g=y_g$, $u_h=x_h$, and $u_i=z_i$. As $G$ is an elementary abelian $2$-group and Lemma \ref{L;Lemma1.2} holds, $z_i=u_i=u_hu_g=x_hy_g$. As $\mathbf{z}\in \mathbf{x}R_4\cap \mathbf{y}R_4$, $z_j\notin\{x_j, y_j\}$ for any $j\in [1,3]$. So $z_h\notin\{x_h, y_h\}$. If $z_h=x_iy_g$, note that $z_g=z_iz_h=x_hy_gx_iy_g=x_hx_iy_gy_g=x_hx_i=x_g$ as $G$ is an elementary abelian $2$-group and Lemma \ref{L;Lemma1.2} holds. This is a contradiction as
$x_j\neq z_j$ for any $j\in [1,3]$. So $z_h\notin \{x_h, y_h, x_iy_g\}$. The desired lemma is proved.
\end{proof}
\begin{lem}\label{L;Lemma1.8}
Assume that $G$ is an elementary abelian $2$-group and $\{g, h,i\}=[1,3]$. Assume that $\mathbf{x}=(x_1, x_2, x_3)$, $\mathbf{y}=(y_1, y_2, y_3)\in \mathbf{x}R_4$, $\mathbf{z}=(z_1, z_2, z_3)\in X_G$. Assume that $z_i=x_hy_g$ and $z_h\in G\setminus\{x_h, y_h, x_iy_g\}$. Then $\mathbf{z}\in \mathbf{x}R_4\cap\mathbf{y}R_4$ and the $(\mathbf{y}, \mathbf{z})$-entry of $E_4^*A_gE_h^*A_iE_4^*$ is one.
\end{lem}
\begin{proof}
Since $G$ is an elementary abelian $2$-group and Lemma \ref{L;Lemma1.2} holds, observe that $x_g=x_ix_h$, $x_i=x_hx_g$, $y_i=y_hy_g$, and $z_g=z_hz_i=z_hx_hy_g$. As $z_h\in G\setminus\{x_h, y_h, x_iy_g\}$ and $G$ is an elementary abelian $2$-group, notice that $z_g=z_hx_hy_g\neq x_iy_gx_hy_g=x_g$, $z_g=z_hx_hy_g\neq x_hx_hy_g=y_g$, $z_h\neq x_h$, and $z_h\neq y_h$. As $\mathbf{y}\in\mathbf{x}R_4$, notice that $x_j\neq y_j$ for any $j\in [1,3]$. So $z_i=x_hy_g\neq x_hx_g=x_i$ and $z_i=x_hy_g\neq y_hy_g=y_i$. Therefore $\mathbf{z}\in \mathbf{x}R_4\cap\mathbf{y}R_4$ as $z_g\notin\{x_g, y_g\}$, $z_h\notin\{x_h, y_h\}$, and $z_i\notin\{x_i, y_i\}$.

As $\mathbf{y}, \mathbf{z}\in \mathbf{x}R_4$, observe that the $(\mathbf{y},\mathbf{z})$-entry of $E_4^*A_gE_h^*A_iE_4^*$ is $\overline{|\mathbf{y}R_g\cap\mathbf{x}R_h\cap\mathbf{z}R_i|}$ by Lemma \ref{L;Tripleproducts} (vi). As $|\mathbf{y}R_g\cap\mathbf{x}R_h\cap\mathbf{z}R_i|\leq |\mathbf{y}R_g\cap\mathbf{x}R_h|=1$ by Theorem \ref{T;Insectionnumbers} (iii), notice that the $(\mathbf{y}, \mathbf{z})$-entry of $E_4^*A_gE_h^*A_iE_4^*$ is one if and only if $\mathbf{y}R_g\cap\mathbf{x}R_h\cap\mathbf{z}R_i\neq \varnothing$. Assume that $\mathbf{u}=(u_1, u_2, u_3)\in \mathbf{y}R_g\cap\mathbf{x}R_h$. Hence $u_g=y_g\neq z_g$ and $u_h=x_h$. As $G$ is an elementary abelian $2$-group and Lemma \ref{L;Lemma1.2} holds, $u_i=u_hu_g=x_hy_g=z_i$. So $\mathbf{u}\in \mathbf{y}R_g\cap\mathbf{x}R_h\cap \mathbf{z}R_i$ as $u_i=z_i$ and $\mathbf{u}\neq \mathbf{z}$. Hence the $(\mathbf{y},\mathbf{z})$-entry of $E_4^*A_gE_h^*A_iE_4^*$ is one.
The desired lemma is proved.
\end{proof}
Lemmas \ref{L;Lemma1.7} and \ref{L;Lemma1.8} can be utilized to deduce the following result.
\begin{lem}\label{L;Lemma1.9}
Assume that $G$ is an elementary abelian $2$-group and $\{g, h,i\}=[1,3]$. Then $E_4^*A_gE_h^*A_iE_4^*\notin T_0$.
\end{lem}
\begin{proof}
Assume that $E_4^*A_gE_h^*A_iE_4^*\in T_0$. By combining \eqref{Eq;2}, Lemma \ref{L;Tripleproducts} (iii), and Theorem \ref{T;Insectionnumbers} (iii), $E_4^*A_gE_h^*A_iE_4^*=\sum_{j=0}^4c_jE_4^*A_jE_4^*$, where $c_j\in \F$ for any $j\in [0,4]$. By Lemma \ref{L;Lemma1.8}, there exist $\mathbf{y}, \mathbf{z}\in \mathbf{x}R_4$ such that $\mathbf{z}\in \mathbf{x}R_4\cap\mathbf{y}R_4$ and the $(\mathbf{y}, \mathbf{z})$-entry of $E_4^*A_gE_h^*A_iE_4^*$ is one. Hence $c_4$ is one. Let $n_1$ denote the number of all $\mathbf{u}$'s such that $\mathbf{u}\in \mathbf{x}R_4\cap\mathbf{y}R_4$ and the $(\mathbf{y},\mathbf{u})$-entry of $E_4^*A_gE_h^*A_iE_4^*$ is one. Similarly, let $n_2$ denote the number of all $\mathbf{u}$'s such that $\mathbf{u}\in \mathbf{x}R_4\cap\mathbf{y}R_4$ and the $(\mathbf{y},\mathbf{u})$-entry of $\sum_{j=0}^4c_jE_4^*A_jE_4^*$ is one. As $E_4^*A_gE_h^*A_iE_4^*=\sum_{j=0}^4c_jE_4^*A_jE_4^*$, notice that $n_1=n_2$.

As $G$ is an elementary abelian $2$-group, $n_1=n-3$ by combining Lemmas \ref{L;Lemma1.7} and \ref{L;Lemma1.8}. As $c_4$ is one, notice that $n_2$ is the number of all nonzero entries in the $\mathbf{y}$-row of $E_4^*A_gE_h^*A_iE_4^*$. Therefore $n-3=n_1=n_2=|\mathbf{x}R_4\cap\mathbf{y}R_4|=n^2-6n+10$ by Theorem \ref{T;Insectionnumbers} (iii). This is a contradiction as $n\in\mathbb{N}$. The desired lemma thus follows.
\end{proof}
\begin{lem}\label{L;Lemma1.10}
Assume that $G$ is an elementary abelian $2$-group and $\{g, h,i\}=[1,3]$. Assume that $\mathbf{x}=(x_1, x_2, x_3)$, $\mathbf{y}=(y_1, y_2, y_3)\in \mathbf{x}R_4$, $\mathbf{z}=(z_1, z_2, z_3)\in \mathbf{x}R_4\cap\mathbf{y}R_h$. Assume that the $(\mathbf{y}, \mathbf{z})$-entry of $E_4^*A_gE_h^*A_iE_4^*$ is nonzero. Then $z_g=x_hy_i$, $z_h=y_h$, and $z_i=x_hy_g$.
\end{lem}
\begin{proof}
Since the $(\mathbf{y}, \mathbf{z})$-entry of $E_4^*A_gE_h^*A_iE_4^*$ is nonzero, $|\mathbf{y}R_g\cap\mathbf{x}R_h\cap \mathbf{z}R_i|=1$ by Lemma \ref{L;Lemma1.6}. So there is $\mathbf{u}=(u_1, u_2, u_3)\in X_G$ such that $\mathbf{y}R_g\cap\mathbf{x}R_h\cap \mathbf{z}R_i=\{\mathbf{u}\}$. So $u_g=y_g$, $u_h=x_h$, and $u_i=z_i$. As $G$ is an elementary abelian $2$-group and Lemma \ref{L;Lemma1.2} holds, notice that $z_i=u_i=u_hu_g=x_hy_g$. As $\mathbf{z}\in \mathbf{x}R_4\cap\mathbf{y}R_h$, note that $z_h=y_h$. As $G$ is an elementary abelian $2$-group and Lemma \ref{L;Lemma1.2} holds, notice that $z_g=z_iz_h$ and $y_i=y_gy_h$. Hence $z_g=z_iz_h=x_hy_gy_h=x_hy_i$. The desired lemma is proved.
\end{proof}
\begin{lem}\label{L;Lemma1.11}
Assume that $G$ is an elementary abelian $2$-group and $\{g,h,i\}=[1,3]$. Assume that $\mathbf{x}=(x_1, x_2, x_3)$, $\mathbf{y}=(y_1,y_2,y_3)\in\mathbf{x}R_4$, and $\mathbf{z}=(z_1, z_2, z_3)\in X_G$. If $z_g=x_hy_i$, $z_h=y_h$, and $z_i=x_hy_g$, then $\mathbf{z}\in \mathbf{x}R_4\cap \mathbf{y}R_h$ and the $(\mathbf{y},\mathbf{z})$-entry of $E_4^*A_gE_h^*A_iE_4^*$ is one.
\end{lem}
\begin{proof}
As $\mathbf{y}\in \mathbf{x}R_4$, observe that $x_j\neq y_j$ for any $j\in [1,3]$. As $G$ is an elementary abelian $2$-group and Lemma \ref{L;Lemma1.2} holds, notice that $x_g=x_hx_i$, $x_i=x_hx_g$, $y_g=y_hy_i$. So $z_g\!=x_hy_i\neq x_hx_i\!=x_g$, $z_h=y_h\neq x_h$, and $z_i=x_hy_g\neq x_hx_g=x_i$. Moreover, note that $z_g=x_hy_i\neq y_hy_i=y_g$ and $z_h=y_h$. Therefore $\mathbf{z}\in \mathbf{x}R_4\cap \mathbf{y}R_h$.

As $\mathbf{y},\mathbf{z}\in \mathbf{x}R_4$, observe that the $(\mathbf{y},\mathbf{z})$-entry of $E_4^*A_gE_h^*A_iE_4^*$ is $\overline{|\mathbf{y}R_g\cap\mathbf{x}R_h\cap\mathbf{z}R_i|}$ by Lemma \ref{L;Tripleproducts} (vi). As $|\mathbf{y}R_g\cap\mathbf{x}R_h\cap\mathbf{z}R_i|\leq |\mathbf{y}R_g\cap\mathbf{x}R_h|=1$ by Theorem \ref{T;Insectionnumbers} (iii), notice that the $(\mathbf{y}, \mathbf{z})$-entry of $E_4^*A_gE_h^*A_iE_4^*$ is one if and only if $\mathbf{y}R_g\cap\mathbf{x}R_h\cap\mathbf{z}R_i\neq \varnothing$. Assume that $\mathbf{u}=(u_1, u_2, u_3)\in \mathbf{y}R_g\cap\mathbf{x}R_h$. Therefore $u_g=y_g\neq z_g$ and $u_h=x_h$. As $G$ is an elementary abelian $2$-group and Lemma \ref{L;Lemma1.2} holds, $u_i=u_hu_g=x_hy_g=z_i$. Therefore $\mathbf{z}\in \mathbf{y}R_g\cap \mathbf{x}R_h\cap\mathbf{z}R_i$ as $u_i=z_i$ and $\mathbf{u}\neq \mathbf{z}$. Therefore the $(\mathbf{y}, \mathbf{z})$-entry of $E_4^*A_gE_h^*A_iE_4^*$ is one. The desired lemma is proved.
\end{proof}
\begin{lem}\label{L;Lemma1.12}
Assume that $G$ is an elementary abelian $2$-group and $\{g, h,i\}=[1,3]$. If $\mathbf{y}, \mathbf{z}\in \mathbf{x}R_4$ and $\mathbf{z}\in \mathbf{y}R_0\cup \mathbf{y}R_g\cup \mathbf{y}R_i$, the $(\mathbf{y},\mathbf{z})$-entry of $E_4^*A_gE_h^*A_iE_4^*$ is zero.
\end{lem}
\begin{proof}
Notice that the $(\mathbf{y}, \mathbf{z})$-entry of $E_4^*A_gE_h^*A_iE_4^*$ is $\overline{|\mathbf{y}R_g\cap\mathbf{x}R_h\cap\mathbf{z}R_i|}$ by Lemma \ref{L;Tripleproducts} (vi).
As $\mathbf{z}\in \mathbf{y}R_0\cup \mathbf{y}R_g\cup \mathbf{y}R_i$, notice that $|\mathbf{y}R_g\cap\mathbf{x}R_h\cap\mathbf{z}R_i|\leq |\mathbf{y}R_g\cap \mathbf{z}R_i|=0$ by Theorem \ref{T;Insectionnumbers} (i) and (ii). The desired lemma thus follows.
\end{proof}
\begin{lem}\label{L;Lemma1.13}
Assume that $G$ is an elementary abelian $2$-group and $\{g, h,i\}=[1,3]$. Assume that $\mathbf{x}=(x_1, x_2, x_3)$, $\mathbf{y}, \mathbf{z}\in \mathbf{x}R_4$, where $\mathbf{y}=(y_1, y_2, y_3)$ and $\mathbf{z}=(z_1, z_2, z_3)$. The $(\mathbf{y}, \mathbf{z})$-entry of $E_4^*A_gE_h^*A_iE_4^*$ and the $(\mathbf{y},\mathbf{z})$-entry of $E_4^*A_iE_h^*A_gE_4^*$ are nonzero if and only if $z_g=x_hy_i$, $z_h=y_h$, and $z_i=x_hy_g$.
\end{lem}
\begin{proof}
Assume that the $(\mathbf{y},\mathbf{z})$-entries of $E_4^*A_gE_h^*A_iE_4^*$ and $E_4^*A_iE_h^*A_gE_4^*$ are nonzero. Notice that $\mathbf{z}\in \mathbf{y}R_4\cup \mathbf{y}R_h$ by Lemma \ref{L;Lemma1.12}. Assume further that $\mathbf{z}\in \mathbf{y}R_4$. Therefore $y_j\neq z_j$ for any $j\in [1,3]$. Apply Lemma \ref{L;Lemma1.7} to $E_4^*A_gE_h^*A_iE_4^*$, $E_4^*A_iE_h^*A_gE_4^*$ and notice that $z_i=x_hy_g$ and $z_g=x_hy_i$. Since $G$ is an elementary abelian $2$-group and Lemma \ref{L;Lemma1.2} holds, $z_h=z_iz_g=x_hy_gx_hy_i=x_hx_hy_gy_i=y_h$. This is absurd as $y_j\neq z_j$ for any $j\in [1,3]$. So $\mathbf{z}\in\mathbf{y}R_h$. So $z_g=x_hy_i$, $z_h=y_h$, $z_i=x_hy_g$ by Lemma \ref{L;Lemma1.10}.

Assume that $z_g=x_hy_i$, $z_h=y_h$, and $z_i=x_hy_g$. By Lemma \ref{L;Lemma1.11}, observe that the $(\mathbf{y},\mathbf{z})$-entries of $E_4^*A_gE_h^*A_iE_4^*$ and $E_4^*A_iE_h^*A_gE_4^*$ equal one. So the $(\mathbf{y},\mathbf{z})$-entries of $E_4^*A_gE_h^*A_iE_4^*$ and $E_4^*A_iE_h^*A_gE_4^*$ are nonzero. The desired lemma is proved.
\end{proof}
\begin{lem}\label{L;Lemma1.14}
Assume that $G$ is an elementary abelian $2$-group and $\{g,h,i\}=[1,3]$. Assume that $\mathbf{x}=(x_1, x_2, x_3)$, $\mathbf{y}, \mathbf{z}\in \mathbf{x}R_4$, where $\mathbf{y}=(y_1, y_2, y_3)$ and $\mathbf{z}=(z_1, z_2, z_3)$. The $(\mathbf{y}, \mathbf{z})$-entry of $E_4^*A_gE_h^*A_iE_4^*$ and the $(\mathbf{y},\mathbf{z})$-entry of $E_4^*A_hE_g^*A_iE_4^*$ are not all nonzero. The $(\mathbf{y},\mathbf{z})$-entry of $E_4^*A_gE_h^*A_iE_4^*$ and the $(\mathbf{y},\mathbf{z})$-entry of $E_4^*A_gE_i^*A_hE_4^*$ are not all nonzero.
\end{lem}
\begin{proof}
Assume that the $(\mathbf{y},\mathbf{z})$-entries of $E_4^*A_gE_h^*A_iE_4^*$ and $E_4^*A_hE_g^*A_iE_4^*$ are nonzero. Notice that $\mathbf{z}\in\mathbf{y}R_4$ by Lemma \ref{L;Lemma1.12}.
Apply Lemma \ref{L;Lemma1.7} to $E_4^*A_gE_h^*A_iE_4^*$ and notice that $z_i=x_hy_g$. Apply Lemma \ref{L;Lemma1.7} to $E_4^*A_hE_g^*A_iE_4^*$ and notice that $z_i=x_gy_h$. Hence  $x_gx_h=y_gy_h$. As $\mathbf{y}\in \mathbf{x}R_4$, $x_j\neq y_j$ for any $j\in [1,3]$. As $G$ is an elementary abelian $2$-group and Lemma \ref{L;Lemma1.2} holds, $x_i=x_gx_h$, $y_i=y_gy_h$, and $x_i=x_gx_h=y_gy_h=y_i$. It is impossible as $x_j\neq y_j$ for any $j\in [1,3]$. The first statement thus follows.

Assume that the $(\mathbf{y},\mathbf{z})$-entries of $E_4^*A_gE_h^*A_iE_4^*$ and $E_4^*A_gE_i^*A_hE_4^*$ are nonzero. Notice that $\mathbf{z}\in \mathbf{y}R_4$ by Lemma \ref{L;Lemma1.12}.
Apply Lemma \ref{L;Lemma1.7} to $E_4^*A_gE_h^*A_iE_4^*$ and notice that $z_i=x_hy_g$. Apply Lemma \ref{L;Lemma1.7} to $E_4^*A_gE_i^*A_hE_4^*$ and notice that $z_h=x_iy_g$. Since $\mathbf{z}\in\mathbf{x}R_4$, observe that $x_j\neq z_j$ for any $j\in [1,3]$. Since $G$ is an elementary abelian $2$-group and Lemma \ref{L;Lemma1.2} holds, notice that $z_g=z_iz_h$ and  $x_g=x_hx_i$. Furthermore, $z_g=z_iz_h=x_hy_gx_iy_g=x_hx_iy_gy_g=x_hx_i=x_g$. This is impossible as $x_j\neq z_j$ for any $j\in [1,3]$. So the second statement also follows. The desired lemma is proved.
\end{proof}
\begin{lem}\label{L;Lemma1.15}
Assume that $G$ is an elementary abelian $2$-group and $\{g, h,i\}=[1,3]$. Assume that $\mathbf{x}=(x_1, x_2, x_3)$, $\mathbf{y}, \mathbf{z}\in \mathbf{x}R_4$, where $\mathbf{y}=(y_1, y_2, y_3)$ and $\mathbf{z}=(z_1,z_2,z_3)$. The $(\mathbf{y},\mathbf{z})$-entry of $E_4^*A_gE_h^*A_iE_4^*$ and the $(\mathbf{y},\mathbf{z})$-entry of $E_4^*A_iE_g^*A_hE_4^*$ are nonzero if and only if $z_g=x_iy_h$, $z_h=x_gy_i$, and $z_i=x_hy_g$. The $(\mathbf{y},\mathbf{z})$-entry of $E_4^*A_gE_h^*A_iE_4^*$ and the $(\mathbf{y},\mathbf{z})$-entry of $E_4^*A_hE_i^*A_gE_4^*$ are nonzero if and only if $z_g=x_iy_h$, $z_h=x_gy_i$, and $z_i=x_hy_g$.
\end{lem}
\begin{proof}
Assume that the $(\mathbf{y},\mathbf{z})$-entries of $E_4^*A_gE_h^*A_iE_4^*$ and $E_4^*A_iE_g^*A_hE_4^*$ are nonzero. Notice that $\mathbf{z}\in \mathbf{y}R_4$ by Lemma \ref{L;Lemma1.12}.
Apply Lemma \ref{L;Lemma1.7} to $E_4^*A_gE_h^*A_iE_4^*$ and notice that $z_i=x_hy_g$. Apply \ref{L;Lemma1.7} Lemma to $E_4^*A_iE_g^*A_hE_4^*$ and notice that $z_h=x_gy_i$. As $G$ is an elementary abelian $2$-group and Lemma \ref{L;Lemma1.2} holds, observe that $z_g=z_iz_h$, $y_h=y_gy_i$, and $x_i=x_hx_g$. Therefore $z_g=z_iz_h=x_hy_gx_gy_i=x_hx_gy_gy_i=x_iy_h$.

Assume that the $(\mathbf{y},\mathbf{z})$-entries of $E_4^*A_gE_h^*A_iE_4^*$ and $E_4^*A_hE_i^*A_gE_4^*$ are nonzero. Notice that $\mathbf{z}\in \mathbf{y}R_4$ by Lemma \ref{L;Lemma1.12}.
Apply Lemma \ref{L;Lemma1.7} to $E_4^*A_gE_h^*A_iE_4^*$ and notice that $z_i=x_hy_g$. Apply Lemma \ref{L;Lemma1.7} to $E_4^*A_hE_i^*A_gE_4^*$ and notice that $z_g=x_iy_h$. As $G$ is an elementary abelian $2$-group and Lemma \ref{L;Lemma1.2} holds, observe that $z_h=z_iz_g$, $x_g=x_hx_i$, and $y_i=y_gy_h$. Therefore $z_h=\!z_iz_g=\!x_hy_gx_iy_h=\!x_hx_iy_gy_h=x_gy_i$.

Assume that $z_g=x_iy_h$, $z_h=x_gy_i$, and $z_i=x_hy_g$. Since $\mathbf{y}\in \mathbf{x}R_4$, notice that $x_j\neq y_j$ for any $j\in [1,3]$. Since $G$ is an elementary abelian $2$-group and Lemma \ref{L;Lemma1.2} holds, notice that $x_g=x_hx_i$, $x_h=x_gx_i$, $y_h=y_gy_i$, and $y_i=y_gy_h$. Therefore $z_h=x_gy_i\neq x_gx_i=x_h$, $z_h=x_gy_i\neq y_gy_i=y_h$, and $z_h=x_gy_i=x_hx_iy_gy_h\neq x_iy_g$. So $z_h\in G\setminus\{x_h, y_h, x_iy_g\}$. Similarly, one can mimic the proof listed in the above four lines to show that $z_g\in G\setminus\{x_g, y_g, x_hy_i\}$ and $z_i\in G\setminus\{x_i, y_i,x_gy_h\}$. By Lemma \ref{L;Lemma1.8}, notice that the $(\mathbf{y},\mathbf{z})$-entries of $E_4^*A_gE_h^*A_iE_4^*$, $E_4^*A_iE_g^*A_hE_4^*$, and $E_4^*A_hE_i^*A_gE_4^*$ equal one. The desired lemma is proved.
\end{proof}
The following lemmas study the $\F$-linear combinations of some elements in $T$.
\begin{lem}\label{L;Lemma1.16}
Assume that $G$ is an elementary abelian $2$-group. Let $L$ denote an $\F$-linear combination of
$E_4^*A_1E_2^*A_3E_4^*$, $E_4^*A_1E_3^*A_2E_4^*$, $E_4^*A_2E_1^*A_3E_4^*$, $E_4^*A_2E_3^*A_1E_4^*$, $E_4^*A_3E_1^*A_2E_4^*$, and $E_4^*A_3E_2^*A_1E_4^*$ whose coefficients of elements are not all zeros. If $n>4$ and $L\in T_0$, then $L=cE_4^*A_4E_4^*$ for some nonzero element $c\in \F$.
\end{lem}
\begin{proof}
By combining the assumption $L\in T_0$, \eqref{Eq;2}, Lemma \ref{L;Tripleproducts} (iii), and Theorem \ref{T;Insectionnumbers} (iii), notice that there are $c_0, c_1, c_2, c_3, c_4\in \F$ such that $L=\sum_{j=0}^4c_jE_4^*A_jE_4^*$. By Lemma \ref{L;Lemma1.12} and the definition of $L$, all diagonal entries of $L$ equal zero. Notice that each diagonal entry of $\sum_{j=0}^4c_jE_4^*A_jE_4^*$ is either $c_0$ or zero. As $L=\sum_{j=0}^4c_jE_4^*A_jE_4^*$, notice that $c_0$ is zero. So $L=\sum_{j=1}^4c_jE_4^*A_jE_4^*$.

Pick $\mathbf{y}\in \mathbf{x}R_4$ and $i\in [1,3]$. Assume further that $c_i$ is nonzero. Let $n_1$ denote the number of all $\mathbf{u}$'s such that $\mathbf{u}\in \mathbf{x}R_4\cap \mathbf{y}R_i$ and the $(\mathbf{y},\mathbf{u})$-entry of $L$ is nonzero. Let $n_2$ denote the number of all $\mathbf{u}$'s such that $\mathbf{u}\in\mathbf{x}R_4\cap \mathbf{y}R_i$ and the $(\mathbf{y},\mathbf{u})$-entry of $\sum_{j=1}^4c_jE_4^*A_jE_4^*$ is nonzero. By combining Lemmas \ref{L;Lemma1.10}, \ref{L;Lemma1.11}, \ref{L;Lemma1.12}, \ref{L;Lemma1.13}, and the definition of $L$, notice that $n_1\leq 1$. As $c_i$ is nonzero, notice that $n_2$ is the number of all nonzero entries in the $\mathbf{y}$-row of $E_4^*A_iE_4^*$. Since $n>4$ and Theorem \ref{T;Insectionnumbers} (iii) holds, notice that $n_2=|\mathbf{x}R_4\cap\mathbf{y}R_i|=n-3>1\geq n_1$. As $n_1\neq n_2$, observe that $L\neq\sum_{j=1}^4c_jE_4^*A_jE_4^*$. It is an obvious contradiction. Hence $c_i$ is zero. As $i$ is chosen from $[1,3]$ arbitrarily, notice that $c_1, c_2, c_3$ equal zero. Hence $L=c_4E_4^*A_4E_4^*$. By the definition of $L$, there is no loss to assume that the coefficient of $E_4^*A_1E_2^*A_3E_4^*$ in $L$ is nonzero. Let $n_3$ denote the number of all $\mathbf{u}$'s such that $\mathbf{u}\in\mathbf{x}R_4\cap \mathbf{y}R_4$ and the $(\mathbf{y},\mathbf{u})$-entry of $E_4^*A_1E_2^*A_3E_4^*$ is nonzero. As $n>4$ and Lemmas \ref{L;Lemma1.7}, \ref{L;Lemma1.8} hold, note that $n_3=n-3>1$. By Lemmas \ref{L;Lemma1.13}, \ref{L;Lemma1.14}, \ref{L;Lemma1.15}, and the definition of $L$, there is $\mathbf{v}\in \mathbf{x}R_4\cap\mathbf{y}R_4$ such that the $(\mathbf{y}, \mathbf{v})$-entries of $L$ and $E_4^*A_1E_2^*A_3^*E_4^*$ are nonzero. So $L\neq O$ and $c_4$ is nonzero. The desired lemma is proved by setting $c=c_4$.
\end{proof}
\begin{lem}\label{L;Lemma1.17}
Assume that $G$ is an elementary abelian $2$-group and $n>4$. Let $L$ denote an $\F$-linear combination of $E_4^*A_1E_2^*A_3E_4^*$, $E_4^*A_1E_3^*A_2E_4^*$, $E_4^*A_2E_1^*A_3E_4^*$, $E_4^*A_2E_3^*A_1E_4^*$, $E_4^*A_3E_1^*A_2E_4^*$, and $E_4^*A_3E_2^*A_1E_4^*$ whose coefficients of elements are not all zeros. If the coefficient of $E_4^*A_1E_2^*A_3E_4^*$ in $L$ is zero or $n>8$, then $L\notin T_0$.
\end{lem}
\begin{proof}
Assume that $L\in T_0$. By Lemma \ref{L;Lemma1.16}, there is $c\in \F$ such that $L=cE_4^*A_4E_4^*$ and $c$ is nonzero. Pick $\mathbf{y}\in \mathbf{x}R_4$. Set $\{g ,h, i\}=[1,3]$. Let $n_{ghi}$ be the number of all $\mathbf{u}$'s such that $\mathbf{u}\in \mathbf{x}R_4\cap\mathbf{y}R_4$ and the $(\mathbf{y},\mathbf{u})$-entry of $E_4^*A_gE_h^*A_iE_4^*$ is nonzero. By combining the assumption that $G$ is an elementary abelian $2$-group, Lemmas \ref{L;Lemma1.7}, and \ref{L;Lemma1.8}, observe that $n_{ghi}=n-3$. Let $n_1$ denote the number of all $\mathbf{u}$'s such that $\mathbf{u}\in \mathbf{x}R_4\cap\mathbf{y}R_4$ and the $(\mathbf{y},\mathbf{u})$-entry of $L$ is nonzero. Let $n_2$ denote the number of all $\mathbf{u}$'s such that $\mathbf{u}\in \mathbf{x}R_4\cap\mathbf{y}R_4$ and the $(\mathbf{y},\mathbf{u})$-entry of $cE_4^*A_4E_4^*$ is nonzero. As $L=cE_4^*A_4E_4^*$ and Theorem \ref{T;Insectionnumbers} (iii) holds, $n_1=n_2=|\mathbf{x}R_4\cap\mathbf{y}R_4|=n^2-6n+10$.

Assume that the coefficient of $E_4^*A_1E_2^*A_3E_4^*$ in $L$ is zero. By the definitions of $L$, $n_1$, and $n_{ghi}$, note that $n_1\leq n_{132}+n_{213}+n_{231}+n_{312}+n_{321}=5(n-3)$. As $G$ is an elementary abelian $2$-group and $n>4$, $n_2=n^2-6n+10>5(n-3)\geq n_1$. This is a contradiction as $n_1=n_2$. So $L\notin T_0$. Assume that $n>8$. By the definitions $L$, $n_1$, and $n_{ghi}$, note that $n_1\leq n_{123}+n_{132}+n_{213}+n_{231}+n_{312}+n_{321}=6(n-3)$. As $n>8$, $n_2=n^2-6n+10>6(n-3)\geq n_1$. It is absurd as $n_1=n_2$. So $L\notin T_0$.
\end{proof}
The following three computational results investigate the cases $n=8$ and $n=4$.
\begin{lem}\label{L;Lemma1.18}
Assume that $G$ is an elementary abelian $2$-group. Let $L$ denote an $\F$-linear combination of
$E_4^*A_1E_2^*A_3E_4^*$, $E_4^*A_1E_3^*A_2E_4^*$, $E_4^*A_2E_1^*A_3E_4^*$, $E_4^*A_2E_3^*A_1E_4^*$, $E_4^*A_3E_1^*A_2E_4^*$, and $E_4^*A_3E_2^*A_1E_4^*$ whose coefficients of elements are not all zeros. If $p\neq2$ and $n>4$, then $L\notin T_0$. In particular, $L\notin T_0$ if $p\neq2$ and $n=8$.
\end{lem}
\begin{proof}
Assume that $L\in T_0$ and $\{g, h,i\}=[1,3]$. Let $c_{ghi}$ denote the coefficient of $E_4^*A_gE_h^*A_iE_4^*$ in $L$. As $n>4$ and Lemma \ref{L;Lemma1.16} holds, notice that $L=cE_4^*A_4E_4^*$ for some nonzero element $c\in \F$. In particular, $L\neq O$. Pick $\mathbf{y}\in \mathbf{x}R_4$. By combining Lemmas \ref{L;Lemma1.10}, \ref{L;Lemma1.11}, \ref{L;Lemma1.12}, \ref{L;Lemma1.13}, and the definition of $L$, notice that there exists a unique $\mathbf{z}\in\mathbf{x}R_4$ such that $\mathbf{z}\in\mathbf{x}R_4\cap\mathbf{y}R_h$ and $(\mathbf{y},\mathbf{z})$-entry of $L$ is $c_{ghi}+c_{ihg}$. Since $h\neq 4$, the $(\mathbf{y},\mathbf{z})$-entry of $cE_4^*A_4E_4^*$ is zero. So $c_{ghi}+c_{ihg}$ is zero as $L=cE_4^*A_4E_4^*$. So $c_{132}=-c_{231}$, $c_{213}=-c_{312}$, and $c_{321}=-c_{123}$. Denote the sum of $c_{123}E_4^*A_1E_2^*A_3E_4^*$, $c_{231}E_4^*A_2E_3^*A_1E_4^*$, and $c_{312}E_4^*A_3E_1^*A_2E_4^*$ by $M$. Hence $L=M-M^T$ by \eqref{Eq;1}. As $L=cE_4^*A_4E_4^*$, notice that $L^T=L$ by \eqref{Eq;1}. Hence $L=L^T=M^T-M=-L$. So $\overline{2}L=O$. As $p\neq 2$, notice that $L=O$. This is a contradiction as $L\neq O$. Hence $L\notin T_0$. The first statement and the second statement are proved. We are done.
\end{proof}
\begin{lem}\label{L;Lemma1.19}
Assume that $G$ is an elementary abelian $2$-group. If $p=2$ and $n=8$, the sum of the matrices $E_4^*A_1E_2^*A_3E_4^*$, $E_4^*A_1E_3^*A_2E_4^*$, $E_4^*A_2E_1^*A_3E_4^*$, $E_4^*A_2E_3^*A_1E_4^*$, $E_4^*A_3E_1^*A_2E_4^*$, and $E_4^*A_3E_2^*A_1E_4^*$ equals $E_4^*A_4E_4^*$.
\end{lem}
\begin{proof}
Set $M=E_4^*A_1E_2^*A_3E_4^*+E_4^*A_2E_3^*A_1E_4^*+E_4^*A_3E_1^*A_2E_4^*$. By \eqref{Eq;1}, it suffices to check that $M+M^T=E_4^*A_4E_4^*$. Pick $\mathbf{y}, \mathbf{z}\in \mathbf{x}R_4$. By combining Lemmas \ref{L;Lemma1.10}, \ref{L;Lemma1.11}, \ref{L;Lemma1.12}, and the definition of $M$, there are unique $\mathbf{z}^{(1)}, \mathbf{z}^{(2)}, \mathbf{z}^{(3)}\in \mathbf{x}R_4$ such that $(\mathbf{y},\mathbf{z}^{(j)})\in R_j$ and the $(\mathbf{y}, \mathbf{z}^{(j)})$-entry of $M$ is one for any $j\in [1,3]$. By combining Lemmas \ref{L;Lemma1.10}, \ref{L;Lemma1.11}, \ref{L;Lemma1.12}, \ref{L;Lemma1.13}, and the definition of $M$, notice that $\mathbf{z}^{(1)}, \mathbf{z}^{(2)}, \mathbf{z}^{(3)}$ are also the unique elements in $\mathbf{x}R_4$ such that $(\mathbf{y},\mathbf{z}^{(j)})\in R_j$ and the $(\mathbf{y}, \mathbf{z}^{(j)})$-entry of $M^T$ is one for any $j\in [1,3]$. As $p=2$ and Lemma \ref{L;Lemma1.12} holds, the $(\mathbf{y},\mathbf{z})$-entry of $M+M^T$ is nonzero only if $\mathbf{z}\in \mathbf{x}R_4\cap \mathbf{y}R_4$. Hence the $(\mathbf{y},\mathbf{z})$-entry of $E_4^*A_4E_4^*$ is nonzero if the $(\mathbf{y},\mathbf{z})$-entry of $M+M^T$ is nonzero.

As the adjacency $\F$-matrices and the dual $\F$-idempotents are all $\{0,1\}$-matrices, the assumption $p=2$ implies that all nonzero entries of $E_4^*A_1E_2^*A_3E_4^*$, $E_4^*A_2E_3^*A_1E_4^*$, $E_4^*A_3E_1^*A_2E_4^*$, and $E_4^*A_4E_4^*$ equal one. So all nonzero entries of $M$, $M^T$, and $M+M^T$ equal one. As $n=8$ and the $(\mathbf{y},\mathbf{z})$-entry of $M+M^T$ is nonzero only if $\mathbf{z}\in \mathbf{x}R_4\cap \mathbf{y}R_4$, notice that the number of all nonzero entries in the $\mathbf{y}$-row of $M+M^T$ is twenty-six by combining Lemmas \ref{L;Lemma1.7}, \ref{L;Lemma1.8}, \ref{L;Lemma1.12}, \ref{L;Lemma1.13}, \ref{L;Lemma1.14}, \ref{L;Lemma1.15}. Observe that the number of all nonzero entries in the $\mathbf{y}$-row of $E_4^*A_4E_4^*$ is $|\mathbf{x}R_4\cap \mathbf{y}R_4|$. It is also twenty-six as $n=8$ and $|\mathbf{x}R_4\cap \mathbf{y}R_4|=26$ by Theorem \ref{T;Insectionnumbers} (iii). Since all nonzero entries of $M+M^T$ and $E_4^*A_4E_4^*$ equal one, $\mathbf{z}$ is chosen from $\mathbf{x}R_4$ arbitrarily, and the $(\mathbf{y}, \mathbf{z})$-entry of $E_4^*A_4E_4^*$ is nonzero if the $(\mathbf{y},\mathbf{z})$-entry of $M+M^T$ is nonzero, notice that the $\mathbf{y}$-row of $E_4^*A_4E_4^*$ equals the $\mathbf{y}$-row of $M+M^T$ as the row vectors of matrices. As $\mathbf{y}$ is chosen from $\mathbf{x}R_4$ arbitrarily and Lemma \ref{L;Tripleproducts} (vi) holds, $M+M^T=E_4^*A_4E_4^*$. We are done.
\end{proof}
\begin{lem}\label{L;Lemma1.20}
Assume that $G$ is a Klein four group.
\begin{enumerate}[(i)]
\item [\em (i)] $E_4^*A_1E_3^*A_2E_4^*+E_4^*A_1E_2^*A_3E_4^*=E_4^*A_2E_4^*+E_4^*A_3E_4^*+E_4^*A_4E_4^*$.
\item [\em (ii)]$E_4^*A_2E_1^*A_3E_4^*+E_4^*A_1E_2^*A_3E_4^*=E_4^*A_1E_4^*+E_4^*A_2E_4^*+E_4^*A_4E_4^*$.
\item [\em (iii)]$E_4^*A_2E_3^*A_1E_4^*-E_4^*A_1E_2^*A_3E_4^*=E_4^*A_3E_4^*-E_4^*A_2E_4^*$.
\item [\em (iv)]$E_4^*A_3E_1^*A_2E_4^*-E_4^*A_1E_2^*A_3E_4^*=E_4^*A_1E_4^*-E_4^*A_2E_4^*$.
\item [\em (v)]$E_4^*A_3E_2^*A_1E_4^*+E_4^*A_1E_2^*A_3E_4^*=\overline{2}E_4^*A_2E_4^*+E_4^*A_4E_4^*$.
\end{enumerate}
\end{lem}
\begin{proof}
Assume that $\{g, h,i\}=[1,3]$ and $\mathbf{y},\mathbf{z}\in \mathbf{x}R_4$. By Lemma \ref{L;Lemma1.12}, notice that the $(\mathbf{y},\mathbf{z})$-entry of $E_4^*A_gE_h^*A_iE_4^*$ is nonzero only if $\mathbf{z}\in \mathbf{y}R_4\cup \mathbf{y}R_h$. According to Lemmas \ref{L;Lemma1.10} and \ref{L;Lemma1.11}, there is a unique $\mathbf{u}\in X_G$ such that $\mathbf{u}\in \mathbf{x}R_4\cap\mathbf{y}R_h$ and the $(\mathbf{y},\mathbf{u})$-entry of $E_4^*A_gE_h^*A_iE_4^*$ is one. Since $n=4$, $|\mathbf{x}R_4\cap\mathbf{y}R_h|=1$ by Theorem \ref{T;Insectionnumbers} (iii). Therefore $\mathbf{x}R_4\cap\mathbf{y}R_h=\{\mathbf{u}\}$. As the number of all nonzero entries in the $\mathbf{y}$-row of $E_4^*A_hE_4^*$ is $|\mathbf{x}R_4\cap\mathbf{y}R_h|$, observe that the $(\mathbf{y},\mathbf{u})$-entry of $E_4^*A_hE_4^*$ is the unique nonzero entry in the $\mathbf{y}$-row of $E_4^*A_hE_4^*$. Moreover, notice that the $(\mathbf{y},\mathbf{u})$-entry of $E_4^*A_hE_4^*$ is one. So the $(\mathbf{y},\mathbf{z})$-entry of $E_4^*A_gE_h^*A_iE_4^*-E_4^*A_hE_4^*$ is nonzero only if $\mathbf{z}\in \mathbf{y}R_4$. If $\mathbf{z}\in \mathbf{y}R_4$, notice that the $(\mathbf{y},\mathbf{z})$-entry of $E_4^*A_gE_h^*A_iE_4^*-E_4^*A_hE_4^*$ equals the $(\mathbf{y},\mathbf{z})$-entry of $E_4^*A_gE_h^*A_iE_4^*$. For convenience, write $M_{ghi}$ for $E_4^*A_gE_h^*A_iE_4^*-E_4^*A_hE_4^*$.

By combining the assumption $n=4$, Lemmas \ref{L;Lemma1.7}, and \ref{L;Lemma1.8}, notice that there is a unique $\mathbf{v}_{ghi}\in X_G$ such that $\mathbf{v}_{ghi}\in\mathbf{x}R_4\cap\mathbf{y}R_4$ and the $(\mathbf{y}, \mathbf{v}_{ghi})$-entry of $M_{ghi}$ is nonzero. Moreover, the $(\mathbf{y}, \mathbf{v}_{ghi})$-entry of
$M_{ghi}$ is one. As the $(\mathbf{y},\mathbf{z})$-entry of $M_{ghi}$ is nonzero only if $\mathbf{z}\in\mathbf{y}R_4$, note that the $(\mathbf{y},\mathbf{v}_{ghi})$-entry of $M_{ghi}$ is the unique nonzero entry in the $\mathbf{y}$-row of $M_{ghi}$. Moreover, the $(\mathbf{y},\mathbf{v}_{ghi})$-entry of $M_{ghi}$ equals one. In the following paragraphs, these facts shall be applied to $M_{123}$, $M_{132}$, $M_{213}$, $M_{231}$, $M_{312}$, $M_{321}$, $\mathbf{v}_{123}$, $\mathbf{v}_{132}$, $\mathbf{v}_{213}$, $\mathbf{v}_{231}$, $\mathbf{v}_{312}$, and $\mathbf{v}_{321}$.

According to Lemma \ref{L;Lemma1.15}, there is a unique $\mathbf{v}\in \mathbf{x}R_4$ such that the $(\mathbf{y},\mathbf{v})$-entries of $E_4^*A_1E_2^*A_3E_4^*$, $E_4^*A_2E_3^*A_1E_4^*$, and $E_4^*A_3E_1^*A_2E_4^*$ are nonzero. Hence $\mathbf{v}\in\mathbf{y}R_4$ by Lemma \ref{L;Lemma1.12}. Since the $(\mathbf{y},\mathbf{v})$-entry of $M_{ghi}$ equals the $(\mathbf{y},\mathbf{v})$-entry of $E_4^*A_gE_h^*A_iE_4^*$ and the $(\mathbf{y},\mathbf{v})$-entry of $M_{ghi}$ is the unique nonzero entry in the $\mathbf{y}$-row of $M_{ghi}$, note that
$\mathbf{v}_{123}=\mathbf{v}_{231}=\mathbf{v}_{312}=\mathbf{v}$. As the $(\mathbf{y}, \mathbf{v}_{ghi})$-entry of $M_{ghi}$ is one, notice that the $\mathbf{y}$-rows of $M_{123}$, $M_{231}$, $M_{312}$ are pairwise identical as the row vectors of matrices. So $M_{123}=M_{231}=M_{312}$ as $\mathbf{y}$ is chosen from $\mathbf{x}R_4$ arbitrarily and Lemma \ref{L;Tripleproducts} (vi) holds. So (iii) and (iv) are from the equality $M_{123}=M_{231}=M_{312}$.

(i), (ii), (v) are proved if $M_{132}+M_{123}=M_{213}+M_{123}=M_{321}+M_{123}=E_4^*A_4E_4^*$. As $M_{132}=M_{213}=M_{321}$ by the equality $M_{123}=M_{231}=M_{312}$, taking transposes of matrices, and \eqref{Eq;1}, (i), (ii), (v) can be proved if the equality $M_{132}+M_{123}=E_4^*A_4E_4^*$ is checked. Notice that $\mathbf{v}_{132}\neq \mathbf{v}_{123}$ as $\mathbf{v}_{132}, \mathbf{v}_{123}\in \mathbf{y}R_4$, the $(\mathbf{y}, \mathbf{v}_{132})$-entry of $M_{132}$ is nonzero, the $(\mathbf{y}, \mathbf{v}_{123})$-entry of $M_{123}$ is nonzero, the $(\mathbf{y},\mathbf{z})$-entry of $M_{ghi}$ is equal to the $(\mathbf{y},\mathbf{z})$-entry of $E_4^*A_gE_h^*A_iE_4^*$ if $\mathbf{z}\in \mathbf{y}R_4$, and Lemma \ref{L;Lemma1.14} holds. Recall that the $(\mathbf{y},\mathbf{v}_{132})$-entry of $M_{132}$ is the unique nonzero entry in the $\mathbf{y}$-row of $M_{132}$, the $(\mathbf{y},\mathbf{v}_{123})$-entry of $M_{123}$ is the unique nonzero entry in the $\mathbf{y}$-row of $M_{123}$, and these nonzero entries equal one. Notice that the $(\mathbf{y},\mathbf{v}_{132})$-entry and the $(\mathbf{y},\mathbf{v}_{123})$-entry are exactly the nonzero entries in the $\mathbf{y}$-row of $M_{132}+M_{123}$. Moreover, these nonzero entries in the $\mathbf{y}$-row of $M_{132}+M_{123}$ equal one. Notice that $E_4^*A_4E_4^*$ is a $\{0,1\}$-matrix and the number of all nonzero entries in the $\mathbf{y}$-row of $E_4^*A_4E_4^*$ is $|\mathbf{x}R_4\cap \mathbf{y}R_4|$. As $n=4$, $|\mathbf{x}R_4\cap \mathbf{y}R_4|=2$ by Theorem \ref{T;Insectionnumbers} (iii). As $\{\mathbf{v}_{132}, \mathbf{v}_{123}\}=\mathbf{x}R_4\cap \mathbf{y}R_4$, notice that the $(\mathbf{y},\mathbf{v}_{132})$-entry and the $(\mathbf{y}, \mathbf{v}_{123})$-entry are exactly the nonzero entries in the $\mathbf{y}$-row of $E_4^*A_4E_4^*$. So the $\mathbf{y}$-row of $M_{132}+M_{123}$ equals the $\mathbf{y}$-row of $E_4^*A_4E_4^*$ as the row vectors of matrices. Therefore $M_{132}+M_{123}=E_4^*A_4E_4^*$ as $\mathbf{y}$ is chosen from $\mathbf{x}R_4$ arbitrarily and Lemma \ref{L;Tripleproducts} (vi) holds. (i), (ii), (v) thus follow. We are done.
\end{proof}
We end this section by presenting the main result of this section.%The following notation shall be repeatedly utilized in the following sections.
\begin{nota}\label{N;Notation1.21}
Set $B_1=\{E_a^*A_bE_c^*: a, c\in [1,4], b\in [0,3], b\neq \max\{a, c\}, p_{cb}^a\neq 0\}$. Put $B_2=\{E_a^*JE_b^*: a, b\in [0,4]\}$. Set $B_3=\{E_4^*A_aE_4^*: a\in [0,3]\}\cup\{E_4^*JE_4^*\}$ and $B_4=\{E_a^*JE_b^*: a, b\in [0,4], p\mid k_ak_b\}$. Notice that $B_3\subseteq B_1\cup B_2$ by Lemma \ref{L;Tripleproducts} (iii) and Theorem \ref{T;Insectionnumbers} (iii). Use $B_5$ to denote the set containing precisely $E_1^*A_2E_3^*A_1E_4^*$, $E_2^*A_1E_3^*A_2E_4^*$, $E_3^*A_1E_2^*A_3E_4^*$, $E_4^*A_1E_3^*A_2E_1^*$, $E_4^*A_2E_3^*A_1E_2^*$, $E_4^*A_3E_2^*A_1E_3^*$. Write $B_6$ for $\{ E_4^*A_aE_b^*A_cE_4^*: \{a,b,c\}=[1,3]\}$. If $n>8$ or $p\neq 2$ and $n=8$, let $B$ denote the set $B_1\cup B_2\cup B_5\cup B_6$. If $p=2$ and $n=8$, set $B=(B_1\cup B_2\cup B_5\cup B_6)\setminus\{E_4^*A_1E_2^*A_3E_4^*\}$. If $n=4$, put $B=B_1\cup B_2\cup \{E_4^*A_1E_2^*A_3E_4^*\}$. Let $T_1=\langle B_1\cup B_2\cup B_5\cup B_6\rangle_\F\subseteq T$.
\end{nota}
\begin{cor}\label{C;Corollary1.22}
If $G$ is an elementary abelian $2$-group, then $B$ is an $\F$-basis of $T_1$.
\end{cor}
\begin{proof}
By \eqref{Eq;3} and Lemma \ref{L;Tripleproducts} (iii), note that each element in $B_1\cup B_2$ is an $\F$-linear combination of the elements in $\{E_a^*A_bE_c^*: a, b,c \in [0,4], p_{cb}^a\neq 0\}$. Furthermore, note that each element in $\{E_a^*A_bE_c^*: a, b,c \in [0,4], p_{cb}^a\neq 0\}$ is an $\F$-linear combination of the elements in $B_1\cup B_2$. As $|B_1\cup B_2|=|\{E_a^*A_bE_c^*: a, b,c \in [0,4], p_{cb}^a\neq 0\}|$, notice that $B_1\cup B_2$ is an $\F$-linearly independent subset of $T$ by Lemma \ref{L;Tripleproducts} (iv).

As $G$ is an elementary abelian $2$-group, notice that $T_1=\langle B\rangle_\F$ by combining the definitions of $B$, $T_1$, Lemmas \ref{L;Lemma1.19}, \ref{L;Lemma1.20}, \ref{L;Lemma1.5} (i), (ii), \eqref{Eq;3}, and Lemma \ref{L;Tripleproducts} (iii). So it suffices to check that $B$ is an $\F$-linearly independent subset of $T$. If $n>8$ or $p\neq 2$ and $n=8$, notice that $B$ is an $\F$-linearly independent subset of $T$ by combining \eqref{Eq;2}, Lemmas \ref{L;Lemma1.4}, \ref{L;Lemma1.17}, \ref{L;Lemma1.18}, and the fact that $B_1\cup B_2$ is an $\F$-linearly independent subset of $T$. If $p=2$ and $n=8$, notice that $B$ is an $\F$-linearly independent subset of $T$ by combining \eqref{Eq;2}, Lemmas \ref{L;Lemma1.4}, \ref{L;Lemma1.17}, and the fact that $B_1\cup B_2$ is an $\F$-linearly independent subset of $T$. If $n=4$, notice that $B$ is an $\F$-linearly independent subset of $T$ by combining \eqref{Eq;2}, Lemma \ref{L;Lemma1.9}, and the fact that $B_1\cup B_2$ is an $\F$-linearly independent subset of $T$. The desired corollary thus follows.
\end{proof}
\section{Some computational results: Part II}
In this section, we recall Notation \ref{N;Notation1.21} and list some computational results for the elements in $B$. For the case that $G$ is an elementary abelian $2$-group, we shall apply these computational results to deduce that $T$ has an $\F$-basis $B$. These computational results will also be used in the following sections. We first investigate the products of elements with at most one $E_4^*$.
\begin{lem}\label{L;Lemma2.1}
Assume that $\{g, h, i\}=[1,3]$. Then $E_g^*A_hE_i^*A_hE_g^*=E_g^*$.
\end{lem}
\begin{proof}
Pick $\mathbf{y}, \mathbf{z}\in \mathbf{x}R_g$. The $(\mathbf{y}, \mathbf{z})$-entry of $E_g^*A_hE_i^*A_hE_g^*$ is $\overline{|\mathbf{y}R_h\cap\mathbf{x}R_i\cap \mathbf{z}R_h|}$ by Lemma \ref{L;Tripleproducts} (vi). Note that $\mathbf{z}\in \mathbf{y}R_g\cup \mathbf{y}R_0$ by Theorem \ref{T;Insectionnumbers} (i), (ii), and (iii). Assume that $\mathbf{z}\in \mathbf{y}R_g$.
Then $|\mathbf{y}R_h\cap\mathbf{x}R_i\cap \mathbf{z}R_h|\leq |\mathbf{y}R_h\cap \mathbf{z}R_h|=0$ by Theorem \ref{T;Insectionnumbers} (ii). So the $(\mathbf{y},\mathbf{z})$-entry of $E_g^*A_hE_i^*A_hE_g^*$ is zero. Assume that $\mathbf{z}\in \mathbf{y}R_0$. As $|\mathbf{y}R_h\cap \mathbf{x}R_i|=1$ by Theorem \ref{T;Insectionnumbers} (ii), notice that $|\mathbf{y}R_h\cap\mathbf{x}R_i\cap \mathbf{z}R_h|=|\mathbf{y}R_h\cap\mathbf{x}R_i|=1$. Hence the $(\mathbf{y}, \mathbf{z})$-entry of $E_g^*A_hE_i^*A_hE_g^*$ is one. As $\mathbf{y}, \mathbf{z}$ are chosen from $\mathbf{x}R_g$ arbitrarily, notice that $E_g^*A_hE_i^*A_hE_g^*=E_g^*$ by Lemma \ref{L;Tripleproducts} (vi). The desired lemma is proved.
\end{proof}
\begin{lem}\label{L;Lemma2.2}
Assume that $G$ is an elementary abelian $2$-group and $\{g ,h,i\}=[1,3]$. Then $E_g^*A_hE_i^*A_gE_h^*=E_g^*A_iE_h^*$.
\end{lem}
\begin{proof}
Assume that $\mathbf{x}=(x_1, x_2, x_3)$. Let $\mathbf{y}\in \mathbf{x}R_g$ and $\mathbf{z}\in \mathbf{x}R_h$, where $\mathbf{y}\!=\!(y_1, y_2, y_3)$ and $\mathbf{z}=(z_1, z_2, z_3)$.
Then the $(\mathbf{y}, \mathbf{z})$-entry of $E_g^*A_hE_i^*A_gE_h^*$ is $\overline{|\mathbf{y}R_h\cap\mathbf{x}R_i\cap \mathbf{z}R_g|}$ by Lemma \ref{L;Tripleproducts} (vi). As $\mathbf{y}\in \mathbf{x}R_g$ and $\mathbf{z}\in\mathbf{x}R_h$, notice that $\mathbf{z}\in \mathbf{y}R_4\cup \mathbf{y}R_i$ by Theorem \ref{T;Insectionnumbers} (i), (ii), (iii). Moreover, observe that $x_g=y_g$, $x_h\neq y_h$, and $x_h=z_h$.

Assume that $\mathbf{z}\in \mathbf{y}R_4$. Therefore $y_j\neq z_j$ for any $j\in [1,3]$. Assume further that $\mathbf{y}R_h\cap\mathbf{x}R_i\cap \mathbf{z}R_g\neq \varnothing$. Pick $\mathbf{u}=(u_1, u_2, u_3)\in \mathbf{y}R_h\cap\mathbf{x}R_i\cap \mathbf{z}R_g$ and observe that $u_h=y_h, u_i=x_i, u_g=z_g$. Since $G$ is an elementary abelian $2$-group and Lemma \ref{L;Lemma1.2} holds, observe that $x_gy_i=y_gy_i=y_h=u_h=u_iu_g=x_iz_g=x_iz_hz_i=x_ix_hz_i=x_gz_i$. Therefore $y_i=z_i$, which contradicts the fact that $y_j\neq z_j$ for any $j\in [1,3]$. Therefore $\mathbf{y}R_h\cap\mathbf{x}R_i\cap \mathbf{z}R_g=\varnothing$, which implies that $(\mathbf{y}, \mathbf{z})$-entry of $E_g^*A_hE_i^*A_gE_h^*$ is zero.

Assume that $\mathbf{z}\in \mathbf{y}R_i$. Hence $y_i=z_i$ and $\mathbf{y}\neq \mathbf{z}$. Since $G$ is an elementary abelian $2$-group and Lemma \ref{L;Lemma1.2} holds, note that
$y_hx_i=y_hx_gx_h=y_hy_gz_h=y_iz_h=z_iz_h=z_g$. Define $v_g=z_g$, $v_h=y_h$, and $v_i=x_i$. Since $G$ is an elementary abelian $2$-group and $y_hx_i=z_g$, notice that $\mathbf{v}=(v_1, v_2, v_3)\in X_G$ by the definition of $X_G$ and Lemma \ref{L;Lemma1.2}. As $y_i=z_i$ and $\mathbf{y}\neq \mathbf{z}$, notice that $y_g\neq z_g$ and $y_h\neq z_h$ by the definition of $X_G$. As $v_g=z_g\neq y_g$ and $v_h=y_h$, notice that $\mathbf{v}\in \mathbf{y}R_h$. As $v_h=y_h\neq x_h$ and $v_i=x_i$, notice that $\mathbf{v}\in \mathbf{x}R_i$. As $v_h=y_h\neq z_h$ and $v_g=z_g$, notice that $\mathbf{v}\in \mathbf{y}R_h\cap\mathbf{x}R_i\cap\mathbf{z}R_g$. Since $|\mathbf{y}R_h\cap\mathbf{x}R_i\cap \mathbf{z}R_g|\leq |\mathbf{y}R_h\cap\mathbf{z}R_g|=1$ by Theorem \ref{T;Insectionnumbers} (iii), $|\mathbf{y}R_h\cap\mathbf{x}R_i\cap \mathbf{z}R_g|=1$. Hence the $(\mathbf{y},\mathbf{z})$-entry of
$E_g^*A_hE_i^*A_gE_h^*$ is one. Since $\mathbf{y}, \mathbf{z}$ are chosen from $\mathbf{x}R_g$, $\mathbf{x}R_h$ arbitrarily, notice that $E_g^*A_hE_i^*A_gE_h^*=E_g^*A_iE_h^*$ by Lemma \ref{L;Tripleproducts} (vi) and the above discussion. The desired lemma is proved.
\end{proof}
\begin{lem}\label{L;Lemma2.3}
Assume that $\{g, h, i\}=[1,3]$.
\begin{enumerate}[(i)]
\item [\em (i)] $E_g^*A_hE_i^*A_hE_4^*=E_g^*A_hE_4^*$.
\item [\em (ii)] $E_4^*A_hE_i^*A_hE_g^*=E_4^*A_hE_g^*$.
\end{enumerate}
\end{lem}
\begin{proof}
Pick $\mathbf{y}=(y_1, y_2, y_3)\in\mathbf{x}R_g$ and $\mathbf{z}=(z_1, z_2, z_3)\in\mathbf{x}R_4$. According to Lemma \ref{L;Tripleproducts} (vi), the $(\mathbf{y}, \mathbf{z})$-entry of
$E_g^*A_hE_i^*A_hE_4^*$ equals $\overline{|\mathbf{y}R_h\cap\mathbf{x}R_i\cap \mathbf{z}R_h|}$. Since $\mathbf{y}\in \mathbf{x}R_g$ and $\mathbf{z}\in \mathbf{x}R_4$, observe that $\mathbf{z}\in \mathbf{y}R_h\cup \mathbf{y}R_i\cup \mathbf{y}R_4$ by Theorem \ref{T;Insectionnumbers} (i), (ii), (iii).

Assume that $\mathbf{z}\in \mathbf{y}R_i\cup\mathbf{y}R_4$. Notice that $|\mathbf{y}R_h\cap\mathbf{x}R_i\cap \mathbf{z}R_h|\leq|\mathbf{y}R_h\cap\mathbf{z}R_h|=0$ by Theorem \ref{T;Insectionnumbers} (ii) and (iii). Hence the $(\mathbf{y},\mathbf{z})$-entry of $E_g^*A_hE_i^*A_hE_4^*$ is zero. Assume that $\mathbf{z}\in \mathbf{y}R_h$. So $y_h=z_h$.
By Theorem \ref{T;Insectionnumbers} (ii), notice that $|\mathbf{y}R_h\cap\mathbf{x}R_i|=1$. Pick $\mathbf{u}=(u_1, u_2, u_3)\in \mathbf{y}R_h\cap\mathbf{x}R_i$.
So $u_h=y_h=z_h$. If $\mathbf{u}=\mathbf{z}$, then $\mathbf{u}\in \mathbf{x}R_4\cap\mathbf{x}R_i$. This is an obvious contradiction as $i\neq 4$. Therefore $\mathbf{u}\neq \mathbf{z}$ and $\mathbf{u}\in \mathbf{y}R_h\cap\mathbf{x}R_i\cap \mathbf{z}R_h$. So $\mathbf{y}R_h\cap\mathbf{x}R_i\cap\mathbf{z}R_h=\mathbf{y}R_h\cap\mathbf{x}R_i$ and $|\mathbf{y}R_h\cap\mathbf{x}R_i\cap \mathbf{z}R_h|=|\mathbf{y}R_h\cap\mathbf{x}R_i|=1$. So the $(\mathbf{y}, \mathbf{z})$-entry of $E_g^*A_hE_i^*A_hE_4^*$ is one.
As $\mathbf{y}, \mathbf{z}$ are chosen from $\mathbf{x}R_g$, $\mathbf{x}R_h$ arbitrarily, $E_g^*A_hE_i^*A_hE_4^*=E_g^*A_hE_4^*$ by Lemma \ref{L;Tripleproducts} (vi) and the above discussion. (i) is proved. (ii) is from combining (i), taking transposes of matrices, and \eqref{Eq;1}. We are done. %The desired lemma is proved.
\end{proof}
\begin{lem}\label{L;Lemma2.4}
Assume that $G$ is an elementary abelian $2$-group and $\{g, h,i\}=[1,3]$.
\begin{enumerate}[(i)]
\item [\em (i)] $E_g^*A_hE_i^*A_gE_4^*=E_g^*A_iE_h^*A_gE_4^*$.
\item [\em (ii)]$E_4^*A_gE_i^*A_hE_g^*=E_4^*A_gE_h^*A_iE_g^*$.
\end{enumerate}
\end{lem}
\begin{proof}
Assume that $\mathbf{x}=(x_1, x_2, x_3)$. Let $\mathbf{y}\in \mathbf{x}R_g$ and $\mathbf{z}\in \mathbf{x}R_4$, where $\mathbf{y}\!=\!(y_1, y_2, y_3)$ and $\mathbf{z}=(z_1, z_2, z_3)$.
According to Lemma \ref{L;Tripleproducts} (vi), the $(\mathbf{y}, \mathbf{z})$-entries of $E_g^*A_hE_i^*A_gE_4^*$ and $E_g^*A_iE_h^*A_gE_4^*$ are $\overline{|\mathbf{y}R_h\cap\mathbf{x}R_i\cap \mathbf{z}R_g|}$ and $\overline{|\mathbf{y}R_i\cap\mathbf{x}R_h\cap \mathbf{z}R_g|}$, respectively. Since $\mathbf{y}\in\mathbf{x}R_g$, notice that $|\mathbf{y}R_h\cap\mathbf{x}R_i\cap \mathbf{z}R_g|\leq|\mathbf{y}R_h\cap\mathbf{x}R_i|=1$ by Theorem \ref{T;Insectionnumbers} (ii). Notice that $|\mathbf{y}R_i\cap\mathbf{x}R_h\cap \mathbf{z}R_g|\leq |\mathbf{y}R_i\cap\mathbf{x}R_h|=1$ by Theorem \ref{T;Insectionnumbers} (ii). Hence the $(\mathbf{y},\mathbf{z})$-entry of  $E_g^*A_hE_i^*A_gE_4^*$ equals the $(\mathbf{y},\mathbf{z})$-entry of $E_g^*A_iE_h^*A_gE_4^*$ if we can prove that $|\mathbf{y}R_h\cap\mathbf{x}R_i\cap \mathbf{z}R_g|\neq 0$ if and only if $|\mathbf{y}R_i\cap\mathbf{x}R_h\cap \mathbf{z}R_g|\neq 0$. As $\mathbf{y}\in \mathbf{x}R_g$, note that $x_g=y_g$ and $\mathbf{x}\neq \mathbf{y}$. Therefore $x_h\neq y_h$ and $x_i\neq y_i$.

Assume that $|\mathbf{y}R_h\cap\mathbf{x}R_i\cap\mathbf{z}R_g|\neq 0$. Pick $\mathbf{u}=(u_1, u_2, u_3)\in\mathbf{y}R_h\cap\mathbf{x}R_i\cap\mathbf{z}R_g$ and note that $u_h=y_h$, $u_i=x_i$, and $u_g=z_g$. As $G$ is an elementary abelian $2$-group and Lemma \ref{L;Lemma1.2} holds, notice that $x_i=x_gx_h$, $y_h=y_gy_i$, and $z_g=u_g=u_hu_i=y_hx_i$. Set $v_g=z_g$, $v_h=x_h$, $v_i=y_i$. Notice that $v_hv_i=x_hy_i=y_gx_gx_hy_i=y_hx_i=z_g=v_g$. So $\mathbf{v}=(v_1, v_2, v_3)\in X_G$ by Lemma \ref{L;Lemma1.2}. Since $v_i=y_i$ and $v_h=x_h\neq y_h$, $\mathbf{v}\in \mathbf{y}R_i$. As $v_h=x_h$ and $v_i=y_i\neq x_i$, $\mathbf{v}\in \mathbf{y}R_i\cap\mathbf{x}R_h\cap\mathbf{z}R_g$. Hence $|\mathbf{y}R_i\cap\mathbf{x}R_h\cap\mathbf{z}R_g|\neq 0$.

Assume that $|\mathbf{y}R_i\cap\mathbf{x}R_h\cap\mathbf{z}R_g|\neq 0$. One may mimic the proof in the second paragraph to check that $|\mathbf{y}R_h\cap\mathbf{x}R_i\cap\mathbf{z}R_g|\neq 0$. So $|\mathbf{y}R_h\cap\mathbf{x}R_i\cap\mathbf{z}R_g|\neq 0$ if and only if $|\mathbf{y}R_i\cap \mathbf{x}R_h\cap \mathbf{z}R_g|\neq 0$. As $\mathbf{y}, \mathbf{z}$ are chosen from $\mathbf{x}R_g, \mathbf{x}R_4$ arbitrarily, notice that $E_g^*A_hE_i^*A_gE_4^*=E_g^*A_iE_h^*A_gE_4^*$ by Lemma \ref{L;Tripleproducts} (vi). (i) is proved. (ii) is proved by combining (i), taking transposes of matrices, and \eqref{Eq;1}. We are done.
\end{proof}
\begin{lem}\label{L;Lemma2.5}
Assume that $G$ is an elementary abelian $2$-group and $\{g, h,i\}=[1,3]$.
\begin{enumerate}[(i)]
\item [\em (i)] $E_g^*A_hE_4^*A_gE_h^*=E_g^*JE_h^*-E_g^*A_iE_h^*$.
\item [\em (ii)] $E_g^*A_hE_4^*A_gE_i^*=E_g^*JE_i^*-E_g^*A_hE_i^*$.
\item [\em (iii)] $E_g^*A_h E_4^* A_h E_g^*=\overline{n-2}E_g^*$.
\item [\em (iv)] $E_g^*A_h E_4^*A_h E_i^*=\overline{n-2}E_g^*A_hE_i^*$.
\item [\em (v)] $E_g^*A_h E_4^*A_i E_g^*=E_g^*JE_g^*-E_g^*$.
\item [\em (vi)] $E_g^*A_h E_4^*A_i E_h^*=E_g^*JE_h^*-E_g^*A_iE_h^*$.
\end{enumerate}
\end{lem}
\begin{proof}
All listed equalities of the desired lemma are proved by combining the equality $E_g^*A_hE_4^*=E_g^*A_h-E_g^*A_hE_i^*$, \eqref{Eq;2}, \eqref{Eq;3}, Theorem \ref{T;Insectionnumbers} (i), (ii), (iii), Lemmas \ref{L;Tripleproducts} (iii), \ref{L;Lemma2.1}, and \ref{L;Lemma2.2}. The desired lemma is proved.
\end{proof}
\begin{lem}\label{L;Lemma2.6}
Assume that $G$ is an elementary abelian $2$-group and $\{g, h, i\}=[1,3]$.
\begin{enumerate}[(i)]
\item [\em (i)] $E_i^*A_hE_g^*A_hE_i^*A_gE_4^*=E_i^*A_gE_4^*$.
\item [\em (ii)] $E_4^*A_gE_i^*A_hE_g^*A_hE_i^*=E_4^*A_gE_i^*$.
\item [\em (iii)] $E_h^*A_iE_g^*A_hE_i^*A_gE_4^*=E_h^*A_gE_4^*$.
\item [\em (iv)] $E_4^*A_gE_i^*A_hE_g^*A_iE_h^*=E_4^*A_gE_h^*$.
\end{enumerate}
\end{lem}
\begin{proof}
By \eqref{Eq;1}, (ii) and (iv) are the transpose versions of (i) and (iii), respectively. Hence it suffices to check (i) and (iii). (i) follows from Lemma \ref{L;Lemma2.1}. (iii) follows from Lemmas \ref{L;Lemma2.2} and \ref{L;Lemma2.3} (i). The desired lemma is proved.
\end{proof}
\begin{lem}\label{L;Lemma2.7}
Assume that $G$ is an elementary abelian $2$-group and $\{g, h,i\}=[1,3]$.
\begin{enumerate}[(i)]
\item [\em (i)] $E_g^*A_hE_i^*A_gE_4^*A_gE_h^*=\overline{n-2}E_g^*A_iE_h^*$.
\item [\em (ii)] $E_h^*A_gE_4^*A_gE_i^*A_hE_g^*=\overline{n-2}E_h^*A_iE_g^*$.
\item [\em (iii)] $E_g^*A_hE_i^*A_gE_4^*A_gE_i^*=\overline{n-2}E_g^*A_hE_i^*$.
\item [\em (iv)] $E_i^*A_gE_4^*A_gE_i^*A_hE_g^*=\overline{n-2}E_i^*A_hE_g^*$.
\item [\em (v)] $E_g^*A_hE_i^*A_gE_4^*A_hE_g^*=E_g^*JE_g^*-E_g^*$.
\item [\em (vi)] $E_g^*A_hE_4^*A_gE_i^*A_hE_g^*=E_g^*JE_g^*-E_g^*$.
\item [\em (vii)] $E_g^*A_hE_i^*A_gE_4^*A_hE_i^*=E_g^*JE_i^*-E_g^*A_hE_i^*$.
\item [\em (viii)] $E_i^*A_hE_4^*A_gE_i^*A_hE_g^*=E_i^*JE_g^*-E_i^*A_hE_g^*$.
\item [\em (ix)] $E_g^*A_hE_i^*A_gE_4^*A_iE_g^*=E_g^*JE_g^*-E_g^*$.
\item [\em (x)] $E_g^*A_iE_4^*A_gE_i^*A_hE_g^*=E_g^*JE_g^*-E_g^*$.
\item [\em (xi)] $E_g^*A_hE_i^*A_gE_4^*A_iE_h^*=E_g^*JE_h^*-E_g^*A_iE_h^*$.
\item [\em (xii)] $E_h^*A_iE_4^*A_gE_i^*A_hE_g^*=E_h^*JE_g^*-E_h^*A_iE_g^*$.
\end{enumerate}
\end{lem}
\begin{proof}
By \eqref{Eq;1}, (ii), (iv), (vi), (viii), (x), (xii) are the transpose versions of (i), (iii), (v), (vii), (ix), (xi), respectively. Hence it suffices to check (i), (iii), (v), (vii), (ix), (xi). (i) is from Lemmas \ref{L;Lemma2.5} (iv) and \ref{L;Lemma2.2}. (iii) is from Lemma \ref{L;Lemma2.5} (iii). (v) is proved by combining Lemmas \ref{L;Lemma2.5} (vi), \ref{L;Tripleproducts} (i), \ref{L;Lemma2.1}, and Theorem \ref{T;Insectionnumbers} (ii). (vii) is proved by combining Lemmas \ref{L;Lemma2.5} (v), \ref{L;Tripleproducts} (i), and Theorem \ref{T;Insectionnumbers} (ii). (ix) is proved by combining Lemmas \ref{L;Lemma2.5} (i), \ref{L;Tripleproducts} (i), \ref{L;Lemma2.1}, and Theorem \ref{T;Insectionnumbers} (ii). (xi) is proved by combining Lemmas \ref{L;Lemma2.5} (ii), \ref{L;Tripleproducts} (i), \ref{L;Lemma2.2}, and Theorem \ref{T;Insectionnumbers} (ii). We are done.
\end{proof}
The following computational results study the products of elements with two $E_4^*$'s.
\begin{lem}\label{L;Lemma2.8}
Assume that $\{g, h, i\}=[1,3]$.
\begin{enumerate}[(i)]
\item [\em (i)] $E_g^*A_hE_4^*A_gE_4^*=E_g^*JE_4^*-E_g^*A_hE_4^*-E_g^*A_hE_i^*A_gE_4^*$.
\item [\em (ii)] $E_4^*A_gE_4^*A_hE_g^*=E_4^*JE_g^*-E_4^*A_hE_g^*-E_4^*A_gE_i^*A_hE_g^*$.
\item [\em (iii)] $E_g^*A_hE_4^*A_hE_4^*=\overline{n-3}E_g^*A_hE_4^*$.
\item [\em (iv)] $E_4^*A_hE_4^*A_hE_g^*=\overline{n-3}E_4^*A_hE_g^*$.
\item [\em (v)] $E_g^*A_hE_4^*A_iE_4^*=E_g^*JE_4^*-E_g^*A_hE_4^*-E_g^*A_iE_4^*$.
\item [\em (vi)] $E_4^*A_iE_4^*A_hE_g^*=E_4^*JE_g^*-E_4^*A_hE_g^*-E_4^*A_iE_g^*$.
\end{enumerate}
\end{lem}
\begin{proof}
By \eqref{Eq;1}, (ii), (iv), (vi) are the transpose versions of (i), (iii), (v), respectively. Hence it suffices to check (i), (iii), and (v). (i) is proved by combining the equality $E_g^*A_hE_4^*=E_g^*A_h-E_g^*A_hE_i^*$, \eqref{Eq;2}, \eqref{Eq;3}, Theorem \ref{T;Insectionnumbers} (i), (ii), (iii), and Lemma \ref{L;Tripleproducts} (iii). (iii) is proved by combining the equality $E_g^*A_hE_4^*=E_g^*A_h-E_g^*A_hE_i^*$, \eqref{Eq;2}, Theorem \ref{T;Insectionnumbers} (i), (ii), (iii), Lemmas \ref{L;Tripleproducts} (iii), and \ref{L;Lemma2.3} (i). (v) is proved by combining the equality $E_g^*A_hE_4^*=E_g^*A_h-E_g^*A_hE_i^*$, \eqref{Eq;2}, \eqref{Eq;3}, Theorem \ref{T;Insectionnumbers} (i), (ii), (iii), and Lemma \ref{L;Tripleproducts} (iii). The desired lemma is proved.
\end{proof}
\begin{lem}\label{L;Lemma2.9}
If $g, h\in [1,3]$ and $g\neq h$, then $E_4^*A_gE_h^*A_gE_4^*=E_4^*+E_4^*A_gE_4^*$.
\end{lem}
\begin{proof}
Assume that $\mathbf{x}=(x_1, x_2, x_3)$, $\mathbf{y}, \mathbf{z}\in \mathbf{x}R_4$, $\mathbf{y}\!=\!(y_1, y_2, y_3)$, and $\mathbf{z}=(z_1, z_2, z_3)$. According to Lemma \ref{L;Tripleproducts} (vi), the $(\mathbf{y}, \mathbf{z})$-entry of $E_4^*A_gE_h^*A_gE_4^*$ is $\overline{|\mathbf{y}R_g\cap\mathbf{x}R_h\cap \mathbf{z}R_g|}$. Assume that $\mathbf{z}\in \mathbf{y}R_h\cup\mathbf{y}R_i\cup\mathbf{y}R_4$. Notice that $|\mathbf{y}R_g\cap\mathbf{x}R_h\cap\mathbf{z}R_g|\leq |\mathbf{y}R_g\cap\mathbf{z}R_g|=0$ by Theorem \ref{T;Insectionnumbers} (ii) and (iii). Hence the $(\mathbf{y},\mathbf{z})$-entry of $E_4^*A_gE_h^*A_gE_4^*$ equals zero. As $\mathbf{y}, \mathbf{z}\in \mathbf{x}R_4$, notice that $x_j\neq z_j$ for any $j\in [1,3]$.

Assume that $(\mathbf{y},\mathbf{z})\in R_0\cup R_g$. Hence $y_g=z_g$. As Theorem \ref{T;Insectionnumbers} (iii) holds, notice that $|\mathbf{y}R_g\cap\mathbf{x}R_h|=1$. Pick $\mathbf{v}=(v_1, v_2, v_3)\in \mathbf{y}R_g\cap\mathbf{x}R_h$. Therefore $v_g=y_g=z_g$ and $v_h=x_h\neq z_h$. Hence $\mathbf{v}\in \mathbf{z}R_g$. Hence $\mathbf{y}R_g\cap\mathbf{x}R_h\cap\mathbf{z}R_g=\mathbf{y}R_g\cap\mathbf{x}R_h=\{\mathbf{v}\}$, which implies that the $(\mathbf{y},\mathbf{z})$-entry of $E_4^*A_gE_h^*A_gE_4^*$ is one.
As $\mathbf{y}, \mathbf{z}$ are chosen from $\mathbf{x}R_4$ arbitrarily, notice that $E_4^*A_gE_h^*A_gE_4^*=E_4^*+E_4^*A_gE_4^*$ by Lemma \ref{L;Tripleproducts} (vi) and the above discussion. The desired lemma is proved.
\end{proof}
\begin{lem}\label{L;Lemma2.10}
Assume that $G$ is an elementary abelian $2$-group and $\{g, h,i\}=[1,3]$.
\begin{enumerate}[(i)]
\item [\em (i)] $E_g^*A_hE_i^*A_gE_4^*A_gE_4^*=\overline{n-3}E_g^*A_hE_i^*A_gE_4^*$.
\item [\em (ii)] $E_4^*A_gE_4^*A_gE_i^*A_hE_g^*=\overline{n-3}E_4^*A_gE_i^*A_hE_g^*$.
\item [\em (iii)] $E_g^*A_hE_i^*A_gE_4^*A_hE_4^*=E_g^*JE_4^*-E_g^*A_hE_4^*-E_g^*A_hE_i^*A_gE_4^*$.
\item [\em (iv)] $E_4^*A_hE_4^*A_gE_i^*A_hE_g^*=E_4^*JE_g^*-E_4^*A_hE_g^*-E_4^*A_gE_i^*A_hE_g^*$.
\item [\em (v)] $E_g^*A_hE_i^*A_gE_4^*A_iE_4^*=E_g^*JE_4^*-E_g^*A_iE_4^*-E_g^*A_hE_i^*A_gE_4^*$.
\item [\em (vi)] $E_4^*A_iE_4^*A_gE_i^*A_hE_g^*=E_4^*JE_g^*-E_4^*A_iE_g^*-E_4^*A_gE_i^*A_hE_g^*$.
\end{enumerate}
\end{lem}
\begin{proof}
By \eqref{Eq;1}, (ii), (iv), (vi) are the transpose versions of (i), (iii), (v), respectively. Hence it suffices to check (i), (iii), and (v). (i) is proved by Lemma \ref{L;Lemma2.8} (iii). (iii) is proved by combining Lemma \ref{L;Lemma2.8} (v), Lemmas \ref{L;Tripleproducts} (i), \ref{L;Lemma2.3} (i), and Theorem \ref{T;Insectionnumbers} (ii). Similarly, (v) is proved by combining Lemma \ref{L;Lemma2.8} (i), Lemmas \ref{L;Tripleproducts} (i), \ref{L;Lemma2.2}, \ref{L;Lemma2.3} (i), and Theorem \ref{T;Insectionnumbers} (ii). The desired lemma is proved.
\end{proof}
\begin{lem}\label{L;Lemma2.11}
Assume that $G$ is an elementary abelian $2$-group and $\{g, h, i\}=[1,3]$.
\begin{enumerate}[(i)]
\item [\em (i)] $E_4^*A_hE_g^*A_hE_i^*A_gE_4^*=E_4^*A_hE_i^*A_gE_4^*$.
\item [\em (ii)] $E_4^*A_gE_i^*A_hE_g^*A_hE_4^*=E_4^*A_gE_i^*A_hE_4^*$.
\item [\em (iii)] $E_4^*A_iE_g^*A_hE_i^*A_gE_4^*=E_4^*A_iE_h^*A_gE_4^*$.
\item [\em (iv)] $E_4^*A_gE_i^*A_hE_g^*A_iE_4^*=E_4^*A_gE_h^*A_iE_4^*$.
\end{enumerate}
\end{lem}
\begin{proof}
By \eqref{Eq;1}, (ii) and (iv) are the transpose versions of (i) and (iii), respectively. Hence it suffices to check (i) and (iii). (i) is proved by Lemma \ref{L;Lemma2.3} (ii). By Lemma \ref{L;Lemma2.4} (ii), $E_4^*A_iE_g^*A_hE_i^*=E_4^*A_iE_h^*A_gE_i^*$. So $E_4^*A_iE_g^*A_hE_i^*A_gE_4^*=E_4^*A_iE_h^*A_gE_i^*A_gE_4^*$ and $E_4^*A_iE_h^*A_gE_i^*A_gE_4^*=E_4^*A_iE_h^*A_gE_4^*$ by Lemma \ref{L;Lemma2.3} (i). Hence (iii) is proved. The desired lemma is proved.
\end{proof}
\begin{lem}\label{L;Lemma2.12}
Assume that $G$ is an elementary abelian $2$-group and $\{g, h, i\}=[1,3]$.
\begin{enumerate}[(i)]
\item [\em (i)] $E_h^*A_gE_4^*A_gE_h^*A_iE_4^*=\overline{n-2}E_h^*A_iE_4^*$.
\item [\em (ii)] $E_4^*A_iE_h^*A_gE_4^*A_gE_h^*=\overline{n-2}E_4^*A_iE_h^*$.
\item [\em (iii)] $E_i^*A_gE_4^*A_gE_h^*A_iE_4^*=\overline{n-2}E_i^*A_gE_h^*A_iE_4^*$.
\item [\em (iv)] $E_4^*A_iE_h^*A_gE_4^*A_gE_i^*=\overline{n-2}E_4^*A_iE_h^*A_gE_i^*$.
\item [\em (v)]$E_g^*A_hE_4^*A_gE_h^*A_iE_4^*=E_g^*JE_4^*-E_g^*A_iE_4^*$.
\item [\em (vi)] $E_4^*A_iE_h^*A_gE_4^*A_hE_g^*=E_4^*JE_g^*-E_4^*A_iE_g^*$.
\item [\em (vii)] $E_i^*A_hE_4^*A_gE_h^*A_iE_4^*=E_i^*JE_4^*-E_i^*A_gE_h^*A_iE_4^*$.
\item [\em (viii)] $E_4^*A_iE_h^*A_gE_4^*A_hE_i^*=E_4^*JE_i^*-E_4^*A_iE_h^*A_gE_i^*$.
\item [\em (ix)] $E_g^*A_iE_4^*A_gE_h^*A_iE_4^*=E_g^*JE_4^*-E_g^*A_iE_4^*$.
\item [\em (x)] $E_4^*A_iE_h^*A_gE_4^*A_iE_g^*=E_4^*JE_g^*-E_4^*A_iE_g^*$.
\item [\em (xi)] $E_h^*A_iE_4^*A_gE_h^*A_iE_4^*=E_h^*JE_4^*-E_h^*A_iE_4^*$.
\item [\em (xii)] $E_4^*A_iE_h^*A_gE_4^*A_iE_h^*=E_4^*JE_h^*-E_4^*A_iE_h^*$.
\end{enumerate}
\end{lem}
\begin{proof}
By \eqref{Eq;1}, (ii), (iv), (vi), (viii), (x), (xii) are the transpose versions of (i), (iii), (v), (vii), (ix), (xi), respectively. Hence it suffices to check (i), (iii), (v), (vii), (ix), (xi). (i) is proved by Lemma \ref{L;Lemma2.5} (iii). (iii) is proved by Lemma \ref{L;Lemma2.5} (iv). (v) can be proved by Lemmas \ref{L;Lemma2.5} (i), \ref{L;Tripleproducts} (i),
\ref{L;Lemma2.3} (i), and Theorem \ref{T;Insectionnumbers} (iii). (vii) is proved by combining Lemmas \ref{L;Lemma2.5} (vi), \ref{L;Tripleproducts} (i), and Theorem \ref{T;Insectionnumbers} (iii). (ix) can be proved by combining Lemmas \ref{L;Lemma2.5} (ii), \ref{L;Tripleproducts} (i), \ref{L;Lemma2.3} (i), and Theorem \ref{T;Insectionnumbers} (iii). (xi) is proved by combining Lemmas \ref{L;Lemma2.5} (v), \ref{L;Tripleproducts} (i), and Theorem \ref{T;Insectionnumbers} (iii). We are done.
\end{proof}
The leaving computational results study the products of elements with three $E_4^*$'s.
\begin{lem}\label{L;Lemma2.13}
Assume that $\{g, h, i\}=[1,3]$.
\begin{enumerate}[(i)]
\item [\em (i)] $E_4^*A_gE_4^*A_gE_4^*=\overline{n-3}E_4^*+\overline{n-4}E_4^*A_gE_4^*$.
\item [\em (ii)] $E_4^*A_gE_4^*A_hE_4^*=E_4^*JE_4^*-E_4^*-E_4^*A_gE_4^*-E_4^*A_hE_4^*-E_4^*A_gE_i^*A_hE_4^*$.
\end{enumerate}
\end{lem}
\begin{proof}
(i) is proved by combining the equality $E_4^*A_gE_4^*=E_4^*A_g-E_4^*A_gE_h^*-E_4^*A_gE_i^*$, \eqref{Eq;2}, Theorem \ref{T;Insectionnumbers} (i), (ii), (iii), and Lemma \ref{L;Lemma2.9}. (ii) can be proved by combining the equality $E_4^*A_gE_4^*=E_4^*A_g-E_4^*A_gE_h^*-E_4^*A_gE_i^*$, \eqref{Eq;2}, \eqref{Eq;3}, Theorem \ref{T;Insectionnumbers} (i), (ii), (iii), and Lemma \ref{L;Tripleproducts} (iii). The desired lemma is proved.
\end{proof}
\begin{lem}\label{L;Lemma2.14}
Assume that $\{g, h, i\}=[1,3]$.
\begin{enumerate}[(i)]
\item [\em (i)] $E_4^*A_gE_4^*A_gE_h^*A_iE_4^*=\overline{n-3}E_4^*A_gE_h^*A_iE_4^*$.
\item [\em (ii)] $E_4^*A_iE_h^*A_gE_4^*A_gE_4^*=\overline{n-3}E_4^*A_iE_h^*A_gE_4^*$.
\item [\em (iii)] $E_4^*A_hE_4^*A_gE_h^*A_iE_4^*=E_4^*JE_4^*-E_4^*A_gE_h^*A_iE_4^*-E_4^*A_hE_g^*A_iE_4^*$.
\item [\em (iv)] $E_4^*A_iE_h^*A_gE_4^*A_hE_4^*=E_4^*JE_4^*-E_4^*A_iE_h^*A_gE_4^*-E_4^*A_iE_g^*A_hE_4^*$.
\item [\em (v)] $E_4^*A_iE_4^*A_gE_h^*A_iE_4^*=E_4^*JE_4^*-E_4^*A_gE_h^*A_iE_4^*-E_4^*-E_4^*A_iE_4^*$.
\item [\em (vi)] $E_4^*A_iE_h^*A_gE_4^*A_iE_4^*=E_4^*JE_4^*-E_4^*A_iE_h^*A_gE_4^*-E_4^*-E_4^*A_iE_4^*$.
\end{enumerate}
\end{lem}
\begin{proof}
By \eqref{Eq;1}, (ii), (iv), (vi) are the transpose versions of (i), (iii), (v), respectively. Hence it suffices to check (i), (iii), and (v). (i) is proved by Lemma \ref{L;Lemma2.8} (iv).
(iii) is proved by combining Lemmas \ref{L;Lemma2.8} (ii), \ref{L;Tripleproducts} (i), Theorem \ref{T;Insectionnumbers} (iii), and Lemma \ref{L;Lemma2.11} (iii). (v) can be proved by combining Lemmas \ref{L;Lemma2.8} (vi), \ref{L;Tripleproducts} (i), Theorem \ref{T;Insectionnumbers} (iii), and Lemma \ref{L;Lemma2.9}. The desired lemma is proved.
\end{proof}
The following lemma shall be repeatedly utilized in the following sections.
\begin{lem}\label{L;Lemma2.15}
Assume that $G$ is an elementary abelian $2$-group and $\{g, h,i\}=[1,3]$.
\begin{enumerate}[(i)]
\item [\em (i)] $E_4^*A_gE_h^*A_iE_4^*A_gE_h^*A_iE_4^*=E_4^*JE_4^*-E_4^*A_gE_h^*A_iE_4^*$.
\item [\em (ii)] $E_4^*A_gE_h^*A_iE_4^*A_gE_i^*A_hE_4^*=E_4^*JE_4^*-E_4^*A_gE_i^*A_hE_4^*$.
\item [\em (iii)] $E_4^*A_gE_h^*A_iE_4^*A_hE_g^*A_iE_4^*=E_4^*JE_4^*-E_4^*A_gE_h^*A_iE_4^*$.
\item [\em (iv)] $E_4^*A_gE_h^*A_iE_4^*A_hE_i^*A_gE_4^*=E_4^*JE_4^*-E_4^*-E_4^*A_gE_4^*$.
\item [\em (v)] $E_4^*A_gE_h^*A_iE_4^*A_iE_g^*A_hE_4^*=\overline{n-2}E_4^*A_gE_i^*A_hE_4^*$.
\item [\em (vi)] $E_4^*A_gE_h^*A_iE_4^*A_iE_h^*A_gE_4^*=\overline{n-2}E_4^*+\overline{n-2}E_4^*A_gE_4^*$.
\end{enumerate}
\end{lem}
\begin{proof}
(i) is proved by combining Lemmas \ref{L;Lemma2.12} (xii), \ref{L;Tripleproducts} (i), and Theorem \ref{T;Insectionnumbers} (iii). (ii) is proved by combining Lemmas \ref{L;Lemma2.12} (x),
\ref{L;Tripleproducts} (i), and Theorem \ref{T;Insectionnumbers} (iii). (iii) is proved by combining Lemmas \ref{L;Lemma2.12} (viii), \ref{L;Tripleproducts} (i), and Theorem \ref{T;Insectionnumbers} (iii). (iv) can be proved by combining Lemmas \ref{L;Lemma2.12} (vi), \ref{L;Tripleproducts} (i), \ref{L;Lemma2.9}, and Theorem \ref{T;Insectionnumbers} (iii). (v) can be proved by Lemmas \ref{L;Lemma2.12} (iv) and \ref{L;Lemma2.11} (iii). (vi) can be proved by Lemmas \ref{L;Lemma2.12} (ii) and \ref{L;Lemma2.9}. The desired lemma is proved.
\end{proof}
We are now ready to deduce the main result of this section.
\begin{cor}\label{C;Corollary2.16}
Assume that $G$ is an elementary abelian $2$-group. If $M\in B_1\cup B_2$ and $N\in B_1\cup B_2\cup B_5\cup B_6$, then $MN\in T_1$. In paricular, $B$ is an $\F$-basis of $T$.
\end{cor}
\begin{proof}
As $T_1=\langle B_1\cup B_2\cup B_5\cup B_6\rangle_\F$, the first statement follows if $MN$ is an $\F$-linear combination of the elements in $B_1\cup B_2\cup B_5\cup B_6$. Assume that one of the elements $M$ and $N$ is in $B_2$. According to \eqref{Eq;2} and Lemma \ref{L;Tripleproducts} (i), observe that $MN$ is an $\F$-linear combination of the elements in $B_2$. Notice that $E_1^*, E_2^*, E_3^*, E_4^*\in B_1$ by \eqref{Eq;2} and \eqref{Eq;3}. Assume that one of the elements $M$ and $N$ is in $\{E_1^*, E_2^*, E_3^*, E_4^*\}$. Then $MN$ is an $\F$-linear combination of the elements in $B_1\cup B_2\cup B_5\cup B_6$. Assume further that both $M$ and $N$ are not in $B_2\cup\{E_1^*, E_2^*, E_3^*, E_4^*\}$. By combining the equalities in Lemmas \ref{L;Lemma2.1}, \ref{L;Lemma2.2}, \ref{L;Lemma2.3}, \ref{L;Lemma2.5}, \ref{L;Lemma2.6}, \ref{L;Lemma2.7}, \ref{L;Lemma2.8}, \ref{L;Lemma2.9}, \ref{L;Lemma2.10}, \ref{L;Lemma2.11}, \ref{L;Lemma2.12}, \ref{L;Lemma2.13}, \ref{L;Lemma2.14}, \eqref{Eq;2}, and the definitions of $B_1$, $B_2$, $B_5$, $B_6$, notice that $MN$ is an $\F$-linear combination of the elements in $B_1\cup B_2\cup B_5\cup B_6$. The first statement is proved.

Notice that $\{E_a^*A_bE_c^*: a,b,c\in [0,4]\}\subseteq T_1$ by combining \eqref{Eq;3}, Lemma \ref{L;Tripleproducts} (iii), and Theorem \ref{T;Insectionnumbers} (i), (ii), (iii). According to the definitions of $B_1, B_2, B_5, B_6$, \eqref{Eq;3}, and the first statement, notice that $T_1$ is a unital $\F$-subalgebra of $T$. Hence $T=T_1$ by the definition of $T$. The second statement thus follows from Corollary \ref{C;Corollary1.22}. The proof of the desired corollary is now complete.
\end{proof}
Our final corollary focuses on $E_4^*TE_4^*$ whose definition is given in Subsection \ref{S;Subsection3}.
\begin{cor}\label{C;Corollary2.17}
Assume that $G$ is an elementary abelian $2$-group. If $n>8$ or $p\neq 2$ and $n=8$, then $B_3\cup B_6$ is an $\F$-basis of the $\F$-subalgebra $E_4^*TE_4^*$ of $T$. If $p=2$ and $n=8$,
then $(B_3\cup B_6)\setminus\{E_4^*A_1E_2^*A_3E_4^*\}$ is an $\F$-basis of the $\F$-subalgebra $E_4^*TE_4^*$ of $T$. If $n=4$, then $B_3\cup\{E_4^*A_1E_2^*A_3E_4^*\}$ is an $\F$-basis of the $\F$-subalgebra $E_4^*TE_4^*$ of $T$.
\end{cor}
\begin{proof}
As $B_3\subseteq B_1\cup B_2$, notice that the desired corollary is proved by combining Corollary \ref{C;Corollary2.16}, the definition of $B$, and \eqref{Eq;2}.
\end{proof}
\section{Algebraic structure of $T$: Case I}
From now on, $G$ is assumed further to be an elementary abelian $2$-group. The aim of this section is to get the algebraic structure of $T$ for the case $n\equiv 1\pmod p$. For our purpose, recall Notation \ref{N;Notation1.21} and the definition of $E_4^*TE_4^*$ in Subsection \ref{S;Subsection3}.

We first introduce the following notation and present some related properties.
\begin{nota}\label{N;Notation3.1}
Let $C_1$ denote the set containing precisely
\begin{align*}
&E_4^*JE_4^*,\  E_4^*A_1E_4^*-E_4^*A_2E_4^*,\ E_4^*A_1E_4^*-E_4^*A_3E_4^*,\\
&E_4^*A_1E_2^*A_3E_4^*-E_4^*-E_4^*A_1E_4^*,\ E_4^*A_1E_2^*A_3E_4^*-E_4^*A_1E_3^*A_2E_4^*,\\
&E_4^*A_1E_2^*A_3E_4^*-E_4^*A_2E_1^*A_3E_4^*,\ E_4^*A_1E_2^*A_3E_4^*-E_4^*A_2E_3^*A_1E_4^*,\\
&E_4^*A_1E_2^*A_3E_4^*-E_4^*A_3E_1^*A_2E_4^*,\ E_4^*A_1E_2^*A_3E_4^*-E_4^*A_3E_2^*A_1E_4^*.
\end{align*}
The set containing exactly $E_1^*A_2E_4^*\!-\!E_1^*A_3E_4^*$, $E_2^*A_1E_4^*\!-\!E_2^*A_3E_4^*$, $E_3^*A_1E_4^*\!-\!E_3^*A_2E_4^*$, $E_1^*A_2E_3^*A_1E_4^*-E_1^*A_3E_4^*$,
$E_2^*A_1E_3^*A_2E_4^*-E_2^*A_3E_4^*$, $E_3^*A_1E_2^*A_3E_4^*\!-\!E_3^*A_2E_4^*$, and their transposes is denoted by $D_1$. Set $U_1=\langle C_1\rangle_\F\subseteq T$. Put $V_1=\langle B_4\cup C_1\cup D_1\rangle_\F\subseteq T$.
\end{nota}
\begin{lem}\label{L;Lemma3.2}
If $n>8$ or $p\neq 2$ and $n=8$, $C_1\cup\{\overline{2}E_4^*+E_4^*A_1E_4^*, -E_4^*A_1E_2^*A_3E_4^*\}$ is an $\F$-basis of the $\F$-subalgebra $E_4^*TE_4^*$ of $T$. In particular, $C_1$ is an $\F$-basis of $U_1$.
\end{lem}
\begin{proof}
Let $C=C_1\cup\{\overline{2}E_4^*+E_4^*A_1E_4^*, -E_4^*A_1E_2^*A_3E_4^*\}$. By \eqref{Eq;2} and \eqref{Eq;3}, notice that each element in $C$ is an $\F$-linear combination of the elements in $B_3\cup B_6$. By \eqref{Eq;2} and \eqref{Eq;3}, notice that each element in $B_3\cup B_6$ is also an $\F$-linear combination of the elements in $C$. As $|C|=|B_3\cup B_6|$ and Corollary \ref{C;Corollary2.17} holds, the first statement follows. As $U_1=\langle C_1\rangle_\F$, the second statement thus follows from the first one. The desired lemma is proved.
\end{proof}
\begin{lem}\label{L;Lemma3.3}
If $n=4$, the set containing precisely $\overline{2}E_4^*+E_4^*A_1E_4^*$, $-E_4^*A_1E_2^*A_3E_4^*$, $E_4^*JE_4^*$, $E_4^*A_1E_4^*-E_4^*A_2E_4^*$, $E_4^*A_1E_4^*-E_4^*A_3E_4^*$, $E_4^*A_1E_2^*A_3E_4^*-E_4^*-E_4^*A_1E_4^*$ is an $\F$-basis of the $\F$-subalgebra $E_4^*TE_4^*$ of $T$. If $p=3$ and $n=4$, the set containing exactly $E_4^*JE_4^*$, $E_4^*A_1E_4^*-E_4^*A_2E_4^*$, $E_4^*A_1E_4^*-E_4^*A_3E_4^*$, $E_4^*A_1E_2^*A_3E_4^*-E_4^*-E_4^*A_1E_4^*$ is an $\F$-basis of $U_1$.
\end{lem}
\begin{proof}
The set containing precisely $E_4^*JE_4^*$, $E_4^*A_1E_4^*-E_4^*A_2E_4^*$, $E_4^*A_1E_4^*-E_4^*A_3E_4^*$, $E_4^*A_1E_2^*A_3E_4^*-E_4^*-E_4^*A_1E_4^*$ is denoted by $C$. According to \eqref{Eq;2} and \eqref{Eq;3}, notice that each element in $C\cup \{\overline{2}E_4^*+E_4^*A_1E_4^*, -E_4^*A_1E_2^*A_3E_4^*\}$ is an $\F$-linear combination of the elements in $B_3\cup\{E_4^*A_1E_2^*A_3E_4^*\}$. According to \eqref{Eq;2} and \ref{Eq;3} again, notice that each element in $B_3\cup \{E_4^*A_1E_2^*A_3E_4^*\}$ is an $\F$-linear combination of the elements in $C\cup \{\overline{2}E_4^*+E_4^*A_1E_4^*, -E_4^*A_1E_2^*A_3E_4^*\}$. The first statement follows as Corollary \ref{C;Corollary2.17} holds and $|C\cup\{\overline{2}E_4^*+E_4^*A_1E_4^*$, $-E_4^*A_1E_2^*A_3E_4^*\}|=|B_3\cup\{E_4^*A_1E_2^*A_3E_4^*\}|$. So $C$ is an $\F$-linearly independent subset of $T$. By combining \eqref{Eq;2}, \eqref{Eq;3}, the equalities in Lemma \ref{L;Lemma1.20}, and the assumption $p=3$, notice that each element in $C_1$ is an $\F$-linear combination of the elements in $C$. As $C\subseteq C_1$, notice that $C$ is an $\F$-basis of $U_1$. The second statement thus follows from the first one. The proof of the desired lemma is now complete. %The desired lemma is proved.
\end{proof}
The following results describe $E_4^*TE_4^*$ for the case $n\equiv 1\pmod p$.
\begin{lem}\label{L;Lemma3.4}
If $n\equiv 1\pmod p$, $U_1$ is a two-sided ideal of the $\F$-subalgebra $E_4^*TE_4^*$ of $T$.
\end{lem}
\begin{proof}
Notice that $U_1\subseteq E_4^*TE_4^*$. If $M\in B_3\cup B_6$, then $M^T\in B_3\cup B_6$ by \eqref{Eq;1}. If $N\in U_1$, notice that $N^T\in U_1$ as $U_1=\langle C_1\rangle_\F$ and \eqref{Eq;1} holds. By Corollary \ref{C;Corollary2.17}, the desired lemma thus follows if $MN\in U_1$ for any $M\in B_3\cup B_6$ and $N\in C_1$. If one of the elements $M$ and $N$ is $E_4^*JE_4^*$, notice that $MN\in U_1$ by \eqref{Eq;2} and Lemma \ref{L;Tripleproducts} (i). Assume that $M\in (B_3\cup B_6)\setminus\{E_4^*JE_4^*\}$ and $N\in C_1\setminus\{E_4^*JE_4^*\}$. By combining \eqref{Eq;2}, \eqref{Eq;3}, the equalities in Lemmas \ref{L;Lemma2.13}, \ref{L;Lemma2.14}, \ref{L;Lemma2.15}, and the assumption $n\equiv 1\pmod p$, notice that $MN\in U_1$. The desired lemma thus follows.
\end{proof}
\begin{lem}\label{L;Lemma3.5}
If $n\equiv 1\pmod p$, the two-sided ideal $U_1$ of the $\F$-subalgebra $E_4^*TE_4^*$ of $T$ is nilpotent. In particular, $U_1\subseteq \mathrm{Rad}E_4^*TE_4^*$.
\end{lem}
\begin{proof}
As $U_1=\langle C_1\rangle_\F$ and Lemma \ref{L;Lemma3.4} holds, the first statement follows if the product of any three elements in $C_1$ is the zero matrix. Pick $M_1, M_2, M_3\in C_1$. Claim that $E_4^*JE_4^*M=ME_4^*JE_4^*=O$ for any $M\in C_1$. It is proved by combining \eqref{Eq;2}, \eqref{Eq;4}, Lemma \ref{L;Tripleproducts} (i), Theorem \ref{T;Insectionnumbers} (i), (ii), (iii), and the assumption $n\equiv 1\pmod p$. So $M_1M_2M_3=O$ if one of the three elements $M_1, M_2, M_3$ is $E_4^*JE_4^*$. Assume that $M_1, M_2, M_3\in C_1\setminus\{E_4^*JE_4^*\}$. Then the equality $M_1M_2M_3=O$ can be proved by combining \eqref{Eq;2}, the equalities in Lemmas \ref{L;Lemma2.13}, \ref{L;Lemma2.14}, \ref{L;Lemma2.15}, the proven claim, and the assumption $n\equiv 1\pmod p$. The first statement thus follows. The second statement follows from the first one. The desired lemma is proved.
\end{proof}
We now determine the algebraic structure of $E_4^*TE_4^*$ for the case $n\equiv 1\pmod p$.
\begin{cor}\label{C;Corollary3.6}
If $n\equiv 1\pmod p$, then $E_4^*TE_4^*/\mathrm{Rad}E_4^*TE_4^*\cong M_1(\F)\oplus M_1(\F)$ as $\F$-algebras and $\mathrm{Rad}E_4^*TE_4^*=U_1$.
\end{cor}
\begin{proof}
By Lemma \ref{L;Lemma3.5}, it is enough to check that $E_4^*TE_4^*/U_1\cong M_1(\F)\oplus M_1(\F)$ as $\F$-algebras. As $n\equiv 1\pmod p$, the inequality $n>4$ yields that $n>8$ or $p\neq 2$ and $n=8$. By Lemmas \ref{L;Lemma3.2} and \ref{L;Lemma3.3}, $\{\overline{2}E_4^*+E_4^*A_1E_4^*+U_1, -E_4^*A_1E_2^*A_3E_4^*+U_1\}$ is an $\F$-basis of $E_4^*TE_4^*/U_1$. Set $M=\overline{2}E_4^*+E_4^*A_1E_4^*+U_1$ and $N=-E_4^*A_1E_2^*A_3E_4^*+U_1$. By combining \eqref{Eq;2}, the equalities in Lemmas \ref{L;Lemma2.13}, \ref{L;Lemma2.14}, \ref{L;Lemma2.15}, and the assumption $n\equiv 1\pmod p$, notice that $M^2=M$, $N^2=N$, and $MN=NM=O+U_1$. Hence $\langle \{M\}\rangle_\F$ and $\langle \{N\}\rangle_\F$ are distinct two-sided ideals of $E_4^*TE_4^*/U_1$. This fact implies that $E_4^*TE_4^*/U_1=\langle \{M\}\rangle_\F\oplus\langle \{N\}\rangle_\F\cong M_1(\F)\oplus M_1(\F)$ as $\F$-algebras. The proof of the desired corollary is now complete.
\end{proof}
We next introduce the following notation and present some related properties.
\begin{nota}\label{N;Notation3.7}
The union of $\{E_a^*A_bE_c^*: a, c\in[1,3], b\in [0,3], b\!\neq\! \max\{a, c\}, p_{cb}^a\neq 0\}$ and $\{E_1^*A_2E_4^*, E_2^*A_1E_4^*, E_3^*A_1E_4^*, -E_4^*A_1E_2^*, -E_4^*A_1E_3^*, -E_4^*A_2E_1^*, -E_4^*A_1E_2^*A_3E_4^*\}$ is denoted by $H_1$. Let $K_1$ be the set containing exactly $E_4^*A_1E_2^*A_3E_4^*-E_4^*-E_4^*A_1E_4^*$, $E_1^*A_2E_4^*-E_1^*A_3E_4^*$, $E_2^*A_1E_4^*-E_2^*A_3E_4^*$, $E_3^*A_1E_4^*-E_3^*A_2E_4^*$, $E_4^*A_1E_2^*-E_4^*A_3E_2^*$, $E_4^*A_1E_3^*-E_4^*A_2E_3^*$, $E_4^*A_1E_4^*-E_4^*A_2E_4^*$, $E_4^*A_1E_4^*-E_4^*A_3E_4^*$, $E_4^*A_2E_1^*-E_4^*A_3E_1^*$, and the elements in $B_4$. Since $k_j=p_{jj}^0$ for any $j\in[0,4]$ and Theorem \ref{T;Insectionnumbers} (i) holds, notice that $B_4=\{E_a^*JE_b^*: a,b\in [0,4]\}\setminus\{E_0^*JE_0^*\}$ for the case $n\equiv 1\pmod p$.
\end{nota}
\begin{lem}\label{L;Lemma3.8}
Assume that $n\equiv 1\pmod p$. If $n>8$ or $p=7$ and $n=8$, then $B_4\cup C_1\cup D_1\cup H_1\cup\{E_0^*, E_4^*+E_4^*A_1E_2^*A_3E_4^*\}$ is an $\F$-basis of $T$. In particular, $V_1$ has an $\F$-basis $B_4\cup C_1\cup D_1$.
\end{lem}
\begin{proof}
By Corollary \ref{C;Corollary2.16}, each element in $B_4\cup C_1\cup D_1\cup H_1\cup\{E_0^*, E_4^*+E_4^*A_1E_2^*A_3E_4^*\}$ is an $\F$-linear combination of the elements in $B$. As $n>8$ or $p=7$ and $n=8$, note that $B=B_1\cup B_2\cup B_5\cup B_6$. Notice that $B_4\!=\!\{E_a^*JE_b^*: a, b\in [0,4]\}\setminus\{E_0^*JE_0^*\}$. By combining \eqref{Eq;2} and \eqref{Eq;3}, each element in $B$ is also an $\F$-linear combination of the elements in $B_4\cup C_1\cup D_1\cup H_1\cup\{E_0^*, E_4^*+E_4^*A_1E_2^*A_3E_4^*\}$. Since Corollary \ref{C;Corollary2.16} holds and $|B|=|B_4\cup C_1\cup D_1\cup H_1\cup\{E_0^*, E_4^*+E_4^*A_1E_2^*A_3E_4^*\}|$, the first statement thus follows. The second statement follows from the first one. We are done.
\end{proof}
\begin{lem}\label{L;Lemma3.9}
Assume that $p=3$ and $n=4$. Then $K_1\cup H_1\cup\{E_0^*, E_4^*+E_4^*A_1E_2^*A_3E_4^*\}$ is an $\F$-basis of $T$. In particular, $K_1$ is an $\F$-basis of $V_1$.
\end{lem}
\begin{proof}
By Corollary \ref{C;Corollary2.16}, each element in $K_1\cup H_1\cup\{E_0^*, E_4^*+E_4^*A_1E_2^*A_3E_4^*\}$ is an $\F$-linear combination of the elements in $B$. Note that $B=B_1\cup B_2\cup\{E_4^*A_1E_2^*A_3E_4^*\}$. Notice that $B_4=\{E_a^*JE_b^*: a, b\in [0,4]\}\setminus\{E_0^*JE_0^*\}$. By combining \eqref{Eq;2} and \eqref{Eq;3}, we observe that each element in $B$ is also an $\F$-linear combination of the elements in $K_1\cup H_1\cup\{E_0^*, E_4^*+E_4^*A_1E_2^*A_3E_4^*\}$. Since $|B|=|K_1\cup H_1\cup\{E_0^*, E_4^*A_1E_2^*A_3E_4^*\}|$ and Corollary \ref{C;Corollary2.16} holds, the first statement thus follows. Hence $K_1$ is an $\F$-linearly independent subset of $T$. By combining \eqref{Eq;2}, \eqref{Eq;3}, the equalities in Lemmas \ref{L;Lemma1.5}, \ref{L;Lemma1.20}, Lemma \ref{L;Tripleproducts} (iii), Theorem \ref{T;Insectionnumbers} (ii), and the assumption $p=3$, each element in $B_4\cup C_1\cup D_1$ is an $\F$-linear combination of the elements in $K_1$. The second statement thus follows as $K_1\subseteq B_4\cup C_1\cup D_1$ and $V_1=\langle B_4\cup C_1\cup D_1\rangle_\F$. We are done.
\end{proof}
The following lemma shall be repeatedly utilized in the following sections.
\begin{lem}\label{L;Lemma3.10}
The unital $\F$-subalgebra $T$ of $M_X(\F)$ is generated by $B_1\cup B_2$.
\end{lem}
\begin{proof}
By combining \eqref{Eq;2}, \eqref{Eq;3}, Lemma \ref{L;Tripleproducts} (iii), Theorem \ref{T;Insectionnumbers} (i), (ii), and (iii), notice that each element in $\{E_a^*A_bE_c^*: a, b, c\in [0,4]\}$ is an $\F$-linear combination of the elements in $B_1\cup B_2$. Hence the desired lemma is from the definition of $T$.
\end{proof}
The following results describe $T$ for the case $n\equiv 1\pmod p$.
\begin{lem}\label{L;Lemma3.11}
Assume that $n\equiv 1\pmod p$. Then $V_1$ is a two-sided ideal of $T$.
\end{lem}
\begin{proof}
Notice that $V_1\subseteq T$. According to the definition of $T$, notice that $M^T\in T$ if $M\in T$. As $n\equiv 1\pmod p$, notice that $B_4=\{E_a^*JE_b^*: a, b\in [0,4]\}\setminus\{E_0^*JE_0^*\}$ and $B_2=B_4\cup\{E_0^*JE_0^*\}$. If $N\in V_1$, then $N^T\in V_1$ as $V_1=\langle B_4\cup C_1\cup D_1\rangle_\F$ and \eqref{Eq;1} holds. By Lemma \ref{L;Lemma3.10}, notice that the desired lemma follows if $MN\in V_1$ for any $M\in B_1\cup B_2$ and $N\in B_4\cup C_1\cup D_1$. If $M\in B_2$ or $N\in B_4$, notice that $MN\in V_1$ by combining \eqref{Eq;2}, \eqref{Eq;3}, and Lemma \ref{L;Tripleproducts} (i).

If $M\in B_1$ and $N\in C_1\setminus\{E_4^*JE_4^*\}$, notice that $MN\in V_1$ by combining \eqref{Eq;2}, the equalities in Lemmas \ref{L;Lemma2.8}, \ref{L;Lemma2.12}, \ref{L;Lemma2.13}, \ref{L;Lemma2.14}, and the assumption $n\equiv 1\pmod p$. If $M\in B_1$ and $N\in D_1$, notice that $MN\in V_1$ by combining \eqref{Eq;2}, the equalities in Lemmas \ref{L;Lemma2.3}, \ref{L;Lemma2.5}, \ref{L;Lemma2.6}, \ref{L;Lemma2.7}, \ref{L;Lemma2.8}, \ref{L;Lemma2.9}, \ref{L;Lemma2.10}, \ref{L;Lemma2.11}, and the assumption $n\equiv 1\pmod p$. Therefore $MN\in V_1$ for any $M\in B_1\cup B_2$ and $N\in B_4\cup C_1\cup D_1$. We are done.
\end{proof}
\begin{lem}\label{L;Lemma3.12}
Assume that $n\equiv 1\pmod p$. Then $V_1\subseteq \mathrm{Rad}T$. In particular, $V_1$ is a nilpotent two-sided ideal of $T$.
\end{lem}
\begin{proof}
As $V_1=\langle B_4\cup C_1\cup D_1\rangle_\F$, the first statement follows if $B_4\cup C_1\cup D_1\subseteq \mathrm{Rad}T$. Note that $B_4\subseteq\mathrm{Rad}T$ by Lemma \ref{L;Tripleproducts} (v). According to Lemmas \ref{L;Lemma3.5} and \ref{L;Radical}, note that $C_1\subseteq U_1\subseteq\mathrm{Rad}E_4^*TE_4^*\subseteq \mathrm{Rad}T$. It is enough to check that $D_1\subseteq \mathrm{Rad}T$. Define $M_1\!=\!E_4^*A_1E_2^*A_3E_4^*-E_4^*A_1E_3^*A_2E_4^*\!\in\! C_1$, $M_2\!=\!E_4^*A_1E_2^*A_3E_4^*-E_4^*A_2E_3^*A_1E_4^*\!\in\! C_1$, and $M_3=E_4^*A_1E_2^*A_3E_4^*-E_4^*A_3E_2^*A_1E_4^*\in C_1$. So $E_1^*A_2E_4^*\!-\!E_1^*A_3E_4^*\!=\!E_1^*A_2E_4^*M_1$, $E_2^*A_1E_3^*A_2E_4^*-E_2^*A_3E_4^*=E_2^*A_3E_4^*M_1$, $E_3^*A_1E_2^*A_3E_4^*-E_3^*A_2E_4^*=-E_3^*A_2E_4^*M_1$, $E_2^*A_1E_4^*-E_2^*A_3E_4^*=E_2^*A_3E_4^*M_2$, $E_1^*A_2E_3^*A_1E_4^*-E_1^*A_3E_4^*=E_1^*A_3E_4^*M_2$, and $E_3^*A_1E_4^*-E_3^*A_1E_2^*A_3E_4^*=E_3^*A_2E_4^*M_3$ by \eqref{Eq;2} and the equalities in Lemma \ref{L;Lemma2.12}. As $M_1, M_2, M_3\in \mathrm{Rad}T$, the above equalities hold, and $M\in \mathrm{Rad}T$ if $M^T\in \mathrm{Rad}T$, all elements in $D_1$ are contained in $\mathrm{Rad}T$. The first statement follows. The second statement is from Lemma \ref{L;Lemma3.10} and the first one. The desired lemma is proved.
\end{proof}
\begin{lem}\label{L;Lemma3.13}
Assume that $n\equiv 1\pmod p$. Then the $\F$-algebra $T/V_1$ has an $\F$-basis $\{M+V_1: M\in H_1\cup\{E_0^*, E_4^*+E_4^*A_1E_2^*A_3E_4^*\}\}$. The $\F$-algebra $T/V_1$ has two-sided ideals $\langle\{E_0^*+V_1\}\rangle_\F$, $\langle\{E_4^*+E_4^*A_1E_2^*A_3E_4^*+V_1\}\rangle_\F$, and $\langle\{M+V_1: M\in H_1\}\rangle_\F$.
\end{lem}
\begin{proof}
The $\F$-algebra $T/V_1$ is defined by Lemma \ref{L;Lemma3.11}. As $n\equiv 1\pmod p$, the first statement follows from Lemmas \ref{L;Lemma3.8} and \ref{L;Lemma3.9}. As explained in Notation \ref{N;Notation3.7}, note that $B_4=\{E_a^*JE_b^*: a, b\in [0,4]\}\setminus\{E_0^*JE_0^*\}$. Let $C$ be the $\F$-basis of $T/V_1$ in the first statement. Pick $M_1\in C$ and $M_2\in\{M+V_1: M\in H_1\}$. Define $M_3=E_0^*+V_1$ and $M_4=E_4^*+E_4^*A_1E_2^*A_3E_4^*+V_1$. Notice that $M_1M_2, M_2M_1\in \langle\{M+V_1: M\in H_1\}\rangle_\F$, $M_3^2=M_3$, $M_4^2=M_4$, $M_2M_3\!=\!M_3M_2\!=\!M_2M_4\!=\!M_4M_2\!=\!M_3M_4\!=\!M_4M_3\!=\!O+V_1$ by combining \eqref{Eq;2}, the equalities in Lemmas \ref{L;Lemma2.9}, \ref{L;Lemma2.12}, \ref{L;Lemma2.15}. As $M_1$ and $M_2$ are chosen arbitrarily and the first statement holds, the second statement thus follows.
\end{proof}
If $a\in \mathbb{N}$ and $b,c\in [1,a]$, recall the definitions of $E_{bc}(a)$ and $\delta_{bc}$ in Subsection \ref{S;Subsection4}.

We are now ready to conclude this section with the main result of this section.
\begin{cor}\label{C;Corollary3.14}
If $n\equiv 1\pmod p$, then $T/\mathrm{Rad}T\cong M_4(\F)\oplus M_1(\F)\oplus M_1(\F)$ as $\F$-algebras and $\mathrm{Rad}T=V_1$.
\end{cor}
\begin{proof}
By Lemma \ref{L;Lemma3.12}, it suffices to check that $T/V_1\cong M_4(\F)\oplus M_1(\F)\oplus M_1(\F)$ as $\F$-algebras. Notice that $B_4=\{E_a^*JE_b^*: a,b\in [0,4]\}\setminus\{E_0^*JE_0^*\}\subseteq V_1$. By Lemma \ref{L;Lemma3.13}, $T/V_1$ is a direct sum of two-sided ideals $\langle\{M+V_1: M\in H_1\}\rangle_\F$, $\langle\{E_0^*+V_1\}\rangle_\F$, $\langle\{E_4^*+E_4^*A_1E_2^*A_3E_4^*+V_1\}\rangle_\F$. According to \eqref{Eq;2} and Lemma \ref{L;Lemma2.15} (i), observe that $\langle\{E_0^*+V_1\}\rangle_\F\cong \langle\{E_4^*+E_4^*A_1E_2^*A_3E_4^*+V_1\}\rangle_\F\cong M_1(\F)$ as $\F$-algebras. Let us define
$M_{11}=E_1^*+V_1$, $M_{12}=E_1^*A_3E_2^*+V_1$, $M_{13}=E_1^*A_2E_3^*+V_1$, $M_{14}\!=\!E_1^*A_2E_4^*\!+\!V_1$. Set $M_{21}=E_2^*A_3E_1^*+V_1$, $M_{22}=E_2^*+V_1$,
$M_{23}=E_2^*A_1E_3^*+V_1$, $M_{24}\!=\!E_2^*A_1E_4^*\!+\!V_1$. Put $M_{31}=E_3^*A_2E_1^*+V_1$, $M_{32}=E_3^*A_1E_2^*+V_1$, $M_{33}\!=\!E_3^*\!+\!V_1$, $M_{34}\!=\!E_3^*A_1E_4^*\!+\!V_1$. Define $M_{41}=-E_4^*A_2E_1^*+V_1$, $M_{42}=-E_4^*A_1E_2^*+V_1$, $M_{43}=-E_4^*A_1E_3^*+V_1$. Also define $M_{44}=-E_4^*A_1E_2^*A_3E_4^*+V_1$. By Lemma \ref{L;Lemma3.13}, notice that $\{M_{ab}: a,b\in [1,4]\}$ is an $\F$-basis of $\langle\{M+V_1: M\in H_1\}\rangle_\F$. By combining \eqref{Eq;2}, the equalities in Lemmas \ref{L;Lemma2.1}, \ref{L;Lemma2.2}, \ref{L;Lemma2.4}, \ref{L;Lemma2.3}, \ref{L;Lemma2.5}, \ref{L;Lemma2.9}, \ref{L;Lemma2.12}, \ref{L;Lemma2.15} (i), $M_{gh}M_{rs}\!=\!\delta_{hr}M_{gs}$ for any $g, h, r,s\in [1,4]$. Hence there is an $\F$-algebra isomorphism from $\langle\{M+V_1: M\in H_1\}\rangle_\F$ to $M_4(\F)$ that sends $M_{gh}$ to $E_{gh}(4)$ for any $g, h\in [1,4]$. The desired corollary thus follows.
\end{proof}
\section{Algebraic structure of $T$: Case II}
Recall that $G$ is assumed further to be an elementary abelian $2$-group. The aim of this section is to get the algebraic structure of $T$ for the case $n\equiv 2\pmod p$. For our purpose, recall Notation \ref{N;Notation1.21} and the definition of $E_4^*TE_4^*$ in Subsection \ref{S;Subsection3}.

We first introduce the required notation and present some related properties.
\begin{nota}\label{N;Notation4.1}
If $p\neq 2$, set $C_2\!=\!\{E_4^*JE_4^*\}$. If $p=2$, $C_2$ is the set containing exactly
\begin{align*}
&E_4^*JE_4^*, E_4^*+E_4^*A_1E_4^*+E_4^*A_2E_3^*A_1E_4^*+E_4^*A_3E_2^*A_1E_4^*,\\
&E_4^*+E_4^*A_2E_4^*+E_4^*A_1E_3^*A_2E_4^*+E_4^*A_3E_1^*A_2E_4^*,\\
&E_4^*+E_4^*A_3E_4^*+E_4^*A_1E_2^*A_3E_4^*+E_4^*A_2E_1^*A_3E_4^*,\\
&E_4^*A_1E_2^*A_3E_4^*+E_4^*A_2E_1^*A_3E_4^*+E_4^*A_3E_1^*A_2E_4^*+E_4^*A_3E_2^*A_1E_4^*,\\
&E_4^*A_2E_3^*A_1E_4^*+E_4^*A_3E_2^*A_1E_4^*+E_4^*A_1E_2^*A_3E_4^*+E_4^*A_1E_3^*A_2E_4^*.
\end{align*}
Use $D_2$ to denote the set containing precisely $E_1^*A_2E_4^*+E_1^*A_3E_4^*+E_1^*A_2E_3^*A_1E_4^*$, $E_2^*A_1E_4^*+E_2^*A_3E_4^*+E_2^*A_1E_3^*A_2E_4^*$, $E_3^*A_1E_4^*+E_3^*A_2E_4^*+E_3^*A_1E_2^*A_3E_4^*$, and their transposes. Define $U_2=\langle C_2\rangle_\F\subseteq T$. If $p\neq 2$, set $V_2=\langle B_4\rangle_\F\subseteq T$. If $p=2$, put $V_2\!=\!\langle B_4\cup C_2\cup D_2\rangle_\F\subseteq T$.
\end{nota}
\begin{lem}\label{L;Lemma4.2}
If $p\neq 2$ and $n\equiv 2\pmod p$, $B_6\cup C_2\cup\{E_4^*, E_4^*A_1E_4^*, E_4^*A_2E_4^*, E_4^*A_3E_4^*\}$ is an $\F$-basis of the $\F$-subalgebra $E_4^*TE_4^*$ of $T$. In particular, $C_2$ is an $\F$-basis of $U_2$.
\end{lem}
\begin{proof}
The first statement is proved by combining \eqref{Eq;2}, \eqref{Eq;3}, and Corollary \ref{C;Corollary2.17}. The second statement follows from the first one. The desired lemma is proved.
\end{proof}
\begin{lem}\label{L;Lemma4.3}
If $p=2$ and $n\!>\!8$, the union of $C_2$ and the set containing exactly $E_4^*+E_4^*A_2E_1^*A_3E_4^*+E_4^*A_3E_1^*A_2E_4^*$, $E_4^*A_1E_2^*A_3E_4^*+E_4^*A_2E_1^*A_3E_4^*$, $E_4^*A_2E_1^*A_3E_4^*$, $E_4^*A_1E_3^*A_2E_4^*+E_4^*A_3E_1^*A_2E_4^*$, $E_4^*A_3E_1^*A_2E_4^*$ is an $\F$-basis of the $\F$-subalgebra $E_4^*TE_4^*$ of $T$. In particular, $C_2$ is an $\F$-basis of $U_2$.
\end{lem}
\begin{proof}
Let $C$ denote the union in the first statement. According to \eqref{Eq;2} and \eqref{Eq;3}, notice that each element in $C$ is an $\F$-linear combination of the elements in $B_3\cup B_6$. According to \eqref{Eq;2} and \eqref{Eq;3}, notice that each element in $B_3\cup B_6$ is also an $\F$-linear combination of the elements in $C$. As $|C|=|B_3\cup B_6|$ and Corollary \ref{C;Corollary2.17} holds, the first statement follows. The second statement is from the first one. We are done.
\end{proof}
\begin{lem}\label{L;Lemma4.4}
If $p=2$ and $n=8$, the union of the set $C_2\setminus\{E_4^*JE_4^*\}$ and the set containing exactly $E_4^*+E_4^*A_2E_1^*A_3E_4^*+E_4^*A_3E_1^*A_2E_4^*$, $E_4^*A_1E_2^*A_3E_4^*+E_4^*A_2E_1^*A_3E_4^*$, $E_4^*A_2E_1^*A_3E_4^*$, $E_4^*A_1E_3^*A_2E_4^*+E_4^*A_3E_1^*A_2E_4^*$, $E_4^*A_3E_1^*A_2E_4^*$ forms an $\F$-basis of the $\F$-subalgebra $E_4^*TE_4^*$ of $T$. In particular, $C_2\setminus\{E_4^*JE_4^*\}$ is an $\F$-basis of $U_2$.
\end{lem}
\begin{proof}
Let $C$ denote the union in the first statement. By combining \eqref{Eq;2}, \eqref{Eq;3}, and Lemma \ref{L;Lemma1.19}, notice that each element in $C$ is an $\F$-linear combination of the elements in $(B_3\cup B_6)\setminus\{E_4^*A_1E_2^*A_3E_4^*\}$. Furthermore, notice that each element in $(B_3\cup B_6)\setminus\{E_4^*A_1E_2^*A_3E_4^*\}$ is also an $\F$-linear combination of the elements in $C$. As $|C|=|(B_3\cup B_6)\setminus\{E_4^*A_1E_2^*A_3E_4^*\}|$ and Corollary \ref{C;Corollary2.17} holds, the first statement follows. The second statement is from the first one and Lemma \ref{L;Lemma1.19}. The proof of the desired lemma is now complete.
\end{proof}
\begin{lem}\label{L;Lemma4.5}
If $p\!=\!2$ and $n\!=\!4$, the union of $\{E_4^*JE_4^*\}$ and the set containing exactly $E_4^*+E_4^*A_2E_1^*A_3E_4^*+E_4^*A_3E_1^*A_2E_4^*$, $E_4^*A_1E_2^*A_3E_4^*+E_4^*A_2E_1^*A_3E_4^*$, $E_4^*A_2E_1^*A_3E_4^*$, $E_4^*A_1E_3^*A_2E_4^*+E_4^*A_3E_1^*A_2E_4^*$, $E_4^*A_3E_1^*A_2E_4^*$ forms an $\F$-basis of the $\F$-subalgebra $E_4^*TE_4^*$ of $T$. In particular, $\{E_4^*JE_4^*\}$ is an $\F$-basis of $U_2$.
\end{lem}
\begin{proof}
Let $C$ denote the union in the first statement. By combining \eqref{Eq;2}, \eqref{Eq;3}, and the equalities in Lemma \ref{L;Lemma1.20}, each element in $C$ is an $\F$-linear combination of the elements in $B_3\cup\{E_4^*A_1E_2^*A_3E_4^*\}$. By combining \eqref{Eq;2}, \eqref{Eq;3}, the assumption $p=2$, and the equalities in Lemma \ref{L;Lemma1.20}, notice that each element in $B_3\cup\{E_4^*A_1E_2^*A_3E_4^*\}$ is also an $\F$-linear combination of the elements in $C$. As $|C|=|B_3\cup\{E_4^*A_1E_2^*A_3E_4^*\}|$ and Corollary \ref{C;Corollary2.17} holds, the first statement follows. As $U_2=\langle\{E_4^*JE_4^*\}\rangle_\F$ by the assumption $p=2$ and the equalities in Lemma \ref{L;Lemma1.20}, the second statement is proved by the first one. The desired lemma is proved.
\end{proof}
The following results describe $E_4^*TE_4^*$ for the case $n\equiv 2\pmod p$.
\begin{lem}\label{L;Lemma4.6}
If $n\equiv 2\pmod p$, $U_2$ is a two-sided ideal of the $\F$-subalgebra $E_4^*TE_4^*$ of $T$.
\end{lem}
\begin{proof}
Note that $U_2\subseteq E_4^*TE_4^*$. Assume that $p\neq 2$. As $n\equiv 2\pmod p$ and $k_4=p_{44}^0$, notice that $U_2=\langle \{E_4^*JE_4^*\}\rangle_\F=\langle B_4\rangle_\F\cap E_4^*TE_4^*$ by combining Theorem \ref{T;Insectionnumbers} (i), \eqref{Eq;2}, and \eqref{Eq;3}. By Lemma \ref{L;Tripleproducts} (v), $U_2$ is a two-sided ideal of the $\F$-subalgebra
$E_4^*TE_4^*$ of $T$. Assume that $p=2$. If $M\in B_3\cup B_6$, observe that $M^T\in B_3\cup B_6$ by \eqref{Eq;1}. If $N\in U_2$, notice that $N^T\in U_2$ as $p=2$, $U_2=\langle C_2\rangle_\F$, and \eqref{Eq;1} holds. According to Corollary \ref{C;Corollary2.17}, the desired lemma thus follows if $MN\in U_2$ for any $M\in B_3\cup B_6$ and $N\in C_2$. If one of the elements $M$ and $N$ is $E_4^*JE_4^*$, then $MN\in U_2$ by \eqref{Eq;2} and Lemma \ref{L;Tripleproducts} (i). By combining \eqref{Eq;2}, \ref{Eq;3}, the equalities in Lemmas \ref{L;Lemma2.13}, \ref{L;Lemma2.14}, \ref{L;Lemma2.15}, the assumptions $p=2$, and $2\mid n$, notice that $MN\in U_2$ for any $M\in B_3\cup B_6$ and $N\in C_2\setminus\{E_4^*JE_4^*\}$. The desired lemma is proved.
\end{proof}
\begin{lem}\label{L;Lemma4.7}
If $n\equiv 2\pmod p$, the two-sided ideal $U_2$ of the $\F$-subalgebra $E_4^*TE_4^*$ of $T$ is nilpotent. In particular, $U_2\subseteq \mathrm{Rad}E_4^*TE_4^*$.
\end{lem}
\begin{proof}
As $U_2=\langle C_2\rangle_\F$ and Lemma \ref{L;Lemma4.6} holds, the first statement follows if the product of any three elements in $C_2$ is the zero matrix. For any $M\in C_2$, we first claim that $ME_4^*JE_4^*=E_4^*JE_4^*M\!=\!O$. As $n\equiv2\pmod p$ and $k_4=p_{44}^0$, it is obvious to see that $(E_4^*JE_4^*)^2=O$ by combining \eqref{Eq;2}, \eqref{Eq;4}, and Theorem \ref{T;Insectionnumbers} (i). By combining \eqref{Eq;2}, Lemma \ref{L;Tripleproducts} (i), Theorem \ref{T;Insectionnumbers} (ii), (iii), and the assumption $n\equiv 2\pmod p$, notice that $ME_4^*JE_4^*=E_4^*JE_4^*M=O$ for any $M\in C_2\setminus\{E_4^*JE_4^*\}$. The desired claim is proved. If $p\neq 2$, notice that $C_2=\{E_4^*JE_4^*\}$ and $(E_4^*JE_4^*)^3=O$ by the proven claim. So the case $p\neq 2$ is solved. Assume further that $p=2$ and $M_1, M_2, M_3\in C_2$. If one of the elements $M_1, M_2, M_3$ is $E_4^*JE_4^*$, then $M_1M_2M_3=O$ by the proven claim. If $M_1, M_2, M_3\in C_2\setminus\{E_4^*JE_4^*\}$, then $M_1M_2M_3=O$ by combining \eqref{Eq;2}, the equalities in Lemmas \ref{L;Lemma2.13}, \ref{L;Lemma2.14}, \ref{L;Lemma2.15}, the proven claim, the assumptions $p=2$, and $2\mid n$. So the first statement follows. The second statement is from the first one. The desired lemma is proved.
\end{proof}
If $a\in \mathbb{N}$ and $b,c\in [1,a]$, recall the definitions of $E_{bc}(a)$ and $\delta_{bc}$ in Subsection \ref{S;Subsection4}.
We now determine the algebraic structure of $E_4^*TE_4^*$ for the case $n\equiv 2\pmod p$.
\begin{cor}\label{C;Corollary4.8}
If $p\neq 2$ and $n\equiv 2\pmod p$, $E_4^*TE_4^*/\mathrm{Rad}E_4^*TE_4^*\cong M_3(\F)\oplus M_1(\F)$ as $\F$-algebras and $\mathrm{Rad}E_4^*TE_4^*=U_2$.
\end{cor}
\begin{proof}
By Lemma \ref{L;Lemma4.7}, we just check that $E_4^*TE_4^*/U_2\cong M_3(\F)\oplus M_1(\F)$ as $\F$-algebras. Let $C=\!=\!\{M+U_2: M\!\in\!B_6\cup\{E_4^*, E_4^*A_1E_4^*, E_4^*A_2E_4^*, E_4^*A_3E_4^*\}\}$. As $p\neq 2$, define
\begin{align*}
& M_{11}=\overline{2}^{-1}(E_4^*+E_4^*A_1E_4^*-E_4^*A_2E_3^*A_1E_4^*-E_4^*A_3E_2^*A_1E_4^*)+U_2,\\
& M_{12}=\overline{2}^{-1}(E_4^*A_1E_3^*A_2E_4^*-E_4^*A_3E_1^*A_2E_4^*-E_4^*-E_4^*A_2E_4^*)+U_2,\\
& M_{13}=\overline{2}^{-1}(E_4^*A_1E_2^*A_3E_4^*-E_4^*A_2E_1^*A_3E_4^*-E_4^*-E_4^*A_3E_4^*)+U_2,\\
& M_{21}=\overline{2}^{-1}(E_4^*A_2E_3^*A_1E_4^*-E_4^*A_3E_2^*A_1E_4^*-E_4^*-E_4^*A_1E_4^*)+U_2,\\
& M_{22}=\overline{2}^{-1}(E_4^*+E_4^*A_2E_4^*-E_4^*A_1E_3^*A_2E_4^*-E_4^*A_3E_1^*A_2E_4^*)+U_2,\\
& M_{23}=\overline{2}^{-1}(E_4^*A_2E_1^*A_3E_4^*-E_4^*A_1E_2^*A_3E_4^*-E_4^*-E_4^*A_3E_4^*)+U_2,\\
& M_{31}=\overline{2}^{-1}(E_4^*A_3E_2^*A_1E_4^*-E_4^*A_2E_3^*A_1E_4^*-E_4^*-E_4^*A_1E_4^*)+U_2,\\
& M_{32}=\overline{2}^{-1}(E_4^*A_3E_1^*A_2E_4^*-E_4^*A_1E_3^*A_2E_4^*-E_4^*-E_4^*A_2E_4^*)+U_2,\\
& M_{33}=\overline{2}^{-1}(E_4^*+E_4^*A_3E_4^*-E_4^*A_1E_2^*A_3E_4^*-E_4^*A_2E_1^*A_3E_4^*)+U_2.
\end{align*}
Put $N=E_4^*-\sum_{j=1}^3M_{jj}$ and $D=\{M_{ab}: a, b\in [1,3]\}\cup\{N\}$. Then each element in $D$ is an $\F$-linear combination of the elements in $C$. Furthermore, each element in $C$ is an $\F$-linear combination of the elements in $D$. As $|C|=|D|$ and Lemma \ref{L;Lemma4.2} holds, $D$ is an $\F$-basis of $E_4^*TE_4^*/U_2$. Hence $E_4^*TE_4^*/U_2=\langle\{M_{ab}: a, b\in [1,3]\}\rangle_\F\oplus\langle\{N\}\rangle_\F$. By combining \eqref{Eq;2}, the equalities in Lemmas \ref{L;Lemma2.13}, \ref{L;Lemma2.14}, \ref{L;Lemma2.15}, and the assumption $n\equiv 2\pmod p$, notice that $M_{gh}M_{rs}=\delta_{hr}M_{gs}$, $M_{gh}N=NM_{gh}=O+U_2$, $N^2=N$ for any $g, h, r, s\in [1,3]$. So both $\langle\{M_{ab}: a, b\in [1,3]\}\rangle_\F$ and $\langle\{N\}\rangle_\F$ are two-sided ideals of $E_4^*TE_4^*/U_2$. Observe that $\langle\{N\}\rangle_\F\cong M_1(\F)$ as $\F$-algebras. Moreover, there is an $\F$-algebra isomorphism from $\langle\{M_{ab}: a, b\in [1,3]\}\rangle_\F$ to $M_3(\F)$ that sends $M_{gh}$ to $E_{gh}(3)$ for any $g, h\in [1,3]$. The desired corollary thus follows.
\end{proof}
\begin{cor}\label{C;Corollary4.9}
If $p=2$, $E_4^*TE_4^*/\mathrm{Rad}E_4^*TE_4^*\cong M_2(\F)\oplus M_1(\F)$ as $\F$-algebras and $\mathrm{Rad}E_4^*TE_4^*=U_2$.
\end{cor}
\begin{proof}
By Lemma \ref{L;Lemma4.7}, it is enough to check that $E_4^*TE_4^*/U_2\cong M_2(\F)\oplus M_1(\F)$ as $\F$-algebras. Set $M_{11}\!=\!E_4^*A_2E_1^*A_3E_4^*+U_2$, $M_{12}\!=\!E_4^*A_1E_3^*A_2E_4^*+E_4^*A_3E_1^*A_2E_4^*+U_2$, $M_{21}\!=\!E_4^*A_1E_2^*A_3E_4^*+E_4^*A_2E_1^*A_3E_4^*+U_2$, $M_{22}\!=\!E_4^*A_3E_1^*A_2E_4^*+U_2$. Moreover, put $N=E_4^*+M_{11}+M_{22}$. By combining Lemmas \ref{L;Lemma4.3}, \ref{L;Lemma4.4}, \ref{L;Lemma4.5}, $\{M_{ab}: a,b\in [1,2]\}\cup\{N\}$ is an $\F$-basis of $E_4^*TE_4^*/U_2$. Hence $E_4^*TE_4^*/U_2=\langle\{M_{ab}: a,b \in [1,2]\}\rangle_\F\oplus\langle\{N\}\rangle_\F$. By combining \eqref{Eq;2}, the equalities in Lemma \ref{L;Lemma2.15}, the assumptions $p=2$, and $2\mid n$, notice that $M_{gh}M_{rs}=\delta_{hr}M_{gs}$, $M_{gh}N=NM_{gh}=O+U_2$, $N^2=N$ for any $g, h, r,s\in [1,2]$. So both $\langle\{M_{ab}: a,b \in [1,2]\}\rangle_\F$ and $\langle \{N\}\rangle_\F$ are two-sided ideals of $E_4^*TE_4^*/U_2$. Moreover, $\langle \{N\}\rangle_\F\cong M_1(\F)$ as $\F$-algebras and there exists an $\F$-algebra isomorphism from $\langle\{M_{ab}: a,b\in [1,2]\}\rangle_\F$ to $M_2(\F)$ that sends $M_{gh}$ to $E_{gh}(2)$ for any $g, h\in [1,2]$. The desired corollary thus follows.
\end{proof}
We next present the following notation and some related properties.
\begin{nota}\label{N;Notation4.10}
Let $H_{2,1}=\{E_a^*JE_b^*:a,b\in [0,3]\}$. Let $H_{2,2}$ be the set containing exactly $E_2^*A_1E_3^*A_2E_4^*$, $E_3^*A_1E_2^*A_3E_4^*$, $E_4^*A_2E_1^*A_3E_4^*$, $E_4^*A_2E_3^*A_1E_2^*$, $E_4^*A_3E_1^*A_2E_4^*$, $E_4^*A_3E_2^*A_1E_3^*$, $E_4^*A_1E_2^*A_3E_4^*\!+\!E_4^*A_2E_1^*A_3E_4^*$, $E_4^*A_1E_3^*A_2E_4^*+E_4^*A_3E_1^*A_2E_4^*$, $E_1^*A_2E_4^*$, $E_1^*A_3E_4^*$,
$E_2^*A_3E_4^*$, $E_3^*A_2E_4^*$, $E_4^*A_2E_1^*$, $E_4^*A_2E_3^*$, $E_4^*A_3E_1^*$, $E_4^*A_3E_2^*$, $E_1^*+E_1^*JE_1^*$, $E_1^*A_2E_3^*\!+\!E_1^*JE_3^*$, $E_1^*A_3E_2^*+E_1^*JE_2^*$,
$E_2^*+E_2^*JE_2^*$, $E_2^*A_1E_3^*\!+\!E_2^*JE_3^*$, $E_2^*A_3E_1^*\!+\!E_2^*JE_1^*$, $E_3^*\!+\!E_3^*JE_3^*$, $E_3^*A_1E_2^*\!+\!E_3^*JE_2^*$, $E_3^*A_2E_1^*+E_3^*JE_1^*$. Set $K_2=(B_4\cup C_2\cup D_2)\setminus\{E_4^*JE_4^*\}$. If $n\equiv2\pmod p$, notice that $B_4=\{E_a^*JE_4^*:a\in [0,4]\}\cup\{E_4^*JE_a^*:a\in [0,3]\}$ as $k_j=p_{jj}^0$ for any $j\in [0,4]$ and Theorem \ref{T;Insectionnumbers} (i) holds.
\end{nota}
\begin{lem}\label{L;Lemma4.11}
Assume that $p=2$ and $n>8$. Then the set containing precisely $E_4^*+E_4^*A_2E_1^*A_3E_4^*+E_4^*A_3E_1^*A_2E_4^*$ and the elements in $B_4\cup C_2\cup D_2\cup H_{2,1}\cup H_{2,2}$ is an $\F$-basis of $T$. In particular, $B_4\cup C_2\cup D_2$ is an $\F$-basis of $V_2$.
\end{lem}
\begin{proof}
Let $C$ be the set containing precisely $E_4^*+E_4^*A_2E_1^*A_3E_4^*+E_4^*A_3E_1^*A_2E_4^*$ and the elements in $B_4\cup C_2\cup D_2\cup H_{2,1}\cup H_{2,2}$. By Corollary \ref{C;Corollary2.16}, each element in $C$ is an $\F$-linear combination of the elements in $B$. As $p=2$ and $n>8$, notice that $B=B_1\cup B_2\cup B_5\cup B_6$.
By combining \eqref{Eq;2}, \eqref{Eq;3}, Lemma \ref{L;Tripleproducts} (iii), Theorem \ref{T;Insectionnumbers} (i), (ii), and (iii), each element in $B$ is also an $\F$-linear combination of the elements in $C$. As $|B|=|C|$ and Corollary \ref{C;Corollary2.16} holds, the first statement thus follows. As $V_2=\langle B_4\cup C_2\cup D_2\rangle_\F$, the second statement is from the first one. We are done.
\end{proof}
\begin{lem}\label{L;Lemma4.12}
Assume that $p=2$ and $n=8$. Then the set containing precisely $E_4^*+E_4^*A_2E_1^*A_3E_4^*+E_4^*A_3E_1^*A_2E_4^*$ and the elements in $K_2\cup H_{2,1}\cup H_{2,2}$ is an $\F$-basis of $T$. In particular, $K_2$ is an $\F$-basis of $V_2$.
\end{lem}
\begin{proof}
Let $C$ be the set containing precisely $E_4^*+E_4^*A_2E_1^*A_3E_4^*+E_4^*A_3E_1^*A_2E_4^*$ and the elements in $K_2\cup H_{2,1}\cup H_{2,2}$. By Corollary \ref{C;Corollary2.16}, notice that each element in $C$ is an $\F$-linear combination of the elements in $B$. As $p=2$ and $n=8$, notice that $B=(B_1\cup B_2\cup B_5\cup B_6)\setminus\{E_4^*A_1E_2^*A_3E_4^*\}$. By combining \eqref{Eq;2}, \eqref{Eq;3}, Lemmas \ref{L;Lemma1.19}, \ref{L;Tripleproducts} (iii), Theorem \ref{T;Insectionnumbers} (i), (ii), and (iii), each element in $B$ is also an $\F$-linear combination of the elements in $C$. As $|B|=|C|$ and Corollary \ref{C;Corollary2.16} holds, the first statement thus follows. So $K_2$ is an $\F$-linearly independent subset of $T$. By Lemma \ref{L;Lemma1.19}, each element in $B_4\cup C_2\cup D_2$ is an $\F$-linear combination of the elements in $K_2$. The second statement follows as $K_2\subseteq\langle B_4\cup C_2\cup D_2\rangle_\F=V_2$. We are done.
\end{proof}
\begin{lem}\label{L;Lemma4.13}
Assume that $p=2$ and $n=4$. Then the set containing precisely $E_4^*+E_4^*A_2E_1^*A_3E_4^*+E_4^*A_3E_1^*A_2E_4^*$ and the elements in $B_4\cup H_{2,1}\cup H_{2,2}$ is an $\F$-basis of $T$. In particular, $B_4$ is an $\F$-basis of $V_2$.
\end{lem}
\begin{proof}
Let $C$ be the set containing precisely $E_4^*+E_4^*A_2E_1^*A_3E_4^*+E_4^*A_3E_1^*A_2E_4^*$ and the elements in $B_4\cup H_{2,1}\cup H_{2,2}$. By Corollary \ref{C;Corollary2.16}, each element in $C$ is an $\F$-linear combination of the elements in $B$. Notice that $B=B_1\cup B_2\cup\{E_4^*A_1E_2^*A_3E_4^*\}$ as $p=2$ and $n=4$. By combining \eqref{Eq;2}, \eqref{Eq;3}, the assumption $p=2$, the equalities in Lemmas \ref{L;Lemma1.5}, \ref{L;Lemma1.20}, Lemma \ref{L;Tripleproducts} (iii), each element in $B$ is also an $\F$-linear combination of the elements in $C$. Since $|B|=|C|$ and Corollary \ref{C;Corollary2.16} holds, the first statement follows. Notice that $B_4\cup C_2\cup D_2=B_4\cup\{O\}$ as $p=2$ and the equalities in \eqref{Eq;3}, Lemmas \ref{L;Lemma1.5}, \ref{L;Lemma1.20}  hold. The second statement is from the first one. We are done.
\end{proof}
The following results describe $T$ for the case $n\equiv 2\pmod p$.
\begin{lem}\label{L;Lemma4.14}
Assume that $n\equiv 2\pmod p$. Then $V_2$ is a two-sided ideal of $T$.
\end{lem}
\begin{proof}
Notice that $V_2\subseteq T$. If $p\neq 2$, then $V_2=\langle B_4\rangle_\F$ and the desired lemma follows from Lemma \ref{L;Tripleproducts} (v). Assume that $p=2$ in this proof.
By the definition of $T$, notice that $M^T\in T$ if $M\in T$. Note that $B_4=\{E_a^*JE_4^*:a\in [0,4]\}\cup\{E_4^*JE_a^*:a\in [0,3]\}$. If $N\in V_2$, then $N^T\in V_2$ as $V_2=\langle B_4\cup C_2\cup D_2\rangle_\F$ and \eqref{Eq;1} holds. By Lemma \ref{L;Lemma3.10}, the desired lemma follows if $MN\in V_2$ for any $M\in B_1\cup B_2$ and $N\in B_4\cup C_2\cup D_2$. If $M\!\in\! B_2$ or $N\!\in\! B_4$, then $MN\!\in \!V_2$ by combining \eqref{Eq;2}, \eqref{Eq;3}, and Lemma \ref{L;Tripleproducts} (i).

If $M\in B_1$ and $N\in C_2\setminus\{E_4^*JE_4^*\}$, notice that $MN\in V_2$ by combining \eqref{Eq;2}, the equalities in Lemmas \ref{L;Lemma2.8}, \ref{L;Lemma2.12}, \ref{L;Lemma2.13},
\ref{L;Lemma2.14}, the assumptions $p=2$, and $2\mid n$. If $M\!\in\! B_1$ and $N\!\in\! D_2$, notice that $MN\in V_2$ by combining \eqref{Eq;2}, the equalities in Lemmas \ref{L;Lemma2.3}, \ref{L;Lemma2.5}, \ref{L;Lemma2.6}, \ref{L;Lemma2.7}, \ref{L;Lemma2.8}, \ref{L;Lemma2.9}, \ref{L;Lemma2.10}, \ref{L;Lemma2.11}, the assumptions $p=2$, and $2\mid n$. Therefore $MN\in V_2$ for any $M\in B_1\cup B_2$ and $N\in B_4\cup C_2\cup D_2$. We are done.
\end{proof}
\begin{lem}\label{L;Lemma4.15}
Assume that $n\equiv 2\pmod p$. Then $V_2\subseteq \mathrm{Rad}T$. In particular, $V_2$ is a nilpotent two-sided ideal of $T$.
\end{lem}
\begin{proof}
The case $p\neq 2$ is from Lemma \ref{L;Tripleproducts} (v). Assume that $p=2$ in this proof. As $V_2=\langle B_4\cup C_2\cup D_2\rangle_\F$, the first statement thus follows if $B_4\cup C_2\cup D_2\subseteq \mathrm{Rad}T$. Notice that $B_4\subseteq \mathrm{Rad} \T$ and $C_2\subseteq U_2\subseteq \mathrm{Rad}E_4^*TE_4^*\subseteq\mathrm{Rad}T$ by combining Lemmas \ref{L;Lemma4.7}, \ref{L;Radical}, \ref{L;Tripleproducts} (v). Therefore it is enough to check the containment $D_2\subseteq \mathrm{Rad}T$. As $p=2$, set  $M=E_4^*A_1E_2^*A_3E_4^*+E_4^*A_2E_1^*A_3E_4^*+E_4^*A_3E_1^*A_2E_4^*+E_4^*A_3E_2^*A_1E_4^*\in C_2$. Notice that $E_1^*A_2E_4^*M+E_1^*JE_4^*$, $E_2^*A_1E_4^*M+E_2^*JE_4$, $E_3^*A_1E_4^*M+E_3^*JE_4^*\in \mathrm{Rad}T$. Then $D_2$ contains exactly $E_1^*A_2E_4^*M+E_1^*JE_4^*$, $E_2^*A_1E_4^*M+E_2^*JE_4$, $E_3^*A_1E_4^*M+E_3^*JE_4^*$, and their transposes by combining \eqref{Eq;2}, \eqref{Eq;3}, the equalities in Lemmas \ref{L;Lemma2.12}, \ref{L;Lemma2.4}, Lemma \ref{L;Tripleproducts} (iii), Theorem \ref{T;Insectionnumbers} (ii), the assumptions $p=2$, $2\mid n$. The first statement thus follows as $N^T\in \mathrm{Rad}T$ if $N\in \mathrm{Rad}T$. The second statement is from Lemma \ref{L;Lemma4.14} and the first one. The desired lemma is proved.
\end{proof}
If $a\in \mathbb{N}$ and $b,c\in [1,a]$, recall the definitions of $E_{bc}(a)$ and $\delta_{bc}$ in Subsection \ref{S;Subsection4}.
\begin{lem}\label{L;Lemma4.16}
Assume that $n\equiv 2\pmod p$. If $p\neq 2$, then the $\F$-algebra $T/V_2$ has an $\F$-basis $\{M+V_2: M\in B_1\cup B_5\cup B_6\cup H_{2,1}\}$. If $p=2$, then the $\F$-algebra $T/V_2 $ has an $\F$-basis $\{M+V_2: M\in H_{2,1}\cup H_{2,3}\cup\{N\}\}$, where $N$ denotes the matrix $E_4^*+E_4^*A_2E_1^*A_3E_4^*+E_4^*A_3E_1^*A_2E_4^*$. In both cases, $\langle\{M+V_2: M\in H_{2,1}\}\rangle_\F$ is a two-sided ideal of the $\F$-algebra $T/V_2$ that is isomorphic to $M_4(\F)$ as $\F$-algebras.
\end{lem}
\begin{proof}
If $p\neq 2$, then $B=B_1\cup B_2\cup B_5\cup B_6$. As $B_2=B_4\cup H_{2,1}$, the first statement is from Corollary \ref{C;Corollary2.16}. The second statement is proved by combining Lemmas \ref{L;Lemma4.12}, \ref{L;Lemma4.13}, \ref{L;Lemma4.14}. Set $M_{gh}\!=\!E_{g-1}^*JE_{h-1}^*+V_2$ for any $g,h\in [1,4]$. So $\{M_{ab}: a,b\in [1,4]\}$ is always an $\F$-basis of $\langle\{M+V_2: M\in H_{2,1}\}\rangle_\F$ by the first two statements. Recall that $B_4\!=\!\{E_a^*JE_4^*: a\in [0,4]\}\cup\{E_4^*JE_a^*: a\in [0,3]\}$. So the $\F$-algebra $T/V_2$ has a two-sided ideal $\langle\{M+V_2: M\in H_{2,1}\}\rangle_\F$ by Lemma \ref{L;Tripleproducts} (i). By combining \eqref{Eq;2}, \eqref{Eq;4}, Theorem \ref{T;Insectionnumbers} (i), (ii), and the assumption $n\!\equiv\! 2\pmod p$, $M_{gh}M_{rs}\!=\!\delta_{hr}M_{gs}$ for any $g, h, r, s\in [1,4]$. So there is an $\F$-algebra isomorphism from $\langle\{M+V_2: M\in H_{2,1}\}\rangle_\F$ to $M_4(\F)$ that sends $M_{gh}$ to $E_{gh}(4)$ for any $g, h\in [1,4]$. We are done.
\end{proof}
If $p\neq 2$ and $n\equiv 2\pmod p$, the next corollary gives the algebraic structure of $T$.
\begin{cor}\label{C;Corollary4.17}
If $p\!\neq\! 2$ and $n\!\equiv\!2\pmod p$, then $T/\mathrm{Rad}T\!\cong\! M_6(\F)\oplus M_4(\F)\oplus M_1(\F)$ as $\F$-algebras and $\mathrm{Rad}T=V_2$.
\end{cor}
\begin{proof}
By Lemma \ref{L;Lemma4.15}, it suffices to check that $T/V_2\cong M_6(\F)\oplus M_4(\F)\oplus M_1(\F)$ as $\F$-algebras. Let $C=\{M+V_2: M\in B_1\cup B_5\cup B_6\cup H_{2,1}\}$. If $j\in\{g,h,i\}=[1,3]$, set $M_{gi}=E_g^*A_hE_i^*-E_g^*JE_i^*+V_2$, $M_{jj}=E_j^*-E_j^*JE_j^*+V_2$, $M_{14}=E_1^*A_2E_3^*A_1E_4^*+V_2$, $M_{15}\!=\!E_1^*A_2E_4^*+V_2$, $M_{16}\!=\!E_1^*A_3E_4^*\!+\!V_2$, $M_{24}\!=\!E_2^*A_1E_4^*\!+\!V_2$, $M_{25}=E_2^*A_1E_3^*A_2E_4^*+V_2$, $M_{26}\!=\!E_2^*A_3E_4^*\!+\!V_2$, $M_{34}\!=\!E_3^*A_1E_4^*\!+\!V_2$, $M_{35}\!=\!E_3^*A_2E_4^*+V_2$, $M_{36}=E_3^*A_1E_2^*A_3E_4^*+V_2$. Use the assumption $p\neq 2$ to continue defining
\begin{align*}
& M_{41}=\overline{2}^{-1}(E_4^*A_1E_3^*A_2E_1^*-E_4^*A_2E_1^*-E_4^*A_3E_1^*)+V_2,\\
& M_{42}=\overline{2}^{-1}(E_4^*A_1E_2^*-E_4^*A_3E_2^*-E_4^*A_2E_3^*A_1E_2^*)+V_2,\\
& M_{43}=\overline{2}^{-1}(E_4^*A_1E_3^*-E_4^*A_2E_3^*-E_4^*A_3E_2^*A_1E_3^*)+V_2,\\
& M_{44}=\overline{2}^{-1}(E_4^*+E_4^*A_1E_4^*-E_4^*A_2E_3^*A_1E_4^*-E_4^*A_3E_2^*A_1E_4^*)+V_2,\\
& M_{45}=\overline{2}^{-1}(E_4^*A_1E_3^*A_2E_4^*-E_4^*A_3E_1^*A_2E_4^*-E_4^*-E_4^*A_2E_4^*)+V_2,\\
& M_{46}=\overline{2}^{-1}(E_4^*A_1E_2^*A_3E_4^*-E_4^*A_2E_1^*A_3E_4^*-E_4^*-E_4^*A_3E_4^*)+V_2,\\
& M_{51}=\overline{2}^{-1}(E_4^*A_2E_1^*-E_4^*A_3E_1^*-E_4^*A_1E_3^*A_2E_1^*)+V_2,\\
& M_{52}=\overline{2}^{-1}(E_4^*A_2E_3^*A_1E_2^*-E_4^*A_1E_2^*-E_4^*A_3E_2^*)+V_2,\\
& M_{53}=\overline{2}^{-1}(E_4^*A_2E_3^*-E_4^*A_1E_3^*-E_4^*A_3E_2^*A_1E_3^*)+V_2,\\
& M_{54}=\overline{2}^{-1}(E_4^*A_2E_3^*A_1E_4^*-E_4^*A_3E_2^*A_1E_4^*-E_4^*-E_4^*A_1E_4^*)+V_2,\\
& M_{55}=\overline{2}^{-1}(E_4^*+E_4^*A_2E_4^*-E_4^*A_1E_3^*A_2E_4^*-E_4^*A_3E_1^*A_2E_4^*)+V_2,\\
& M_{56}=\overline{2}^{-1}(E_4^*A_2E_1^*A_3E_4^*-E_4^*A_1E_2^*A_3E_4^*-E_4^*-E_4^*A_3E_4^*)+V_2,\\
& M_{61}=\overline{2}^{-1}(E_4^*A_3E_1^*-E_4^*A_2E_1^*-E_4^*A_1E_3^*A_2E_1^*)+V_2,\\
& M_{62}=\overline{2}^{-1}(E_4^*A_3E_2^*-E_4^*A_1E_2^*-E_4^*A_2E_3^*A_1E_2^*)+V_2,\\
& M_{63}=\overline{2}^{-1}(E_4^*A_3E_2^*A_1E_3^*-E_4^*A_1E_3^*-E_4^*A_2E_3^*)+V_2,\\
& M_{64}=\overline{2}^{-1}(E_4^*A_3E_2^*A_1E_4^*-E_4^*A_2E_3^*A_1E_4^*-E_4^*-E_4^*A_1E_4^*)+V_2,\\
& M_{65}=\overline{2}^{-1}(E_4^*A_3E_1^*A_2E_4^*-E_4^*A_1E_3^*A_2E_4^*-E_4^*-E_4^*A_2E_4^*)+V_2,\\
& M_{66}=\overline{2}^{-1}(E_4^*+E_4^*A_3E_4^*-E_4^*A_1E_2^*A_3E_4^*-E_4^*A_2E_1^*A_3E_4^*)+V_2.
\end{align*}
Set $N_{gh}=E_{g-1}^*JE_{h-1}^*+V_2$ if $g, h\in [1,4]$. Put $N=I-\sum_{j=1}^4N_{jj}-\sum_{j=1}^6M_{jj}$ and $D=\{M_{ab}: a,b\in [1,6]\}\cup\{N_{ab}: a,b\in [1,4]\}\cup\{N\}$. It is obvious to observe that $\langle \{M+V_2: M\in H_{2,1}\}\rangle_\F=\langle\{N_{ab}: a,b\in [1,4]\}\rangle_\F$. By Lemma \ref{L;Lemma4.16}, notice that each element in $D$ is an $\F$-linear combination of the elements in $C$. Moreover, notice that each element in $C$ is also an $\F$-linear combination of the elements in $D$. Since $|C|=|D|$ and Corollary \ref{C;Corollary2.16} holds, notice that $D$ is also an $\F$-basis of $T/V_2$.

Notice that $B_4=\{E_a^*JE_4^*: a\in [0,4]\}\cup\{E_4^*JE_a^*: a\in [0,3]\}\subseteq V_2$. By combining \eqref{Eq;2}, \eqref{Eq;3}, the equalities in Lemmas \ref{L;Lemma2.1}, \ref{L;Lemma2.2}, \ref{L;Lemma2.3}, \ref{L;Lemma2.4}, \ref{L;Lemma2.5}, \ref{L;Lemma2.6}, \ref{L;Lemma2.7}, \ref{L;Lemma2.8}, \ref{L;Lemma2.9}, \ref{L;Lemma2.10}, \ref{L;Lemma2.11},
\ref{L;Lemma2.12}, \ref{L;Lemma2.13}, \ref{L;Lemma2.14}, \ref{L;Lemma2.15}, Lemma \ref{L;Tripleproducts} (i), (iii), Theorem \ref{T;Insectionnumbers} (ii), (iii), and the assumption $n\equiv 2\pmod p$, notice that $M_{gh}M_{rs}=\delta_{hr}M_{gs}$ for any $g, h, r, s\in [1,6]$. By combining \eqref{Eq;2}, \eqref{Eq;4}, Lemma \ref{L;Tripleproducts} (i), Theorem \ref{T;Insectionnumbers} (i), (ii), (iii), the assumption $n\equiv 2\pmod p$, notice that $M_{gh}N_{rs}=N_{rs}M_{gh}=O+V_2$ and $N_{rs}N_{uv}=\delta_{su}N_{rv}$ for any $g,h\in [1,6]$ and $r, s, u, v\in [1,4]$. For any $g, h\in [1,6]$ and $r,s \in [1,4]$, notice that $M_{gh}N=NM_{gh}=N_{rs}N=NN_{rs}=O+V_2$ and $N^2=N$. So the $\F$-algebra $T/V_2$ has two-sided ideals $\langle\{M_{ab}: a,b \in [1,6]\}\rangle_\F$ and $\langle \{N\}\rangle_\F$. So $\langle\{N\}\rangle_\F\cong M_{1}(\F)$ as $\F$-algebras. There exists an $\F$-algebra isomorphism from $\langle\{M_{ab}: a,b \in [1,6]\}\rangle_\F$ to $M_6(\F)$ that sends $M_{gh}$ to $E_{gh}(6)$ for any $g, h\in [1,6]$. As Lemma \ref{L;Lemma4.16} holds and $T/V_2=\langle\{M_{ab}: a,b\in [1,6]\}\rangle_\F\oplus\langle\{N_{ab}: a,b \in [1,4]\}\rangle_\F\oplus\langle\{N\}\rangle_\F$, the desired corollary thus follows.
\end{proof}
For the case $p=2$, our final corollary describes the algebraic structure of $T$.
\begin{cor}\label{C;Corollary4.18}
If $p=2$, then $T/\mathrm{Rad}T\cong M_5(\F)\oplus M_4(\F)\oplus M_1(\F)$ as $\F$-algebras and $\mathrm{Rad}T=V_2$.
\end{cor}
\begin{proof}
By Lemma \ref{L;Lemma4.15}, it suffices to check that $T/V_2\cong M_5(\F)\oplus M_4(\F)\oplus M_1(\F)$ as $\F$-algebras. Define $N\!=\!E_4^*+E_4^*A_2E_1^*A_3E_4^*+E_4^*A_3E_1^*A_2E_4^*+V_2$. By Lemma \ref{L;Lemma4.16}, notice that $T/V_2=\langle \{M+V_2: M\in H_{2,3}\}\rangle_\F\oplus\langle \{M+V_2: M\in H_{2,1}\}\rangle_\F\oplus\langle\{N\}\rangle_\F$. If $j\in\{g, h, i\}=[1,3]$, define $M_{gi}=E_g^*A_hE_i^*+E_g^*JE_i^*+V_2$, $M_{jj}=E_j^*+E_j^*JE_j^*+V_2$, $M_{14}\!=\!E_1^*A_3E_4^*+V_2, M_{15}\!=\!E_1^*A_2E_4^*+V_2, M_{24}\!=\!E_2^*A_3E_4^*+V_2, M_{25}\!=\!E_2^*A_1E_3^*A_2E_4^*+V_2$, $M_{34}\!=\!E_3^*A_1E_2^*A_3E_4^*+V_2, M_{35}=E_3^*A_2E_4^*+V_2$. Moreover, set $M_{41}\!=\!E_4^*A_2E_1^*+V_2$, $M_{42}\!=\!E_4^*A_2E_3^*A_1E_2^*+V_2$, $M_{43}\!=\!E_4^*A_2E_3^*+V_2$, $M_{51}\!=\!E_4^*A_3E_1^*\!+\!V_2$, $M_{52}\!=\!E_4^*A_3E_2^*+V_2$, $M_{53}=E_4^*A_3E_2^*A_1E_3^*\!+\!V_2$. Define $M_{44}\!=\!E_4^*A_2E_1^*A_3E_4^*+V_2$, $M_{55}\!=\!E_4^*A_3E_1^*A_2E_4^*+V_2$, $M_{45}=E_4^*A_1E_3^*A_2E_4^*+E_4^*A_3E_1^*A_2E_4^*+V_2$,
$M_{54}=E_4^*A_1E_2^*A_3E_4^*+E_4^*A_2E_1^*A_3E_4^*+V_2$. So $\langle\{M_{ab}: a,b\in [1,5]\}\rangle_\F=\langle\{M+V_2: M\in H_{2,3}\}\rangle_\F$. Set $N_{gh}=E_{g-1}^*JE_{h-1}^*+V_2$ if $g, h\in [1,4]$. Hence $\langle\{N_{ab}: a,b\in [1,4]\}\rangle_\F=\langle\{M+V_2: M\in H_{2,1}\}\rangle_\F$. According to Lemma \ref{L;Lemma4.16}, $T/V_2$ has an $\F$-basis $\{M_{ab}: a,b\in [1,5]\}\cup\{N_{ab}: a, b\in [1,4]\}\cup\{N\}$.

Notice that $B_4=\{E_a^*JE_4^*: a\in [0,4]\}\cup\{E_4^*JE_a^*: a\in [0,3]\}\subseteq V_2$. By combining \eqref{Eq;2}, the equalities in Lemmas \ref{L;Lemma2.1}, \ref{L;Lemma2.2}, \ref{L;Lemma2.3}, \ref{L;Lemma2.4}, \ref{L;Lemma2.5}, \ref{L;Lemma2.6}, \ref{L;Lemma2.7}, \ref{L;Lemma2.9}, \ref{L;Lemma2.11}, \ref{L;Lemma2.12}, \ref{L;Lemma2.15}, Lemma \ref{L;Tripleproducts} (i), Theorem \ref{T;Insectionnumbers} (ii), (iii), the assumptions $p=2$, and $2\mid n$, notice that $M_{gh}M_{rs}=\delta_{hr}M_{gs}$ for any $g, h, r,s\in [1,5]$.
According to \eqref{Eq;2}, \eqref{Eq;4}, Lemma \ref{L;Tripleproducts} (i), Theorem \ref{T;Insectionnumbers} (i), (ii), (iii), the assumptions $p=2$, and $2\mid n$, notice that $M_{gh}N_{rs}\!=\!N_{rs}M_{gh}\!=\!O+V_2$ for any $g, h\in [1,5]$ and $r,s\in [1,4]$. Moreover, notice that $N_{rs}N_{uv}\!=\!\delta_{su}N_{rv}$ for any $r,s, u, v\in [1,4]$. Then $N=I+\sum_{j=1}^4N_{jj}+\sum_{j=1}^5M_{jj}$ since $p=2$. Therefore $M_{gh}N=NM_{gh}=N_{rs}N=NN_{rs}=O+V_2$ and $N^2=N$ for any $g, h\in [1,5]$ and $r, s\in [1,4]$. So $\langle\{M+V_2: M\in H_{2,3}\}\rangle_\F$ and $\langle\{N\}\rangle_\F$ are two-sided ideals of $T/V_2$. So $\langle\{N\}\rangle_\F\cong M_1(\F)$ as $\F$-algebras. Moreover, there exists an $\F$-algebra isomorphism from $\langle\{M+V_2: M\in H_{2,3}\}\rangle_\F$ to $M_5(\F)$ that sends $M_{gh}$ to $E_{gh}(5)$ for any $g, h\in [1,5]$. So $\langle\{M+V_2: M\in H_{2,3}\}\rangle_\F\cong M_5(\F)$ as $\F$-algebras. The desired corollary thus follows from Lemma \ref{L;Lemma4.16}.
\end{proof}
In conclusion, the case $n\equiv 2\pmod p$ is solved by Corollaries \ref{C;Corollary4.17} and \ref{C;Corollary4.18}.
\section{Algebraic structure of $T$: Case III}
Recall that $G$ is assumed further to be an elementary abelian $2$-group. The aim of this section is to determine the algebraic structure of $T$ for the case $p\neq 2,3$ and $n\!\equiv \!4\pmod p$. Recall Notation \ref{N;Notation1.21} and the definition of $E_4^*TE_4^*$ in Subsection \ref{S;Subsection3}.

We first introduce the required notation and present some preliminary results.
\begin{nota}\label{N;Notation5.1}
Use $C_3$ to denote the set containing precisely
\begin{align*}
& E_4^*JE_4^*-E_4^*-E_4^*A_1E_4^*-E_4^*A_2E_3^*A_1E_4^*-E_4^*A_3E_2^*A_1E_4^*,\\
& E_4^*JE_4^*-E_4^*-E_4^*A_2E_4^*-E_4^*A_1E_3^*A_2E_4^*-E_4^*A_3E_1^*A_2E_4^*,\\
& E_4^*JE_4^*-E_4^*-E_4^*A_3E_4^*-E_4^*A_1E_2^*A_3E_4^*-E_4^*A_2E_1^*A_3E_4^*,\\
& E_4^*A_1E_2^*A_3E_4^*+E_4^*A_2E_1^*A_3E_4^*-E_4^*A_3E_1^*A_2E_4^*-E_4^*A_3E_2^*A_1E_4^*,\\
& E_4^*A_2E_3^*A_1E_4^*+E_4^*A_3E_2^*A_1E_4^*-E_4^*A_1E_2^*A_3E_4^*-E_4^*A_1E_3^*A_2E_4^*.
\end{align*}
Let $D_3$ be the set containing exactly $E_1^*JE_4^*-E_1^*A_2E_4^*-E_1^*A_3E_4^*-E_1^*A_2E_3^*A_1E_4^*$, $E_2^*JE_4^*\!-\!E_2^*A_1E_4^*\!-\!E_2^*A_3E_4^*\!-\!E_2^*A_1E_3^*A_2E_4^*, E_3^*JE_4^*\!-\!E_3^*A_1E_4^*\!-\!E_3^*A_2E_4^*\!-\!E_3^*A_1E_2^*A_3E_4^*$, and their transposes. Set $U_3=\langle C_3\rangle_\F\subseteq T$. Put $V_3=\langle C_3\cup D_3\rangle_\F\subseteq T$.
\end{nota}
\begin{lem}\label{L;Lemma5.2}
If $n>8$ or $p\neq 2$ and $n=8$, the union of $C_3$ and the set containing precisely $E_4^*JE_4^*$, $E_4^*$, $E_4^*A_1E_2^*A_3E_4^*$, $E_4^*A_1E_3^*A_2E_4^*$, $E_4^*A_2E_1^*A_3E_4^*$,
$E_4^*A_2E_3^*A_1E_4^*$ is an $\F$-basis of the $\F$-subalgebra $E_4^*TE_4^*$ of $T$. In particular, $C_3$ is an $\F$-basis of $U_3$.
\end{lem}
\begin{proof}
Let $C$ denote the union in the first statement. By \eqref{Eq;2} and \eqref{Eq;3}, notice that each element in $C$ is an $\F$-linear combination of the elements in $B_3\cup B_6$. By \eqref{Eq;2} and \eqref{Eq;3}, notice that each element in $B_3\cup B_6$ is also an $\F$-linear combination of the elements in $C$. As $|C|\!=\!|B_3\cup B_6|$ and Corollary \ref{C;Corollary2.17} holds, the first statement follows. The second statement is from the first one. The desired lemma is proved.
\end{proof}
The following results describe $E_4^*TE_4^*$ for the case $p\neq 2, 3$ and $n\equiv 4\pmod p$.
\begin{lem}\label{L;Lemma5.3}
If $n\equiv 4\pmod p$, $U_3$ is a two-sided ideal of the $\F$-subalgebra $E_4^*TE_4^*$ of $T$.
\end{lem}
\begin{proof}
Notice that $U_3\subseteq E_4^*TE_4^*$. If $n=4$, $U_3=\langle C_3\rangle_\F=\{O\}$ by \eqref{Eq;3} and the equalities in Lemma \ref{L;Lemma1.20}. Assume that $n>4$. If $M\in B_3\cup B_6$, then $M^T\in B_3\cup B_6$ by \eqref{Eq;1}. If $N\in U_3$, then $N^T\!\in\! U_3$ as $U_3\!=\!\langle C_3\rangle_\F$ and \eqref{Eq;1} holds. The desired lemma thus follows if $MN\in U_2$ for any $M\in B_3\cup B_6$ and $N\in C_3$. By combining \eqref{Eq;2}, \eqref{Eq;4}, the equalities in Lemmas \ref{L;Lemma2.13}, \ref{L;Lemma2.14}, \ref{L;Lemma2.15}, Theorem \ref{T;Insectionnumbers} (ii), (iii), the assumption $n\!\equiv\! 4\pmod p$, $MN\in U_2$ for any $B_3\cup B_6$ and $N\in C_3$. We are done.
\end{proof}
\begin{lem}\label{L;Lemma5.4}
If $n\equiv 4\pmod p$, the two-sided ideal $U_2$ of the $\F$-subalgebra $E_4^*TE_4^*$ of $T$ is nilpotent. In particular, $U_2\subseteq \mathrm{Rad}E_4^*TE_4^*$.
\end{lem}
\begin{proof}
As $U_3=\langle C_3\rangle_\F$ and Lemma \ref{L;Lemma5.3} holds, the first statement follows if the product of any three elements in $C_3$ is the zero matrix. If $n=4$, $C_3=\{O\}$ by \eqref{Eq;3} and the equalities in Lemma \ref{L;Lemma1.20}. Assume that $n>4$. By combining \eqref{Eq;2}, \eqref{Eq;4}, the equalities in Lemmas \ref{L;Lemma2.13}, \ref{L;Lemma2.14}, \ref{L;Lemma2.15}, Theorem \ref{T;Insectionnumbers} (ii), (iii), the assumption $n\equiv 4\pmod p$, $M_1M_2M_3=O$ for any $M_1, M_2, M_3\in C_3$. The first statement follows. The second statement is from the first one. The desired lemma is proved.
\end{proof}
If $a\in \mathbb{N}$ and $b,c\in [1,a]$, recall the definitions of $E_{bc}(a)$ and $\delta_{bc}$ in Subsection \ref{S;Subsection4}.
\begin{cor}\label{C;Corollary5.5}
If $p\neq 2, 3$ and $n\equiv4\pmod p$, $E_4^*TE_4^*/\mathrm{Rad}E_4^*TE_4^*$ is isomorphic to $M_2(\F)\oplus M_1(\F)\oplus M_1(\F)$ as $\F$-algebras and $\mathrm{Rad}E_4^*TE_4^*=U_3$.
\end{cor}
\begin{proof}
By Lemma \ref{L;Lemma5.4}, we just check that $E_4^*TE_4^*/U_3\cong M_2(\F)\oplus M_1(\F)\oplus M_1(\F)$ as $\F$-algebras. If $n=4$, set $C=\{M+U_3: M\in B_3\cup\{E_4^*A_1E_2^*A_3E_4^*\}\}$. If $n>4$, use $C$ to denote the set containing precisely $E_4^*JE_4^*+U_3$, $E_4^*+U_3$, $E_4^*A_1E_2^*A_3E_4^*+U_3$, $E_4^*A_1E_3^*A_2E_4^*+U_3$,
$E_4^*A_2E_1^*A_3E_4^*+U_3$, $E_4^*A_2E_3^*A_1E_4^*+U_3$. As $p\neq 2, 3$, define
\begin{align*}
& M_{11}=\overline{3}^{-1}(E_4^*JE_4^*-\overline{2}E_4^*A_1E_2^*A_3E_4^*-E_4^*A_1E_3^*A_2E_4^*)+U_3,\\
& M_{12}=\overline{3}^{-1}(E_4^*A_1E_2^*A_3E_4^*-E_4^*A_1E_3^*A_2E_4^*)+U_3,\\
& M_{21}=\overline{3}^{-1}(E_4^*A_2E_1^*A_3E_4^*-E_4^*A_2E_3^*A_1E_4^*)+U_3,\\
& M_{22}=\overline{3}^{-1}(E_4^*JE_4^*-\overline{2}E_4^*A_2E_1^*A_3E_4^*-E_4^*A_2E_3^*A_1E_4^*)+U_3.
\end{align*}
Put $M\!=\!\overline{6}^{-1}E_4^*JE_4^*+U_3$, $N\!=\!E_4^*-M-M_{11}-M_{22}$, $D\!=\!\{M_{ab}: a,b\in [1,2]\}\cup\{M, N\}$. If $n=4$, notice that $C_3=\{O\}$ by \eqref{Eq;3} and the equalities in Lemma \ref{L;Lemma1.20}. Hence $C$ is always an $\F$-basis of $E_4^*TE_4^*/U_3$ by Corollary \ref{C;Corollary2.17} and Lemma \ref{L;Lemma5.2}. By combining \eqref{Eq;2}, \eqref{Eq;3}, and Lemma \ref{L;Lemma1.20}, each element in $D$ is an $\F$-linear combination of the elements in $C$. Moreover, each element in $C$ is always an $\F$-linear combination of the elements in $D$. As $|C|=|D|$ and $C$ is an $\F$-basis of $E_4^*TE_4^*/U_3$, $D$ is an $\F$-basis of $E_4^*TE_4^*/U_3$. Hence $E_4^*TE_4^*/U_3=\langle\{M_{ab}: a,b \in [1,2]\}\rangle_\F\oplus\langle\{M\}\rangle_\F\oplus\langle\{N\}\rangle_\F$.

By combining \eqref{Eq;2}, \eqref{Eq;4}, the equalities in Lemma \ref{L;Lemma2.15}, Lemma \ref{L;Tripleproducts} (i), Theorem \ref{T;Insectionnumbers} (i), (ii), (iii), and the assumption $n\equiv 4\pmod p$, notice that $M_{gh}M_{rs}\!=\!\delta_{hr}M_{gs}$ for any $g, h, r, s\in [1,2]$. By combining \eqref{Eq;2}, \eqref{Eq;4}, Lemma \ref{L;Tripleproducts} (i), Theorem \ref{T;Insectionnumbers} (i), (ii), (iii), and the assumption $n\!\equiv \!4\pmod p$, notice that $M^2=M$, $N^2=N$, and $MN=NM=M_{gh}N=NM_{gh}=O$ for any $g, h\in [1,2]$. So $\langle\{M_{ab}: a, b\in [1,2]\}\rangle_\F$, $\langle\{M\}\rangle_\F$, $\langle\{N\}\rangle_\F$ are two-sided ideals of $T$. Furthermore, $\langle\{M\}\rangle_\F\!\cong\!\langle\{N\}\rangle_\F\!\cong\! M_1(\F)$ as $\F$-algebras. There exists an $\F$-algebra isomorphism from $\langle\{M_{ab}: a,b\in [1,2]\}\rangle_\F$ to $M_2(\F)$ that sends $M_{gh}$ to $E_{gh}(2)$ for any $g, h\!\in\! [1,2]$. The proof of the desired corollary is now complete.
\end{proof}
We next provide the following notation and some related properties.
\begin{nota}\label{N;Notation5.6}
Assume that $p\neq2, 3$. Let $H_{3,1}$ be the union of $\{E_0^*JE_a^*: a\in [0,4]\}$, $\{\overline{3}^{-1}E_a^*JE_b^*\!:a\!\in\! [1,3], b\!\in\! [0,4]\}$, $\{\overline{6}^{-1}E_4^*JE_a^*\!: a\!\in\! [0,4]\}$. Let $H_{3,2}$ denote the union of $\{E_a^*A_bE_c^*\!-\!\overline{3}^{-1}E_a^*JE_c^*\!: \{a, b, c\}\!\!=\!\![1,3]\}\cup\{E_a^*\!-\!\overline{3}^{-1}E_a^*JE_a^*: a\!\in \! [1,3]\}$, the set containing exactly \!$\overline{3}^{-1}(\overline{2}E_1^*A_3E_4^*\!+\!E_1^*A_2E_4^*\!-\!E_1^*JE_4^*), \overline{3}^{-1}(\overline{2}E_2^*A_3E_4^*\!+\!E_2^*A_1E_4^*\!-\!E_2^*JE_4^*),$ $\overline{3}^{-1}(\overline{2}E_1^*A_3E_4^*+E_1^*A_2E_3^*A_1E_4^*-E_1^*JE_4^*)$, $\overline{3}^{-1}(\overline{2}E_2^*A_3E_4^*+E_2^*A_1E_3^*A_2E_4^*-E_2^*JE_4^*)$,
$\overline{3}^{-1}(\overline{2}E_3^*A_1E_2^*A_3E_4^*+E_3^*A_2E_4^*-E_3^*JE_4^*)$, $\overline{3}^{-1}(\overline{2}E_3^*A_1E_2^*A_3E_4^*+E_3^*A_1E_4^*-E_3^*JE_4^*)$,
and the set containing exactly $\overline{3}^{-1}E_4^*JE_1^*-E_4^*A_1E_3^*A_2E_1^*$, $\overline{3}^{-1}E_4^*JE_2^*-E_4^*A_2E_3^*A_1E_2^*$, $\overline{3}^{-1}E_4^*JE_1^*\!-\!E_4^*A_2E_1^*, \overline{3}^{-1}E_4^*JE_2^*\!-\!E_4^*A_1E_2^*, \overline{3}^{-1}E_4^*JE_3^*\!-\!E_4^*A_1E_3^*, \overline{3}^{-1}E_4^*JE_3^*\!-\!E_4^*A_2E_3^*$, $\overline{3}^{-1}(E_4^*JE_4^*-\overline{2}E_4^*A_1E_2^*A_3E_4^*-E_4^*A_1E_3^*A_2E_4^*), \overline{3}^{-1}(E_4^*A_1E_3^*A_2E_4^*-E_4^*A_1E_2^*A_3E_4^*)$,
$\overline{3}^{-1}(E_4^*JE_4^*-\overline{2}E_4^*A_2E_1^*A_3E_4^*-E_4^*A_2E_3^*A_1E_4^*), \overline{3}^{-1}(E_4^*A_2E_3^*A_1E_4^*-E_4^*A_2E_1^*A_3E_4^*)$.
\end{nota}
\begin{lem}\label{L;Lemma5.7}
Assume that $p\neq 2, 3$ and $n\equiv 4\pmod p$. If $n>4$, the set containing precisely $\overline{6}^{-1}(\overline{2}E_4^*-E_4^*JE_4^*-\overline{4}E_4^*A_3E_4^*+\overline{2}E_4^*A_1E_3^*A_2E_4^*+\overline{2}E_4^*A_2E_3^*A_1E_4^*)$ and the elements in $C_3\cup D_3\cup H_{3,1}\cup H_{3,2}$ is an $\F$-basis of $T$. In particular, $C_3\cup D_3$ is an $\F$-basis of $V_3$.
\end{lem}
\begin{proof}
Let $C$ be the set containing precisely the elements in $C_3\cup D_3\cup H_{3,1}\cup H_{3,2}$ and $\overline{6}^{-1}(\overline{2}E_4^*-E_4^*JE_4^*-\overline{4}E_4^*A_3E_4^*+\overline{2}E_4^*A_1E_3^*A_2E_4^*+\overline{2}E_4^*A_2E_3^*A_1E_4^*)$. By Corollary \ref{C;Corollary2.16}, each element in $C$ is an $\F$-linear combination of the elements in $B$. As $p\neq 2$ and $n>4$, $B=B_1\cup B_2\cup B_5\cup B_6$. By combining \eqref{Eq;2}, \eqref{Eq;3}, Lemma \ref{L;Tripleproducts} (iii), Theorem \ref{T;Insectionnumbers} (i), (ii), and (iii), each element in $B$ is an $\F$-linear combination of the elements in $C$. As $|B|=|C|$ and Corollary \ref{C;Corollary2.16} holds, the first statement follows. As $V_3\!=\!\langle C_3\cup D_3\rangle_\F$, the second statement is from the first one. We are done.
\end{proof}
\begin{lem}\label{L;Lemma5.8}
Assume that $p\neq 2, 3$ and $n\equiv 4\pmod p$. If $n=4$, the set containing precisely $\overline{6}^{-1}(\overline{2}E_4^*-E_4^*JE_4^*-\overline{4}E_4^*A_3E_4^*+\overline{2}E_4^*A_1E_3^*A_2E_4^*+\overline{2}E_4^*A_2E_3^*A_1E_4^*)$ and the elements in $H_{3,1}\cup H_{3,2}$ is an $\F$-basis of $T$. Moreover, $V_3=\{O\}$.
\end{lem}
\begin{proof}
Use $C$ to denote the set containing precisely the elements in $H_{3,1}\cup H_{3,2}$ and
$\overline{6}^{-1}(\overline{2}E_4^*-E_4^*JE_4^*-\overline{4}E_4^*A_3E_4^*+\overline{2}E_4^*A_1E_3^*A_2E_4^*+\overline{2}E_4^*A_2E_3^*A_1E_4^*)$. By Corollary \ref{C;Corollary2.16},
each element in $C$ is an $\F$-linear combination of the elements in $B$. As $n=4$, notice that $B=B_1\cup B_2\cup\{E_4^*A_1E_2^*A_3E_4^*\}$. By combining \eqref{Eq;2}, \eqref{Eq;3}, the equalities in Lemmas \ref{L;Lemma1.5}, \ref{L;Lemma1.20}, Lemma \ref{L;Tripleproducts} (iii), Theorem \ref{T;Insectionnumbers} (i), (ii), (iii), notice that $V_3=\{O\}$ and each element in $B$ is an $\F$-linear combination of the elements in $C$. Hence $C$ is an $\F$-basis of $T$ by the equality $|B|=|C|$ and Corollary \ref{C;Corollary2.16}. We are done.
\end{proof}
The following results describe $T$ for the case $p\neq 2, 3$ and $n\equiv 4\pmod p$.
\begin{lem}\label{L;Lemma5.9}
Assume that $n\equiv4\pmod p$. Then $V_3$ is a two-sided ideal of $T$.
\end{lem}
\begin{proof}
Notice that $V_3\subseteq T$. If $n=4$, then the desired lemma is from Lemma \ref{L;Lemma5.8}. Assume that $n>4$ in this proof. By the definition of $T$, notice that $M^T\in T$ if $M\in T$. If $N\in V_3$, then $N^T\in V_3$ as $V_3=\langle C_3\cup D_3\rangle_\F$ and \eqref{Eq;1} holds. By Lemma \ref{L;Lemma3.10}, the desired lemma is proved if $MN\in V_3$ for any $M\in B_1\cup B_2$ and $N\in C_3\cup D_3$. By combining \eqref{Eq;2}, Lemma \ref{L;Tripleproducts} (i), the equalities in Lemmas \ref{L;Lemma2.3}, \ref{L;Lemma2.5}, \ref{L;Lemma2.6}, \ref{L;Lemma2.7}, \ref{L;Lemma2.8}, \ref{L;Lemma2.9}, \ref{L;Lemma2.10}, \ref{L;Lemma2.11}, \ref{L;Lemma2.12}, \ref{L;Lemma2.13}, \ref{L;Lemma2.14}, Theorem \ref{T;Insectionnumbers} (i), (ii), (iii), and the assumption $n\equiv 4\pmod p$, note that $MN\!\in\! V_3$ for any $M\!\in\! B_1\cup B_2$ and $N\!\in\! C_3\cup D_3$. We are done. %If $M\in B_1$ and $N\in C_3\cup D_3$,
\end{proof}
\begin{lem}\label{L;Lemma5.10}
Assume that $n\equiv 4\pmod p$. Then $V_3\subseteq\mathrm{Rad}T$. In particular, $V_3$ is a nilpotent two-sided ideal of $T$.
\end{lem}
\begin{proof}
The case $n=4$ is from Lemma \ref{L;Lemma5.8}. Assume that $n>4$ in this proof. As $V_3=\langle C_3\cup D_3\rangle_\F$, the first statement follows if $C_3\cup D_3\subseteq \mathrm{Rad}T$. By Lemmas \ref{L;Lemma5.4} and \ref{L;Radical}, $C_3\subseteq U_3\subseteq\mathrm{Rad}E_4^*TE_4^*\subseteq \mathrm{Rad}T$. It is enough to check that $D_3\subseteq \mathrm{Rad}T$. Set $M=E_4^*A_3E_1^*A_2E_4^*\!+\!E_4^*A_3E_2^*A_1E_4^*\!-\!E_4^*A_1E_2^*A_3E_4^*\!-\!E_4^*A_2E_1^*A_3E_4^*\in C_3$. Notice that $D_3$ contains exactly $E_1^*A_2E_4^*M$, $E_2^*A_1E_4^*M$, $E_3^*A_1E_4^*M$ and their transposes by combining \eqref{Eq;2}, \eqref{Eq;3}, the equalities in Lemmas \ref{L;Lemma2.12}, \ref{L;Lemma2.4}, Lemma \ref{L;Tripleproducts} (iii), and the
assumption $n\equiv 4\pmod p$. So $D_3\subseteq\mathrm{Rad}T$ as $N^T\in \mathrm{Rad}T$ if $N\in \mathrm{Rad}T$. The first statement follows.
The second statement is from Lemma \ref{L;Lemma5.9} and the first one.
\end{proof}
If $a\in \mathbb{N}$ and $b,c\in [1,a]$, recall the definitions of $E_{bc}(a)$ and $\delta_{bc}$ in Subsection \ref{S;Subsection4}.
\begin{lem}\label{L;Lemma5.11}
Assume that $p\!\neq\! 2, 3$ and $n\!\equiv\! 4\pmod p$. Then $\langle\{M\!+\!V_3: M\!\in\! H_{3,1}\}\rangle_\F$ is a two-sided ideal of the $\F$-algebra $T/V_3$ that is isomorphic to $M_5(\F)$ as $\F$-algebras.
\end{lem}
\begin{proof}
As $p\neq 2, 3$, define $M_{gh}=\overline{3}^{-1}E_{g-1}^*JE_{h-1}^*+V_3$, $M_{5h}=\overline{6}^{-1}E_4^*JE_{h-1}^*+V_3$, and $M_{1h}=E_0^*JE_{h-1}^*+V_3$ for any $g\in [2,4]$ and $h\in [1,5]$. Hence $\langle\{M+V_3: M\in H_{3,1}\}\rangle_\F$ has an $\F$-basis $\{M_{ab}: a,b\in [1,5]\}$ by Lemmas \ref{L;Lemma5.7} and \ref{L;Lemma5.8}. Therefore $T/V_3$ has a two-sided ideal $\langle\{M\!+\!V_3: M\in H_{3,1}\}\rangle_\F$ by Lemma \ref{L;Tripleproducts} (i). By combining \eqref{Eq;2}, \eqref{Eq;4}, Theorem \ref{T;Insectionnumbers} (i), (ii), and the assumption $n\equiv 4\pmod p$, $M_{gh}M_{rs}=\delta_{hr}M_{gs}$ for any $g, h,r,s\in [1,5]$. So there is an $\F$-algebra isomorphism from $\langle\{M\!+\!V_3: M\!\in \!H_{3,1}\}\rangle_\F$ to $M_5(\F)$ that sends $M_{gh}$ to $E_{gh}(5)$ for any $g, h\in [1,5]$. We are done.
\end{proof}
We are now ready to close this section by presenting the main result of this section.
\begin{cor}\label{C;Corollary5.12}
If $p\!\neq\!2, 3$, $n\!\equiv\!4\pmod p$, then $T/\mathrm{Rad}T\!\cong\! M_5(\F)\oplus M_5(\F)\oplus M_1(\F)$ as $\F$-algebras and $\mathrm{Rad}T=V_3$.
\end{cor}
\begin{proof}
By Lemma \ref{L;Lemma5.10}, it suffices to check that $T/V_3\cong M_5(\F)\oplus M_5(\F)\oplus M_1(\F)$ as $\F$-algebras. As $p\neq 3$, define
$M_{gi}\!\!=\!\!E_g^*A_hE_i^*\!-\!\overline{3}^{-1}E_g^*JE_i^*\!+\!V_3$, $M_{jj}\!=\!E_j^*\!-\!\overline{3}^{-1}E_j^*JE_j^*+V_3$,
$M_{14}\!=\!\overline{3}^{-1}(\overline{2}E_1^*A_3E_4^*+E_1^*A_2E_4^*-E_1^*JE_4^*)+V_3$, $M_{41}\!=\!\overline{3}^{-1}E_4^*JE_1^*\!-\!E_4^*A_1E_3^*A_2E_1^*\!+\!V_3$,
$M_{15}\!=\!\overline{3}^{-1}(\overline{2}E_1^*A_3E_4^*\!+\!E_1^*A_2E_3^*A_1E_4^*\!-\!E_1^*JE_4^*)\!+\!V_3$, $M_{51}=\overline{3}^{-1}E_4^*JE_1^*-E_4^*A_2E_1^*+V_3$,
$M_{24}=\overline{3}^{-1}(\overline{2}E_2^*A_3E_4^*+E_2^*A_1E_3^*A_2E_4^*-E_2^*JE_4^*)+V_3$, $M_{42}=\overline{3}^{-1}E_4^*JE_2^*-E_4^*A_1E_2^*+V_3$,
$M_{25}=\overline{3}^{-1}(\overline{2}E_2^*A_3E_4^*+E_2^*A_1E_4^*-E_2^*JE_4^*)+V_3$, $M_{52}=\overline{3}^{-1}E_4^*JE_2^*-E_4^*A_2E_3^*A_1E_2^*+V_3$,
$M_{34}=\overline{3}^{-1}(\overline{2}E_3^*A_1E_2^*A_3E_4^*+E_3^*A_2E_4^*-E_3^*JE_4^*)+V_3$, $M_{43}=\overline{3}^{-1}E_4^*JE_3^*-E_4^*A_1E_3^*+V_3$,
$M_{35}=\overline{3}^{-1}(\overline{2}E_3^*A_1E_2^*A_3E_4^*+E_3^*A_1E_4^*-E_3^*JE_4^*)+V_3$, $M_{53}=\overline{3}^{-1}E_4^*JE_3^*-E_4^*A_2E_3^*+V_3$ if $j\in \{g, h,i\}=[1,3]$.
Use the assumption $p\neq 2,3$ to continue defining
\begin{align*}
& M_{44}=\overline{3}^{-1}(E_4^*JE_4^*-\overline{2}E_4^*A_1E_2^*A_3E_4^*-E_4^*A_1E_3^*A_2E_4^*)+V_3,\\
& M_{45}=\overline{3}^{-1}(E_4^*A_1E_3^*A_2E_4^*-E_4^*A_1E_2^*A_3E_4^*)+V_3,\\
& M_{54}=\overline{3}^{-1}(E_4^*A_2E_3^*A_1E_4^*-E_4^*A_2E_1^*A_3E_4^*)+V_3,\\
& M_{55}=\overline{3}^{-1}(E_4^*JE_4^*-\overline{2}E_4^*A_2E_1^*A_3E_4^*-E_4^*A_2E_3^*A_1E_4^*)+V_3.
\end{align*}
Moreover, $N_{gh}=\overline{3}^{-1}E_{g-1}^*JE_{h-1}^*+V_3, N_{5h}=\overline{6}^{-1}E_4^*JE_{h-1}^*+V_3, N_{1h}=E_0^*JE_{h-1}^*+V_3$, $N\!=\!I-\sum_{j=1}^5N_{jj}-\sum_{j=1}^5M_{jj}$ if $g\in [2,4]$ and $h\in [1,5]$. By Lemmas \ref{L;Lemma5.7} and \ref{L;Lemma5.8}, $T/V_3=\langle\{M_{ab}: a,b\in [1,5]\}\rangle_\F\oplus\langle\{N_{ab}: a,b\in [1,5]\}\rangle_\F\oplus \langle\{N\}\rangle_\F$. By combining \eqref{Eq;2}, \eqref{Eq;3}, \eqref{Eq;4}, the equalities in Lemmas \ref{L;Lemma2.1}, \ref{L;Lemma2.2}, \ref{L;Lemma2.3}, \ref{L;Lemma2.4}, \ref{L;Lemma2.5}, \ref{L;Lemma2.6}, \ref{L;Lemma2.7}, \ref{L;Lemma2.9}, \ref{L;Lemma2.11}, \ref{L;Lemma2.12}, \ref{L;Lemma2.15}, Lemma \ref{L;Tripleproducts} (i), (iii), Theorem \ref{T;Insectionnumbers} (i), (ii), (iii), and the assumption $n\equiv 4\pmod p$, note that $M_{gh}M_{rs}\!=\!\delta_{hr}M_{gs}$, $M_{gh}N_{rs}\!=\!N_{rs}M_{gh}\!=\!O+V_3$, $N_{gh}N_{rs}\!=\!\delta_{hr}N_{gs}$, $M_{gh}N\!=\!NM_{gh}\!=\!N_{gh}N\!=\!NN_{gh}=O+V_3$, and $N^2\!=\!N$ for any $g,h,r,s\in [1,5]$. So $T/V_3$ has two-sided ideals $\langle\{M_{ab}: a,b\in [1,5]\}\rangle_\F$ and $\langle\{N\}\rangle_\F$. Also notice that $\langle\{N\}\rangle_\F\cong M_1(\F)$ as $\F$-algebras. Moreover, there is an $\F$-algebra isomorphism from $\langle\{M_{ab}: a,b\in [1,5]\}\rangle_\F$ to $M_5(\F)$ that sends $M_{gh}$ to $E_{gh}(5)$ for any $g, h\in [1,5]$. The desired corollary thus follows from Lemma \ref{L;Lemma5.11}.
\end{proof}
\section{Algebraic structure of $T$: Case IV}
Recall that $G$ is assumed further to be an elementary abelian $2$-group. The aim of this section is to determine the algebraic structure of $T$ for the remaining case. Some examples and counterexamples for the study of the Terwilliger $\F$-algebras of schemes are also presented. For this aim, recall Notations \ref{N;Notation1.21}, \ref{N;Notation3.1}, \ref{N;Notation4.1}, \ref{N;Notation5.1}, and the definition of $E_4^*TE_4^*$ in Subsection \ref{S;Subsection3}. We first list a notation and some corollaries.%For this aim, recall Notation \ref{N;Notation1.21} and the definition of $E_4^*TE_4^*$ in Subsection \ref{S;Subsection3}. We first list the following notation and some corollaries.
\begin{nota}\label{N;Notation6.1}
Set $c_a=\overline{n-a}$, $c_4=\overline{(n-1)(n-2)}$, $c_5=\overline{(n-1)(n-4)}$ if $a\in [1,3]$.
\end{nota}
If $a\in \mathbb{N}$ and $b,c\in [1,a]$, recall the definitions of $E_{bc}(a)$ and $\delta_{bc}$ in Subsection \ref{S;Subsection4}.
\begin{cor}\label{C;Corollary6.2}
If $n\not\equiv 1, 2, 4\pmod p$, then $E_4^*TE_4^*\cong M_3(\F)\oplus M_1(\F)\oplus M_1(\F)$ as $\F$-algebras.
\end{cor}
\begin{proof}
As $n\not\equiv 1,2,4\pmod p$, notice that $n>8$ or $p\neq 2$ and $n=8$. So $E_4^*TE_4^*$ has an $\F$-basis $B_3\cup B_6$ by Corollary \ref{C;Corollary2.17}. By the assumption $n\!\not\equiv\! 1, 2, 4 \pmod p$, define
\begin{align*}
& M_{11}=c_5^{-1}(E_4^*A_2E_3^*A_1E_4^*+E_4^*A_3E_2^*A_1E_4^*+c_3(E_4^*+E_4^*A_1E_4^*)-E_4^*JE_4^*),\\
& M_{12}=c_5^{-1}(E_4^*+E_4^*A_2E_4^*+E_4^*A_3E_1^*A_2E_4^*+c_3E_4^*A_1E_3^*A_2E_4^*-E_4^*JE_4^*),\\
& M_{13}=c_5^{-1}(E_4^*+E_4^*A_3E_4^*+E_4^*A_2E_1^*A_3E_4^*+c_3E_4^*A_1E_2^*A_3E_4^*-E_4^*JE_4^*),\\
& M_{21}=c_5^{-1}(E_4^*+E_4^*A_1E_4^*+E_4^*A_3E_2^*A_1E_4^*+c_3E_4^*A_2E_3^*A_1E_4^*-E_4^*JE_4^*),\\
& M_{22}=c_5^{-1}(E_4^*A_1E_3^*A_2E_4^*+E_4^*A_3E_1^*A_2E_4^*+c_3(E_4^*+E_4^*A_2E_4^*)-E_4^*JE_4^*),\\
& M_{23}=c_5^{-1}(E_4^*+E_4^*A_3E_4^*+E_4^*A_1E_2^*A_3E_4^*+c_3E_4^*A_2E_1^*A_3E_4^*-E_4^*JE_4^*),\\
& M_{31}=c_5^{-1}(E_4^*+E_4^*A_1E_4^*+E_4^*A_2E_3^*A_1E_4^*+c_3E_4^*A_3E_2^*A_1E_4^*-E_4^*JE_4^*),\\
& M_{32}=c_5^{-1}(E_4^*+E_4^*A_2E_4^*+E_4^*A_1E_3^*A_2E_4^*+c_3E_4^*A_3E_1^*A_2E_4^*-E_4^*JE_4^*),\\
& M_{33}=c_5^{-1}(E_4^*A_1E_2^*A_3E_4^*+E_4^*A_2E_1^*A_3E_4^*+c_3(E_4^*+E_4^*A_3E_4^*)-E_4^*JE_4^*),
\end{align*}
$M=c_4^{-1}E_4^*JE_4^*$, $N=E_4^*-M-\sum_{j=1}^3M_{jj}$, and $C=\{M_{ab}: a,b\in [1,3]\}\cup\{M, N\}$. By Corollary \ref{C;Corollary2.17}, each element in $C$ is an $\F$-linear combination of the elements in $B_3\cup B_6$. By \eqref{Eq;2} and \eqref{Eq;3}, each element in $B_3\cup B_6$ is an $\F$-linear combination of the elements in $C$. As $|B_3\cup B_6|=|C|$ and Corollary \ref{C;Corollary2.17} holds, notice that $C$ is an $\F$-basis of $E_4^*TE_4^*$. Therefore $E_4^*TE_4^*=\langle\{M_{ab}: a,b\in [1,3]\}\rangle_\F\oplus\langle\{M\}\rangle_\F\oplus\langle\{N\}\rangle_\F$. By combining \eqref{Eq;2}, \eqref{Eq;4}, the equalities in Lemmas \ref{L;Lemma2.13}, \ref{L;Lemma2.14}, \ref{L;Lemma2.15}, Lemma \ref{L;Tripleproducts} (i), Theorem \ref{T;Insectionnumbers} (i), (ii), and (iii), notice that $M^2=M$, $N^2=N$, $M_{gh}M_{rs}=\delta_{hr}M_{gs}$, and $M_{gh}M=MM_{gh}=M_{gh}N=NM_{gh}=MN=NM=O$. Therefore $E_4^*TE_4^*$ has two-sided ideals $\langle\{M_{ab}: a,b\in [1,3]\}\rangle_\F$, $\langle\{M\}\rangle_\F$, and $\langle\{N\}\rangle_\F$.
Moreover, notice that $\langle \{M\}\rangle_\F\cong\langle\{N\}\rangle_\F\cong M_1(\F)$ as $\F$-algebras. There is an $\F$-algebra isomorphism from $\langle\{M_{ab}: a,b\in [1,3]\}\rangle_\F$ to $M_3(\F)$ that sends $M_{gh}$ to $E_{gh}(3)$ for any $g, h\in [1,3]$. The proof of the desired corollary is now complete.
\end{proof}
\begin{cor}\label{C;Corollary6.3}
The $\F$-subalgebra $E_4^*TE_4^*$ of $T$ is semisimple if and only if $p\neq 2, 3$ and $n=4$ or $n\not\equiv 1,2,4\pmod p$.
\end{cor}
\begin{proof}
Assume that $p\!\neq\! 2, 3$ and $n\!\equiv\! 4\pmod p$. By combining Corollaries \ref{C;Corollary5.5}, \ref{C;Corollary2.17}, Lemma \ref{L;Lemma1.20}, and \eqref{Eq;3}, $E_4^*TE_4^*$ is semisimple if and only if $n=4$. The remaining case is solved by combining Corollaries \ref{C;Corollary3.6}, \ref{C;Corollary4.8}, \ref{C;Corollary4.9}, \ref{C;Corollary6.2}, and \eqref{Eq;4}. We are done.
\end{proof}
The following corollary completes the analysis of the algebraic structure of $T$.
\begin{cor}\label{C;Corollary6.4}
If $n\not\equiv1,2, 4\pmod p$, $T\cong M_6(\F)\oplus M_5(\F)\oplus M_1(\F)$ as $\F$-algebras.
\end{cor}
\begin{proof}
As $n\not\equiv 1,2,4\pmod p$, $n>8$ or $p\neq 2$ and $n=8$. Hence $T$ has an $\F$-basis $B_1\cup B_2\cup B_5\cup B_6$ by Corollary \ref{C;Corollary2.16}. As $n\!\not\equiv\!1\pmod p$, set $M_{jj}\!=\!E_j^*-c_1^{-1}E_j^*JE_j^*$, $M_{gi}\!\!=\!\!E_g^*A_hE_i^*\!\!-\!\!c_1^{-1}E_g^*JE_i^*, M_{14}\!\!=\!\!E_1^*A_2E_3^*A_1E_4^*\!\!-\!\!c_1^{-1}E_1^*JE_4^*, M_{15}\!=\!E_1^*A_2E_4^*\!-\!c_1^{-1}E_1^*JE_4^*$, $M_{16}\!\!=\!\!E_1^*A_3E_4^*\!\!-\!\!c_1^{-1}E_1^*JE_4^*, M_{24}\!\!=\!\!E_2^*A_1E_4^*\!\!-\!\!c_1^{-1}E_2^*JE_4^*,
M_{25}\!\!=\!\!E_2^*A_1E_3^*A_2E_4^*\!\!-\!\!c_1^{-1}E_2^*JE_4^*$, $M_{26}\!\!=\!\!E_2^*A_3E_4^*\!\!-\!\!c_1^{-1}E_2^*JE_4^*$, $M_{34}\!\!=\!\!E_3^*A_1E_4^*\!\!-\!\!c_1^{-1}E_3^*JE_4^*$,
$M_{35}\!\!=\!\!E_3^*A_2E_4^*\!\!-\!\!c_1^{-1}E_3^*JE_4^*$, and $M_{36}\!=\!E_3^*A_1E_2^*A_3E_4^*-c_1^{-1}E_3^*JE_4^*$ if $j\in\{g, h, i\}=[1,3]$. Use the given assumption
$n\not\equiv 1, 2, 4\pmod p$ to continue defining
\begin{align*}
& M_{41}=c_5^{-1}(E_4^*A_2E_1^*+E_4^*A_3E_1^*+c_3E_4^*A_1E_3^*A_2E_1^*-E_4^*JE_1^*),\\
& M_{42}=c_5^{-1}(E_4^*A_3E_2^*+E_4^*A_2E_3^*A_1E_2^*+c_3E_4^*A_1E_2^*-E_4^*JE_2^*),\\
& M_{43}=c_5^{-1}(E_4^*A_2E_3^*+E_4^*A_3E_2^*A_1E_3^*+c_3E_4^*A_1E_3^*-E_4^*JE_3^*),\\
& M_{44}=c_5^{-1}(E_4^*A_2E_3^*A_1E_4^*+E_4^*A_3E_2^*A_1E_4^*+c_3(E_4^*+E_4^*A_1E_4^*)-E_4^*JE_4^*),\\
& M_{45}=c_5^{-1}(E_4^*+E_4^*A_2E_4^*+E_4^*A_3E_1^*A_2E_4^*+c_3E_4^*A_1E_3^*A_2E_4^*-E_4^*JE_4^*),\\
& M_{46}=c_5^{-1}(E_4^*+E_4^*A_3E_4^*+E_4^*A_2E_1^*A_3E_4^*+c_3E_4^*A_1E_2^*A_3E_4^*-E_4^*JE_4^*),\\
& M_{51}=c_5^{-1}(E_4^*A_3E_1^*+E_4^*A_1E_3^*A_2E_1^*+c_3E_4^*A_2E_1^*-E_4^*JE_1^*),\\
& M_{52}=c_5^{-1}(E_4^*A_1E_2^*+E_4^*A_3E_2^*+c_3E_4^*A_2E_3^*A_1E_2^*-E_4^*JE_2^*),\\
& M_{53}=c_5^{-1}(E_4^*A_1E_3^*+E_4^*A_3E_2^*A_1E_3^*+c_3E_4^*A_2E_3^*-E_4^*JE_3^*),\\
& M_{54}=c_5^{-1}(E_4^*+E_4^*A_1E_4^*+E_4^*A_3E_2^*A_1E_4^*+c_3E_4^*A_2E_3^*A_1E_4^*-E_4^*JE_4^*),\\
& M_{55}=c_5^{-1}(E_4^*A_1E_3^*A_2E_4^*+E_4^*A_3E_1^*A_2E_4^*+c_3(E_4^*+E_4^*A_2E_4^*)-E_4^*JE_4^*),\\
& M_{56}=c_5^{-1}(E_4^*+E_4^*A_3E_4^*+E_4^*A_1E_2^*A_3E_4^*+c_3E_4^*A_2E_1^*A_3E_4^*-E_4^*JE_4^*),\\
& M_{61}=c_5^{-1}(E_4^*A_2E_1^*+E_4^*A_1E_3^*A_2E_1^*+c_3E_4^*A_3E_1^*-E_4^*JE_1^*),\\
& M_{62}=c_5^{-1}(E_4^*A_1E_2^*+E_4^*A_2E_3^*A_1E_2^*+c_3E_4^*A_3E_2^*-E_4^*JE_2^*),\\
& M_{63}=c_5^{-1}(E_4^*A_1E_3^*+E_4^*A_2E_3^*+c_3E_4^*A_3E_2^*A_1E_3^*-E_4^*JE_3^*),\\
& M_{64}=c_5^{-1}(E_4^*+E_4^*A_1E_4^*+E_4^*A_2E_3^*A_1E_4^*+c_3E_4^*A_3E_2^*A_1E_4^*-E_4^*JE_4^*),\\
& M_{65}=c_5^{-1}(E_4^*+E_4^*A_2E_4^*+E_4^*A_1E_3^*A_2E_4^*+c_3E_4^*A_3E_1^*A_2E_4^*-E_4^*JE_4^*),\\
& M_{66}=c_5^{-1}(E_4^*A_1E_2^*A_3E_4^*+E_4^*A_2E_1^*A_3E_4^*+c_3(E_4^*+E_4^*A_3E_4^*)-E_4^*JE_4^*),
\end{align*}
$N_{1h}=E_0^*JE_{h-1}^*$, $N_{15}=c_2^{-1}E_0^*JE_4^*$, $N_{gh}=c_1^{-1}E_{g-1}^*JE_{h-1}^*$, $N_{g5}=c_4^{-1}E_{g-1}^*JE_4^*$ for any $g\in [2, 5]$ and $h\in [1,4]$. Define
$N=I-\sum_{j=1}^5N_{jj}-\sum_{j=1}^6M_{jj}$. Use $C$ to denote $\{M_{ab}: a, b\in [1,6]\}\cup\{N_{ab}: a, b\in [1,5]\}\cup\{N\}$. By Corollary \ref{C;Corollary2.16}, each element in $C$ is an $\F$-linear combination of the elements in $B_1\cup B_2\cup B_5\cup B_6$. By combining \eqref{Eq;2}, \eqref{Eq;3}, Lemma \ref{L;Tripleproducts} (iii), Theorem \ref{T;Insectionnumbers}
(i), (ii), and (iii), notice that each element in $B_1\cup B_2\cup B_5\cup B_6$ is an $\F$-linear combination of the elements in $C$. As Corollary \ref{C;Corollary2.16} holds and $|B_1\cup B_2\cup B_5\cup B_6|=|C|$, notice that $C$ is an $\F$-basis of $T$. Therefore $T\!=\!\langle\{M_{ab}: a,b\in [1,6]\}\rangle_\F\oplus\langle\{N_{ab}: a,
b\in [1,5]\}\rangle_\F\oplus\langle\{N\}\rangle_\F$.

By combining \eqref{Eq;2}, \eqref{Eq;3}, \eqref{Eq;4}, the equalities in Lemmas \ref{L;Lemma2.1}, \ref{L;Lemma2.2}, \ref{L;Lemma2.3}, \ref{L;Lemma2.4}, \ref{L;Lemma2.5}, \ref{L;Lemma2.6},
\ref{L;Lemma2.7}, \ref{L;Lemma2.8}, \ref{L;Lemma2.9}, \ref{L;Lemma2.10}, \ref{L;Lemma2.11}, \ref{L;Lemma2.12}, \ref{L;Lemma2.13}, \ref{L;Lemma2.14}, \ref{L;Lemma2.15}, Lemma \ref{L;Tripleproducts} (i), (iii), Theorem \ref{T;Insectionnumbers} (i), (ii), (iii), notice that $M_{gh}M_{rs}=\delta_{hr}M_{gs}$ for any $g, h, r,s\in [1,6]$. By combining \eqref{Eq;2},
\eqref{Eq;4}, Lemma \ref{L;Tripleproducts} (i), Theorem \ref{T;Insectionnumbers} (i), (ii), (iii), $M_{gh}N_{rs}\!=\!N_{rs}M_{gh}=O$ and $N_{rs}N_{uv}\!=\!\delta_{su}N_{rv}$ for
any $g, h\in [1,6]$ and $r,s, u,v\in [1,5]$. Therefore $N^2=N$ and $M_{gh}N=NM_{gh}=N_{rs}N=NN_{rs}=O$ for any $g, h\in [1,6]$ and $r, s\in [1,5]$. So $T$ has two-sided ideals $\langle\{M_{ab}: a,b\in [1,6]\}\rangle_\F$, $\langle\{N_{ab}: a, b \in [1,5]\}\rangle_\F$, $\langle\{N\}\rangle_\F$. Moreover, $\langle\{N\}\rangle_\F\cong M_1(\F)$ as $\F$-algebras.
Therefore there is an $\F$-algebra isomorphism from $\langle\{ M_{ab}: a,b \in [1,6]\}\rangle_\F$ to $M_6(\F)$ that sends $M_{gh}$ to $E_{gh}(6)$ for any $g, h\in [1,6]$. There is also an $\F$-algebra isomorphism from $\langle\{N_{ab}: a,b \in [1,5]\}\rangle_\F$ to $M_5(\F)$ that sends $N_{gh}$ to $E_{gh}(5)$ for any $g, h\in [1,5]$. The desired corollary thus follows.
\end{proof}
We are now ready to deduce the main result of this paper.
\begin{thm}\label{T;Coretheorem}
As $\F$-algebras, we have
\[T/\mathrm{Rad}T\cong\begin{cases}
M_4(\F)\oplus M_1(\F)\oplus M_1(\F),& \text{if}\ n\equiv 1\pmod p,\\
M_6(\F)\oplus M_4(\F)\oplus M_1(\F),& \text{if}\ p\neq 2\ \text{and}\ n\equiv 2\pmod p,\\
M_5(\F)\oplus M_4(\F)\oplus M_1(\F),& \text{if}\ p=2,\\
M_5(\F)\oplus M_5(\F)\oplus M_1(\F),& \text{if}\ p\neq 2, 3\ \text{and}\ n\equiv4\pmod p,\\
M_6(\F)\oplus M_5(\F)\oplus M_1(\F),& \text{if}\ n\not\equiv 1,2,4\pmod p,
\end{cases}\]where
\[\mathrm{Rad}T=\begin{cases}
V_1, & \text{if}\ n\equiv 1\pmod p,\\
V_2, & \text{if}\ n\equiv 2\pmod p,\\
V_3, & \text{if}\ p\neq 2, 3\ \text{and}\ n\equiv4\pmod p,\\
\{O\}, & \text{if}\ n\not\equiv 1,2,4\pmod p.
\end{cases}\]
\end{thm}
\begin{proof}
The theorem is proved by combining Corollaries \ref{C;Corollary3.14}, \ref{C;Corollary4.17}, \ref{C;Corollary4.18}, \ref{C;Corollary5.12}, \ref{C;Corollary6.4}.
\end{proof}
A direct corollary of Theorem \ref{T;Coretheorem} is also presented as follows.
\begin{cor}\label{C;Corollary6.5}
The $\F$-algebra $T$ is semisimple if and only if $p\neq 2, 3$ and $n=4$ or $n\not\equiv 1,2,4\pmod p$.
\end{cor}
\begin{proof}
By Corollary \ref{C;Corollary6.3} and Lemma \ref{L;Radical}, notice that $T$ is semisimple only if $p\neq 2,3$ and $n=4$ or $n\not\equiv 1,2,4\pmod p$. If $p\neq 2,3$ and $n=4$ or $n\not\equiv 1,2,4\pmod p$, $T$ is semisimple by Lemma \ref{L;Lemma5.8} and Theorem \ref{T;Coretheorem}. The desired corollary thus follows.
\end{proof}
We conclude the whole paper with two additional remarks of Theorem \ref{T;Coretheorem}.
\begin{rem}\label{R;Coreremark1}
By \cite[Theorem 3.4]{Han}, a Terwilliger $\F$-algebra of a scheme is semisimple only if the valencies of all elements in this scheme are not divisible by $p$. The inverse of this statement is not true. In particular, a nonsymmetric scheme counterexample is already known (see \cite[5.1]{Han}). If $p\neq 2, 3$, $n>4$, and $n\!\equiv\! 4\pmod p$, Theorem \ref{T;Insectionnumbers} (i) implies that the valencies of all elements in $S$ are not divisible by $p$. If $m\in \mathbb{N}\setminus [1,2]$, notice that there is a prime $q\neq 3$ such that $q\mid 2^m-1$. So there are infinite many choices of the pair $(p, n)$ such that $p\neq 2,3$, $n>4$, and $n\!\equiv\! 4\pmod p$. So infinite many symmetric scheme counterexamples can be obtained by Corollary \ref{C;Corollary6.5}.
\end{rem}
\begin{rem}\label{R;Coreremark2}
Recall the definitions of $E_0^*TE_0^*$, $E_1^*TE_1^*$, $E_2^*TE_2^*$, $E_3^*TE_3^*$ given in Subsection \ref{S;Subsection3}. According to Corollary \ref{C;Corollary2.16} and \eqref{Eq;2}, notice that $E_0^*TE_0^*\cong M_1(\F)$ as $\F$-algebras. Pick $a\in [1,3]$. If $n\equiv 1\pmod p$, then $E_a^*TE_a^*/\mathrm{Rad}E_a^*TE_a^*\cong M_1(\F)$ as $\F$-algebras. Otherwise, $E_a^*TE_a^*\cong M_1(\F)\oplus M_1(\F)$ as $\F$-algebras. By Corollaries \ref{C;Corollary6.3} and \ref{C;Corollary6.5}, notice that $T$ is semisimple if and only if the $\F$-subalgebra $E_b^*TE_b^*$ of $T$ is semisimple for any $b\in [0,4]$. So there is a connection between the semisimplicity of $T$ and the semisimplicity of some $\F$-subalgebras of $T$. A question in  \!\cite[Page 1524]{Han} asks whether this connection holds for the Terwilliger $\F$-algebras of all symmetric schemes. The study for $T$ gives infinite many nontrivial examples to this question.
\end{rem}
%\subsection*{Disclosure statement} There are no relevant financial or nonfinancial interests to report.
%\subsection*{Acknowledgements}
%The author thanks a referee for his or her helpful comments. The present work is supported by the Mathematical Center in Akademgorodok under Agreement No. 075-15-2019-1613 with the Ministry of Science and Higher Education of the Russian Federation.

\end{document}